\documentclass[11pt,a4paper]{amsart}
\usepackage{amsmath, amssymb,amsthm}
\usepackage[hidelinks]{hyperref}
\usepackage[all]{xy}
\usepackage{enumerate,enumitem}
\usepackage{comment}
\usepackage{mathrsfs}
\usepackage{MnSymbol} 
\usepackage{tikz-cd} 
\usepackage{multicol}

\newtheorem{thm}{Theorem}[section]
\newtheorem{prop}[thm]{Proposition}
\newtheorem{lemma}[thm]{Lemma}
\newtheorem{cor}[thm]{Corollary}
\newtheorem{defn}[thm]{Definition}
\newtheorem{conj}[thm]{Conjecture}

\newtheorem{theoremletter}{Theorem}

\theoremstyle{remark}

\newtheorem{rem}[thm]{Remark}

\newcommand{\C}{\mathbb{C}}
\newcommand{\Q}{\mathbb{Q}}

\newcommand{\Z}{\mathbb{Z}}
\newcommand{\W}{\mathbb{W}}

\newcommand{\I}{\mathrm{I}}

\newcommand{\psim}{\psi_{\mbox{-}}}
\newcommand{\chim}{\chi_{\mbox{-}}}
\newcommand{\cLm}{\mathscr{L}_{\mbox{-}}}

\newcommand{\cE}{\mathscr{E}}
\newcommand{\cL}{\mathscr{L}}
\newcommand{\cC}{\mathscr{C}}
\newcommand{\cX}{\mathscr{X}}

\newcommand{\cA}{\mathcal{A}}
\newcommand{\cD}{\mathcal{D}}
\newcommand{\cF}{\mathcal{F}}
\newcommand{\cM}{\mathcal{M}}
\newcommand{\cO}{\mathcal{O}}
\newcommand{\cR}{\mathcal{R}}
\newcommand{\cT}{\mathcal{T}}
\newcommand{\cU}{\mathcal{U}}
\newcommand{\cV}{\mathcal{V}}
\newcommand{\cW}{\mathcal{W}}
\newcommand{\cZ}{\mathcal{Z}}

\newcommand{\gc}{\mathfrak{c}}
\newcommand{\gm}{\mathfrak{m}}
\newcommand{\gp}{\mathfrak{p}}

\DeclareMathOperator{\ad}{ad}

\DeclareMathOperator{\loc}{loc}
\DeclareMathOperator{\End}{End}
\DeclareMathOperator{\Ext}{Ext}
\DeclareMathOperator{\full}{full}
\DeclareMathOperator{\Frob}{Frob}
\DeclareMathOperator{\Fil}{Fil}
\DeclareMathOperator{\Ind}{Ind}
\DeclareMathOperator{\Gal}{Gal}
\DeclareMathOperator{\GL}{GL}
\DeclareMathOperator{\rH}{H}
\DeclareMathOperator{\rZ}{Z}
\DeclareMathOperator{\Hom}{Hom}
\DeclareMathOperator{\ord}{ord}

\DeclareMathOperator{\red}{red}
\DeclareMathOperator{\Spec}{Spec }
\DeclareMathOperator{\spl}{split}
\DeclareMathOperator{\pr}{pr} 

\DeclareMathOperator{\res}{res}

\DeclareMathOperator{\tr}{tr}
\DeclareMathOperator{\tor}{tor}
\DeclareMathOperator{\univ}{univ}

\title[Geometry of the eigencurve at CM points]{{Geometry of the eigencurve at CM points and \\
trivial zeros of  Katz  $\MakeLowercase{p}$-adic $L$-functions}}
\author{Adel Betina and Mladen Dimitrov}

\address{Univ. Vienna, Faculty of Mathematics,  Oskar-Morgenstern-Platz 1, 1090 Wien, Austria}

\email{adelbetina@gmail.com }

\address{Univ. Lille, CNRS, UMR 8524 -- Laboratoire Paul Painlev\'e, 59000 Lille, France.}
\email{mladen.dimitrov@gmail.com }

\begin{document}
\maketitle

\begin{abstract}   The primary goal of this paper is to investigate the geometry of the $p$-adic eigencurve at a
point $f$  corresponding to  a weight one cuspidal theta series $\theta_\psi$  irregular at the prime number $p$. We show that $f$ 
 belongs  to  exactly  three or  four  irreducible components and study their intersection multiplicities. 
In particular, we introduce an anti-cyclotomic $\cL$-invariant $\cLm(\psim)$ and show that the congruence ideal of the CM component $\Theta_\psi$  has a simple zero at $f$ if and only if $\cLm(\psim)$  does not vanish. Further, using Roy's Strong Six Exponential Theorem we show that at least one amongst $\cLm(\psim)$ and $\cLm(\psim^{-1})$ is non-zero. Combined with a divisibility proved by  Hida and Tilouine,  we deduce that the anti-cyclotomic Katz $p$-adic $L$-function of   $\psim$ has a simple (trivial) zero at $s=0$ if $\cLm(\psim)$ is non-zero, which  
can be seen as an  anti-cyclotomic analogue of a  result of Ferrero and Greenberg. 
Finally, we propose a formula  for   the linear term of the two-variable Katz $p$-adic $L$-function of  $\psim$
at $s=0$  extending   a  conjecture of Gross.  
\end{abstract}

\addtocontents{toc}{\setcounter{tocdepth}{0}}

\section*{Introduction}

Our investigations lie at the crossroad between Iwasawa theory and Hida theory. 
Iwasawa theory and its subsequent generalizations study  $p$-adic  $L$-functions of automorphic forms or motives. 
The Iwasawa Main Conjecture postulates that those  generate   the characteristic ideals of  appropriately defined  
Greenberg--Selmer groups. In the few cases where a proof is available, comparing the orders 
of the trivial (or extra) zeros is quite  involved and requires additional arguments.  
On the other hand, Hida has extended the construction from his seminal trilogy \cite{Hidatrilogy1,Hidatrilogy2,Hidatrilogy3}  to congruence modules for $p$-adic families of automorphic forms, and has conjectured that the corresponding characteristic power series is generated by  the adjoint $p$-adic  $L$-function of the family, and that it does not vanish at classical points of weight at least two.   Hida theory has allowed to seek understanding of the Iwasawa Main Conjecture through the study of the geometry of  $p$-adic eigenvarieties, and to detect even finer local and global geometric phenomena such as  intersection numbers  and ramification indices over the weight space.

Let $K$ be an  imaginary quadratic field  in which    $p$ splits as $(p)=\gp\bar\gp$. Let $\psi$ be a   finite order  Hecke character of $K$ which does not descend to $\Q$ and has  conductor relatively prime to $p$. We assume that  
  $\psi(\gp)=\psi(\bar\gp)$, so that  the weight $1$  theta series $\theta_\psi$ attached to $\psi$ has a unique $p$-stabilization, denoted by $f$, 
  and we study the   conjectural relation between  the congruence  power series of a Hida family
containing  $f$ and its adjoint $p$-adic  $L$-function. In the case of the CM family $\Theta_{\psi}$, the latter is  essentially given by the 
Katz anti-cyclotomic  $p$-adic $L$-function $L_p^{-}(\psim,s)$  of $\psim=\psi/\psi^\tau$, where 
 $\tau$ denotes  a complex conjugation, and Katz's $p$-adic  Kronecker Second Limit Formula  implies that   $L_p^{-}(\psim,s)$ has 
 a trivial zero at $s=0$. While the  Iwasawa Main Conjecture relates $L_p^{-}(\psim,s)$ to  the characteristic series of a certain Iwasawa module over the  anti-cyclotomic tower of $K$,   neither the original formulation, nor its proofs by Rubin \cite{rubin} and by Tilouine \cite{tilouine}, provide a formula for the orders of their common trivial zeros.  
 
A main theme of this paper is to describe as precisely as possible the 
local geometry of the eigencurve $\cE$ at $f$ and draw some arithmetic consequences, such as a 
 criterion for $L_p^{-}(\psim,s)$  to have a simple trivial zero. 
Our investigations pushed us to go further and propose the following expression for the leading term
of  the Katz  $2$-variable $p$-adic $L$-function 
\begin{align}\label{linear-term-intro}
L_p(\psim, s_{\gp}, s_{\bar\gp})\overset{?}{=} \big(\log_p(\bar\gp)\cdot s_{\gp}+(\cLm(\psim)+\log_p(\bar\gp)) \cdot s_{\bar\gp}  \big)
 \cdot  L_p^*(\psim,0)+\text {higher order terms},
\end{align}
where  $\cLm(\psim)$ is  the anti-cyclotomic $\cL$-invariant introduced in  Definition~\ref{d:new-L-inv}, while  $L_p^*(\psim,0)$ is related via 
the $p$-adic  Kronecker Second Limit Formula  \eqref{kronecker}
to a  $p$-adic logarithm of an elliptic unit, which is non-zero by the Baker--Brumer Theorem (see \S\ref{Katz-padic} for more details).

 The above formula  lies outside the scope of Gross'  famous  conjecture \cite{gross} on the leading term at $s=0$ of $p$-adic $L$-functions of Artin  Hecke characters. Indeed,  the fixed field of   $\psim$ is not a CM field, unless $\psim$ is  quadratic, 
in which case  we  prove  \eqref{linear-term-intro} up to a scalar in $\Q^\times$ using Gross' factorization formula.    
In the general case, we provide the following evidence. 
\begin{theoremletter} \label{theo-triv-zero}
If $\cLm(\psim) \ne 0$, then $L_p^{-}(\psim,s)=L_p(\psim,s,-s)$ has a simple  (trivial) zero at $s=0$. 
 Moreover, at least one amongst $\cLm(\psim)$  and $\cLm(\psim^\tau)$ is non-zero.
\end{theoremletter}

Theorem \ref{theo-triv-zero} gives an affirmative answer, for more than half of the anti-cyclotomic characters, to a long standing question raised by Hida--Tilouine in \cite[p.106]{HT2}, as $\ord_{s=0}L_p^{-}(\varphi, s)=1$ for at least one of the  characters $\varphi$ occurring in  each set  $\{\psim,\psim^\tau\}$.  

Let us now briefly describe our approach. 
The second part is proved in  Proposition~\ref{6expocusp}
using the Strong  Six Exponentials Theorem, while the Strong Four Exponentials Conjecture would imply that  
$\cLm(\psim)\cdot \cLm(\psim^\tau)\ne 0$.  
An important  observation made in  Theorem~\ref{existencenon-cm}  is that in addition to the two Hida families  $\Theta_{\psi}$ and $\Theta_{\psi^{\tau}}$ having CM by $K$,  $f$ also belongs to at least one  Hida family without CM by $K$, hence lies in  the closed  analytic subspace $\cE^{\perp}$ of $\cE$, union of irreducible components having no  CM by $K$.
 Denote by  $\cT$, resp. by $\cT^{\perp}$, the completed local ring at $f$ of $\cE$, resp. of  $\cE^{\perp}$.  Both algebras $\cT$ and 
 $\cT^{\perp}$   are  finite and flat over  the completed local ring
  $\varLambda = \overline{\Q}_p\lsem X\rsem $ of the weight space 
  at the point corresponding to the weight of $f$.   Consider   the congruence ideal 
$C^0_\psi=\pi_\psi (\mathrm{Ann}_{\cT}(\ker(\pi_\psi)))$ attached to the $\Theta_{\psi}$-projection $\pi_\psi:\cT \twoheadrightarrow \varLambda$.   Letting  $\zeta^{-}_{\psim}$ denote the $p$-adic inverse Mellin transform of $L_p^{-}(\psim,s)$, seen as an element of $\varLambda$, Hida and Tilouine showed in \cite[Thm.~I]{HT}  that
\[ C^0_\psi\subset (\zeta^{-}_{\psim})\subset (X),\]
  thus providing an  upper bound for $\ord_{s=0}L_p^{-}(\psim, s)$ in terms of the local geometry of  $\cE$ at $f$.

To prove the first part of Theorem~\ref{theo-triv-zero} it then suffices  to show  that $C^0_\psi=(X)$  if and only if  $\cLm(\psim)\neq0$, which is a consequence of the main result of this paper.  

\begin{theoremletter}[Thm.~\ref{C0-thm}]\label{main.geo}
The  $\varLambda$-algebra   $ \cT$ is isomorphic to $\varLambda \times_{\overline{\Q}_p} ({\cT}^\perp \times_{\overline{\Q}_p[X]/(X^{r-1})} \varLambda)$, where 
\begin{itemize}
\item  if $\cLm(\psim^\tau) \cdot  \cLm(\psim)=0$, then $r \geq 3$ and  $\cT^{\perp}= \varLambda[Z]/(Z^2-X^{r})$, 
\item  if $\cLm(\psim^\tau) +  \cLm(\psim)=0$,   then $r=2$ and $ \cT^{\perp}=\overline{\Q}_p\lsem X^{1/e}\rsem \times_{\overline{\Q}_p}\overline{\Q}_p\lsem X^{1/e}\rsem$ for some  $e \geq 2$, 
\item in all other cases we have   $r=2$ and $ \cT^{\perp}=\varLambda \times_{\overline{\Q}_p} \varLambda$. 
\end{itemize}
\end{theoremletter}

If $\cLm(\psim)\neq 0$, then  the first projection in the fiber product description of  $\cT$ above  is given by 
 $\pi_\psi$, hence  both $\zeta_{\psim}^{-}$ and  the characteristic power series of the congruence module attached to $\Theta_\psi$  have a simple zero at $X=0$, as conjectured by Hida--Tilouine \cite[p.192]{HT}. The conjecture was proven by Tilouine \cite{tilouine} and Mazur--Tilouine \cite{mazur-tilouine},  under some technical hypotheses, when  $\psim(\gp) \ne 1$. We would like to emphasize that cases with a trivial zero like ours lie beyond the reach of the methods  used in {\it loc. cit.} as the ordinary deformation functor is not representable and, more importantly,  $C^0_\psi$ and  $C^1_\psi=\ker(\pi_\psi)/\ker(\pi_\psi)^2$ are not equal, as the 
  Hecke algebra $\cT$ is not a local complete intersection.
Theorem  \ref{main.geo} also allows us to  establish a conjecture of Darmon, Lauder and Rotger  \cite{DLR4} on the dimension of the generalized eigenspace of $p$-adic modular forms corresponding to 
a weight one cusform with CM, which was previously only known for  $\psim$  quadratic by the work \cite{lee} of Lee (see Corollary~\ref{DLRconj}).

Let us now explain some of the ideas and techniques  that go in the proof of Theorem~\ref{main.geo}. 
As the $2$-dimensional $\Gamma_\Q$-representation $\rho_f=\Ind^\Q_K \psi$  is locally scalar  at $p$, neither its Mazur's $p$-ordinary deformation functor is representable, 
nor its  deformation to $\cT$ has a  $p$-ordinary filtration. 
  In contrast with the $p$-regular setting considered in \cite{bellaiche-dimitrov},
where $\cT$ is isomorphic to the $p$-ordinary deformation ring of $\rho_f$,  we need to resort to new methods.
First, we divide the difficulties and study  separately  the components having no CM by $K$. 
 As the restriction to $\Gamma_K$ of the $\cT^{\perp}$-valued pseudo-character is generically irreducible and lifts $\psi+\psi^{\tau}$, 
 and as $\Ext_{\Gamma_K}^1(\psi^{\tau},\psi)$ and $\Ext_{\Gamma_K}^1(\psi,\psi^{\tau})$ are both $1$-dimensional  generated  by $\rho$ and $\rho'$, respectively,  a Ribet style argument shows the existence of pair $(\rho_{\cT^{\perp}},\rho_{\cT^{\perp}}')$ of  $\gp$-ordinary $\Gamma_K$-deformations  lifting $(\rho,\rho')$. This leads naturally to a universal deformation $\varLambda$-algebra $\cR^{\perp}$ classifying pairs of $\gp$-ordinary deformations of $(\rho,\rho')$ sharing the  same $\tau$-invariant traces and rank one 
 $\gp$-unramified  quotients. We show that $\cR^{\perp}\simeq\overline{\Q}_p\lsem X^{1/e}\rsem$ is a discrete valuation ring 
 endowed with flat morphism $\cR^{\perp} \hookrightarrow \cT^{\perp}$ of $\varLambda$-algebras
 (see Thm.~\ref{isom-cusp}). Furthermore, the reducibility ideal of the $2$-dimensional pseudo-character carried by $\cR^{\perp}$ is generated by an element of valuation $r \geqslant 2$, with $r=2$ if and only if $\cLm(\psim^\tau) \cdot  \cLm(\psim) \ne 0$. 
In \S\ref{wihtoutCMgenericpoints} we determine all possible extensions to $\Gamma_\Q$ of the pseudo-character carried by $\cR^{\perp}$ 
and deduce  that $\cR^{\perp}[Z]/(Z^{2}-X^{r/e})\simeq \cT^{\perp}$ as  $\varLambda$-algebras. 
With all the components of $\cE$ containing $f$ determined in \S\ref{section3},  \S\ref{CM-Ideal} is devoted to the
study of the CM ideal, measuring the congruences between a component having CM by $K$ and a component without CM by $K$. 
Finally in \S\ref{s:cong-ideal} we determine the  ideal $C^0_\psi$ governing the congruences between  $\Theta_{\psi}$
and all other families. 

While the deformation to $\cT$ is not  $p$-ordinary, its localization at any generic point is,  motivating  the use  in \S\ref{trianglesec}  of a universal ``generically'' $p$-ordinary deformation ring $\cR^{\triangle}$ of $\rho_f$. As $\cT$ is not Gorenstein (see  Prop.~\ref{commalg})
the following  result lies beyond the scope of the  Taylor--Wiles method and its proof resorts to a novel approach to modularity. 

\begin{theoremletter} \label{theorem-C} There exists an isomorphism of non-Gorenstein local $\varLambda$-algebras $\cR^{\triangle}_{\red} \simeq \cT$,  where $\cR^{\triangle}_{\red}$ is the nilreduction of $\cR^{\triangle}$. 
\end{theoremletter}

In closing, let us specify the above results in the simplest case of a quadratic character $\psim$  defining a biquadratic extension $H$ of 
$\Q$ in which $p$ splits completely. Denote by $K'$ (resp. $F$) the other imaginary quadratic field (resp. the unique  real quadratic field) contained in $H$. 
In that case we have 
$\cT \simeq  \varLambda \times_{\overline{\Q}_p} \varLambda \times_{\overline{\Q}_p} \varLambda \times_{\overline{\Q}_p} \varLambda$
and, in addition to the families $\Theta_{\psi}$ and  $\Theta_{\psi^{\tau}}=\Theta_{\psi}\otimes \epsilon_F  $ having CM by $K$, $f$ is contained in   families $\Theta_{\psi'}$ and  $\Theta_{\psi'{}^\tau}=\Theta_{\psi'}\otimes \epsilon_F $ having CM by $K'$.  
Both $p$-adic $L$-functions $L_p^{-}(\psim,s)$ and $L_p^{-}(\psim',s)$ then have simple  zeros at $s=0$. 

\vspace{-6mm}
\footnotesize
\tableofcontents
\normalsize
\newpage 

\addtocontents{toc}{\setcounter{tocdepth}{2}}

\section{Slopes and  \texorpdfstring{$\cL$}{L}-invariants}\label{L-invariants}

 We denote by $\Gamma_L=\Gal(\bar{L}/L)$ the absolute Galois group of a perfect field $L$. 
We consider the field of algebraic numbers $\overline{\Q}$ as a subfield of $\C$ which  determines a complex conjugation $\tau \in \Gamma_{\Q}$. We let $\epsilon_{L}$ denote the Dirichlet character of a quadratic extension $L/\Q$.

 Denote by $H$ the splitting field of the anti-cyclotomic character $\psim=\psi/\psi^\tau\ne\mathbf{1}$, 
  where for any  function  $h$ on $\Gamma_K$ we let  $h^{\tau}(\cdot)= h(\tau \cdot \tau)$.
  Recall that  $\psi$  has  conductor relatively prime to $p$ and  $\psim(\gp)=1$, {\it i.e.}, $p$ splits completely in $H$.
We  fix  an embedding $\iota_p: \overline{\Q} \hookrightarrow \overline{\Q}_p$ which determines an embedding $\Gamma_{\Q_p}\hookrightarrow \Gamma_K$ and   a  place  $v_0$ of $H$ above the place  $\gp$ of $K$. We let  $\I_{\gp}$ denote the 
inertia subgroup  of $\Gamma_{K_\gp}\simeq \Gamma_{\Q_p}$.

\subsection{The slope of  \texorpdfstring{$\psim$}{}}\label{subs-slope}
The dihedral group $G=\Gamma_{H/\Q}$ is a semi-direct product of the cyclic group $C=\Gamma_{H/K}$ with $\{1,\tau\}$. 
The product is direct if and only if the character $\psim$ is quadratic, in which case $G$ is the Klein group.

Given any global unit  $u\in \cO_H^\times$ we consider the element 
\[u_{\psim}=\sum_{g\in C}  g^{-1}(u)\otimes \psim(g)\in (\cO_H^\times \otimes \overline \Q)[\psim],   \]
where  $[\psim]$ denotes the $\psim$-isotypic part for the natural left action of $C$. 

Minkowski's proof of Dirichlet's Unit  Theorem shows that the left $G$-module $\cO_H^{\times}\otimes \Q$ is isomorphic
to $(\mathrm{Ind}_{\{1,\tau\}}^{G} \mathbf{1})/(\Q\cdot \mathbf{1})$  (see  \cite[(7)]{bellaiche-dimitrov}), hence the  $2$-dimensional odd representation $\Ind^\Q_K \psim$ occurs with multiplicity one in it. 
By the Frobenius Reciprocity Theorem  $(\cO_H^\times \otimes \overline \Q)[\psim]$ is a line, 
and similarly $(\cO_H^\times \otimes \overline \Q)[\psim^\tau]$ is a line containing $\tau(u_{\psim})=\sum\limits_{g\in C}  \tau g^{-1}(u)\otimes \psim(g)$. Minkowski's Theorem guarantees the existence of $u\in \cO_H^\times$ such that 
$\{g(u)| g\in C\backslash\{1\}\}$ is a basis of $\cO_H^\times \otimes \Q$, called a Minkowski unit,  to which one can additionally impose to be fixed by $\tau$. 
As $\cO_K^\times$ is finite, the only non-trivial relation between 
$\{g(u)| g\in C\}$ is given by the kernel of the relative norm $\mathrm{N}_{H/K}$, hence for any Minkowski unit $u$ one has $u_{\psim}\ne 0$. 
 We record those  useful  facts as a lemma.  
\begin{lemma}\label{l:useful}
There exists $u\in \cO_H^\times$, such that  
\[(\cO_H^{\times}\otimes \overline{\Q})[\psim]=  \overline{\Q} \cdot u_{\psim}, \text{ and  } (\cO_H^{\times}\otimes \overline{\Q})[\psim^\tau]= \overline{\Q} \cdot \tau(u_{\psim}). \] 
\end{lemma}
\noindent The Baker--Brumer Theorem \cite{brumer} demonstrates the injectivity of the  $\overline{\Q}$-linear  homomorphism \begin{align}\label{eq:baker-brumer} 
\log_p : \cO_H^\times\otimes \overline{\Q} \longrightarrow \overline{\Q}_p\text{ , } u \otimes x \mapsto \log_p(u)\cdot \iota_p(x), 
\end{align} 
allowing  to define, for $u$ as in Lemma~\ref{l:useful},  the following $p$-adic regulator, that we call the slope
\begin{align}\label{eq:slope-def}
 \mathscr{S}(\psim)=-\frac{\log_p(u_{\psim})}{\log_p(\tau(u_{\psim}))}.
\end{align} 

\begin{lemma} \label{l:slope}
The slope $\mathscr{S}(\psim)$ does not depend on the choice of $u$ as in Lemma~\ref{l:useful}. 
Moreover $\mathscr{S}(\psim)\notin \overline{\Q}$, unless $\psim$ is quadratic, in which case $\mathscr{S}(\psim)=-1$. 
Finally $\mathscr{S}(\psim^\tau)=\mathscr{S}(\psim)^{-1}$. 
\end{lemma} \label{lem:ell-inv}
\begin{proof} The value of $\mathscr{S}(\psim)$ remains clearly unchanged if we replace $u$ by $g(u)$ for any $g\in C$. 
The independence then follows from the fact that any unit lies in the $\Q$-linear span of these. If 
$\psim^\tau\ne \psim$, then $\mathscr{S}(\psim)\notin \overline{\Q}$ by  \eqref{eq:baker-brumer}. Finally, taking $u$ to be a 
Minkowski unit fixed by $\tau$, one finds that $\tau(u_{\psim})=u_{\psim^\tau}$ which, in view of the definition \eqref{eq:slope-def}, 
simultaneously shows that $\mathscr{S}(\psim^\tau)=\mathscr{S}(\psim)^{-1}$ and further, when   $\psim$ is quadratic, that
$\mathscr{S}(\psim)=-1$.\end{proof}

\subsection{A Galois theoretic interpretation of the slope}\label{def-eta} 

Let   $\eta_{\gp}$  be the unique element of $\ker(\rH^1\left(K, \overline{\Q}_p)\to \rH^1(\I_{\bar{\gp}},\overline{\Q}_p)\right)$ whose image in $\mathrm{im}\left(\rH^1(K_{\gp}, \overline{\Q}_p) \to 
 \rH^1(\I_{\gp}, \overline{\Q}_p)\right)\simeq \Hom (\cO_{K_{\gp}}^\times, \overline{\Q}_p)$
 is given by  $\log_p$. Note that $\eta_{\gp}$ and $\eta_{\bar{\gp}}=\eta_{\gp}^\tau$ form a basis of $\rH^1(K, \overline{\Q}_p)$.

\begin{prop}\label{restinj} The image of the natural  restriction homomorphism
\[\rH^1(K, \psim) \hookrightarrow \mathrm{im}\left(\rH^1(K_{\gp}, \overline{\Q}_p)\times\rH^1(K_{\bar{\gp}}, \overline{\Q}_p) \to 
 \rH^1(\I_{\gp}, \overline{\Q}_p) \times  \rH^1(\I_{\bar{\gp}}, \overline{\Q}_p)\right)\simeq \Hom (\cO_{K_{\gp}}^\times \times\cO_{K_{\bar{\gp}}}^\times, \overline{\Q}_p) \]
 is a line generated by $(\log_p, \mathscr{S}(\psim) \log_p)$. In particular both maps $\rH^1(K, \psim)\to \rH^1(\I_{\gp}, \overline{\Q}_p)$ and
 $\rH^1(K, \psim)\to \rH^1(\I_{\bar{\gp}}, \overline{\Q}_p)$ are injective. 
 \end{prop}

\begin{proof}
The  inflation-restriction exact sequence (see \cite[(8)]{bellaiche-dimitrov}) yields  an isomorphism 
 \begin{align}\label{res-infl}
 \rH^1(K, \psim) \xrightarrow{\sim}  \rH^1(H,\overline{\Q}_p)[\psim^\tau]
= ( \Hom(\Gamma_H, \overline{\Q}_p)\otimes  \psim)^{C}. 
\end{align} 
Using Shapiro's Lemma the vector spaces in \eqref{res-infl} are further isomorphic to \begin{align}\label{dimext}
\rH^1(\Q, \Ind^\Q_K \psim) \xrightarrow{\sim} \rH^1(K,  \psim\oplus \psim^\tau)^{\tau=1}\xrightarrow{\sim} 
( \Hom(\Gamma_H, \overline{\Q}_p)\otimes  \Ind^\Q_K \psim)^{G}.
\end{align} 
As  Leopoldt's Conjecture holds for cyclic extensions of imaginary quadratic fields, 
the dimension of the latter equals   $\dim (\Ind^\Q_K \psim)^{\tau=-1}=1$.

Taking  the $\psim^\tau$-isotypic part of \cite[(6)]{bellaiche-dimitrov} for  the left $C$-action yields  an exact sequence 
\[0 \to \rH^1(H,\overline{\Q}_p)[\psim^\tau]\xrightarrow{\res_{p}}  \Hom (\cO_{H_{\gp}}^\times , \overline{\Q}_p)[\psim^\tau]\times 
\Hom (\cO_{H_{\bar{\gp}}}^\times, \overline{\Q}_p)[\psim^\tau]\to  \Hom (\cO_{H}^\times , \overline{\Q}_p).\]
As  $\Hom(\cO_{H_{\gp}}^\times, \overline{\Q}_p)$, resp.    $\Hom(\cO_{H_{\bar{\gp}}}^\times, \overline{\Q}_p)$, is 
the regular representation of $C$,  its $(\psim^\tau)$-component is $1$-dimensional, generated by  
$\sum_{g \in C} \psim(g) \log_p( g^{-1}\cdot )$, resp. $\sum_{g \in C} \psim(g) \log_p( \tau  g^{-1}\cdot )$. 
It follows that   $\rH^1(K, \psim) \xrightarrow{\sim}  \rH^1(H,\overline{\Q}_p)[\psim^\tau]$  has a basis whose 
 restriction  to $(\cO_{H}\otimes \Z_p)^\times$ is given by 
\[\sum_{g \in C} \psim(g) \log_p( g^{-1}\cdot ) + s \cdot \sum_{g \in C} \psim(g) \log_p( \tau  g^{-1}\cdot ),\] for some  $s \in \overline{\Q}_p$. The triviality on  global units (in particular on $u_{\psim}$) implies,   in view of  Lemma~\ref{l:slope},   that $s= \mathscr{S}(\psim)$.\end{proof}

\subsection{\texorpdfstring{$\cL$}{L}-invariants}\label{L-invariantdefn}
As $\gp$ splits completely in $H$, the set of places of $H$ above $\gp$ consists of  $g(v_0)$ for $g\in C$.
Choose $y_0\in \cO_{H}[\tfrac{1}{v_0}]^\times $ having a non-zero  valuation $\ord_{v_0}(y_0)$ at $v_0$. Then
\begin{align}\label{defn-L-inv}
\cL(\psim^\tau)=-\frac{\sum_{g\in C}  \psim(g) \log_p (g^{-1}(y_0))+\mathscr{S}(\psim) \sum_{g\in C}  \psim(g) \log_p (\tau g^{-1}(y_0))}{\ord_{v_0}(y_0)}
\end{align}
is independent of the choice of $y_0$, as by  Lemma~\ref{l:slope} multiplying $y_0$ by an element of $\cO_{H}^{\times}$ does not affect its value.  Another description of $\cL(\psim^\tau)$  can be given using the exact sequence 
\[1 \to \cO_{H}^\times  \to \cO_{H}[\tfrac{1}{\gp}]^\times \xrightarrow{\ord} \Z[C]\]
of $C$-modules, where the last map is obtained by taking valuations at all places $v$ of $H$ dividing $\gp$. Taking the  $\psim$-isotypic part for the left action of $C$ yields a short exact sequence 
\[0 \to \overline{\Q}\cdot  u_{\psim}  \to (\cO_{H}[\tfrac{1}{\gp}]^\times \otimes \overline{\Q})[\psim] \xrightarrow{\ord_{v_0}}  \overline{\Q} \to 0.\]
Letting $u_{\gp,\psim} \in (\cO_{H}[\tfrac{1}{\gp}]^\times \otimes \overline{\Q})[\psim] $ be any element such that
$\ord_{v_0}(u_{\gp,\psim})\ne 0$, for example $u_{\gp,\psim}= \sum_{g \in C} g^{-1}(y_0)\otimes\psim(g) $, one easily checks that 
\[\cL(\psim^\tau)=-\frac{\log_p(u_{\gp,\psim})+\mathscr{S}(\psim)\log_p(\tau(u_{\gp,\psim}))}{\ord_{v_0}(u_{\gp,\psim})}.\]

Finally taking $u_{\gp}$ to be any non-torsion element of $\cO_K[\tfrac{1}{\gp}]^\times$,  for example $\mathrm{N}_{H/K}(y_0)=\sum_{g \in C} g^{-1}(y_0)$, one defines 
\begin{align}\label{defn-L1}
\cL_{\gp}=- \frac{\log_p(u_{\gp})}{\ord_{\gp}(u_{\gp})}. 
\end{align}

\subsection{A Galois theoretic interpretation of the \texorpdfstring{$\cL$}{L}-invariants}
Denote by  $\eta$ a representative of the canonical generator $[\eta]$ of $ \rH^1(K, \psim)$ described in Proposition~\ref{restinj}. 

\begin{lemma}\label{linvariant1}
One has  $\eta_{\bar\gp}(\Frob_{\gp})=\eta_{\gp}(\Frob_{\bar\gp})= -\cL_{\gp}$.
\end{lemma}
\begin{proof} As $u_{\gp}\bar u_{\gp}\in \Z_p^\times\cdot p^\Z$, 
 using the exact sequence from Global Class Field Theory 
\[0 \to \rH^1(K,\overline{\Q}_p) \to \Hom(K_{\gp}^\times \times  \cO_{K_{\bar\gp}}^\times, \overline{\Q}_p)
\to \Hom( \cO_K[\tfrac{1}{\gp}]^\times,  \overline{\Q}_p)\]
one finds $\ord_{\gp}(u_{\gp}) \cdot \eta_{\bar\gp}(\Frob_{\gp})=\res_{\gp}(\eta_{\bar\gp})(u_{\gp})=- \res_{\bar\gp}(\eta_{\bar\gp})(u_{\gp})=-\log_p(\bar u_{\gp})=\log_p(u_{\gp})$.\end{proof}

As $[\eta] \in \rH^1(K,\psim) $ and $\eta_{\gp}\in \rH^1(K,\overline{\Q}_p)$ both restrict to the same element of 
$\rH^1(\I_{\gp},\overline{\Q}_p)$, their difference  (considered as element of $\rH^1(K_{\gp},\overline{\Q}_p)$) is unramified at $\gp$, 
and the following proposition computes its value on   an arithmetic Frobenius element. 

\begin{prop}\label{linvariant2} One  has  $(\eta-\eta_{\gp})(\Frob_{\gp})= \cL(\psim^\tau) - \cL_{\gp}$. 
\end{prop}

\begin{proof} Restricting to $\Gamma_H$ allows us to see $(\eta-\eta_{\gp})$ in 
$\rH^1(H,\overline{\Q}_p)$ and one has  to compute its value at $\Frob_{v_0}$. 
  By the exact sequence from Global Class Field Theory 
\[0 \to \rH^1(H,\overline{\Q}_p) \to \Hom(H_{v_0}^\times \times \prod\limits_{v_0\ne v\mid p} \cO_{H_v}^\times, \overline{\Q}_p)
\to \Hom( \cO_H[\tfrac{1}{v_0}]^\times,  \overline{\Q}_p)\]
and by Proposition~\ref{restinj} and its proof, the restriction of $\eta$ to $(\cO_{H}\otimes \Z_p)^\times$ is given by 
\[\sum_{g \in C} \psim(g) \log_p( g^{-1}\cdot ) + \mathscr{S}(\psim) \sum_{g \in C} \psim(g) \log_p( \tau  g^{-1}\cdot ).\]
Therefore, for   $y_0 \in \cO_H[\tfrac{1}{v_0}]^\times$  such that $\ord_{v_0}(y_0)\ne 0$, one has 
\[\ord_{v_0}(y_0) \res_{v_0}(\eta-\eta_{\gp})(\Frob_{\gp})=\sum_{1\ne g \in C} (1- \psim(g)) \log_p( g^{-1}(y_0))- \mathscr{S}(\psim) \sum_{g \in C} \psim(g) \log_p( \tau  g^{-1}(y_0)).\]
The desired formula then follows from the definition of $\cL(\psim^\tau)$ in \eqref{defn-L-inv}, using the computation 
\[\sum_{g \in C}  \log_p( g^{-1}(y_0))= \log_p( \mathrm{N}_{H/K}(y_0))= -\cL_{\gp} \ord_{v_0}(y_0), \]
the last equality  following from the fact that $p$ splits completely in $H$.\end{proof}

\subsection{An interlude on the Six Exponentials Theorem}

We first recall some standard results and conjectures in  $p$-adic transcendence. 
As in the previous section  $\log_p: \overline{\Q}{}_p^{\times}\to \overline{\Q}_p$ is the standard $p$-adic logarithm sending $p$ to $0$ and we use the same notation for its pre-composition with $\iota_p:\overline{\Q}^\times \hookrightarrow \overline{\Q}_p^\times$. 

Denote by $\mathcal{L}\subset \overline{\Q}_p$ the $\overline{\Q}$-vector space generated by $1$ and  elements of 
$\log_p(\overline{\Q}{}^\times)$. 

\begin{thm}[Baker--Brumer \cite{brumer}]\label{thm:BB} 
If $\lambda_1, ....., \lambda_n \in \mathcal{L}$ are linearly independent over $\Q$, then they are linearly independent over $\overline{\Q}$.
\end{thm}

\begin{conj}[Strong Four Exponentials Conjecture]\label{conj:4exp} 
A matrix $\left(\begin{smallmatrix} \lambda_{11} & \lambda_{12} \\ \lambda_{21} & \lambda_{22} \end{smallmatrix} \right)$ 
with entries in $\mathcal{L}$  whose  lines and  whose columns are linearly independent over $\Q$ has rank $2$. 
\end{conj}

We will need the following Strong version of the Six Exponentials Theorem, proved by Roy \cite{Roy}
(see also  \cite{Wal}). 

\begin{thm}[Strong Six Exponentials Theorem]\label{thm:6exp} 
A matrix $L=\left(\begin{smallmatrix} \lambda_{11} & \lambda_{12} & \lambda_{13}  \\ \lambda_{21} & \lambda_{22} & \lambda_{23} \end{smallmatrix} \right)$ 
with entries in $\mathcal{L}$  whose  lines and whose columns are linearly independent over $\Q$ has rank $2$. 
\end{thm}

\subsection{Non-vanishing of \texorpdfstring{$\cL$}{L}-invariants}

The following $\cL$-invariant will play a prominent role in the paper, as will be intimately related to the local geometry of the eigencurve at $f$. 

\begin{defn}\label{d:new-L-inv}
The anti-cyclotomic $\cL$-invariant of $\psim$ is  $\cLm(\psim) = \cL(\psim) - 2\cL_{\gp}$.
\end{defn}

\begin{rem} One can give an intrinsic definition of $\cLm(\psim)$. 
 Consider the `anti-cyclotomic' basis $\{\ord_{\gp}, \eta_{\gp}- \eta_{\bar\gp} \}$ of $\Hom(\Gamma_{K_\gp},\overline\Q_p)$.
 According to Lemma~\ref{linvariant1}  and Proposition~\ref{linvariant2} one has 
\[\res_{\gp}(\eta)=(\cL(\psim^\tau) - \cL_{\gp})\cdot \ord_{\gp} + \eta_{\gp}= \cLm(\psim^\tau) \cdot \ord_{\gp} +  (\eta_{\gp}-\eta_{\bar\gp}),\]
 hence, analogously to \cite[\S1.3]{DDP},  the anti-cyclotomic $\cL$-invariant  $\cLm(\psim^\tau)$ is (the opposite of) the slope of the line generated by $\eta$ is this local basis. 
 \end{rem}

\begin{prop}\label{6expocusp}
 At least one amongst $\cLm(\psim)$  and $\cLm(\psim^\tau)$ is non-zero.
 If either $\psim$ is quadratic, or the Four Exponentials Conjecture~\ref{conj:4exp} holds, then both  are non-zero. 
\end{prop}

\begin{proof} Recall that $C=\Gal(H/K)$ and that   $y_0 \in \cO_H[\tfrac{1}{v_0}]^\times$ is  such that $\ord_{v_0}(y_0)\ne 0$.  
One has 
\begin{align*}
 &\ord_{v_0}(y_0)\cLm(\psim^\tau)=  \res_{v_0}(\eta-\eta_{\gp}+\eta_{\bar{\gp}})(y_0)=\\
 &= \!\!\sum_{g \in C-\{1\}} (1-\psim(g)) \log_p( g^{-1}(y_0))- \mathscr{S}(\psim) \sum_{g \in C} \psim(g) \log_p( \tau  g^{-1}(y_0))  - \sum_{g \in C} \log_p( \tau  g^{-1}(y_0))=\\ &=- \log_p(u_{\gp,\psim})- \mathscr{S}(\psim) \log_p( \tau (u_{\gp,\psim}))  + 2 \log_p( \mathrm{N}_{H/K}(y_0)).  \end{align*}
 
Suppose first that $\psim$ is quadratic and denote by $g$ the non-trivial element of $C$ so that 
$G=\Gal(H/\Q)=\{1,g,\tau,g\tau\}$. To see that  
$\ord_{v_0}(y_0) \cLm(\psim^\tau)= 2 \log_p\left(\tfrac{g(y_0)}{g\tau (y_0)}\right)$ does not vanish, it suffices to show  that
$\iota_p\left(\tfrac{g(y_0)}{g\tau (y_0)}\right)\in \cO_{H,v_0}^\times$ is a non-torsion element. This is true as 
$\tfrac{g(y_0)}{g\tau (y_0)}\in H^\times$ has a non-zero valuation at $g(v_0)$.

  Consider next the case $\psim\ne \psim^\tau$. A computation similar to the one above shows that   
\[\ord_{v_0}(y_0)\cLm(\psim)= - \log_p(u_{\gp,\psim^\tau})- \mathscr{S}(\psim^\tau) \log_p( \tau (u_{\gp,\psim^\tau}))  +2 \log_p(  \mathrm{N}_{H/K}(y_0)).\]

It suffices to show that  Roy's Six Exponentials Theorem~\ref{thm:6exp} applies to the matrix 
\[\begin{pmatrix}  \log_p(u_{\psim})  &\log_p(u_{\gp,\psim}) -2 \log_p( \mathrm{N}_{H/K}(y_0))  &\log_p(u_{\bar\gp,\psim})   \\ 
 \log_p(u_{\psim^\tau}) & \log_p( u_{\bar\gp,\psim^\tau}) & \log_p(u_{\gp,\psim^\tau})-2 \log_p( \mathrm{N}_{H/K}(y_0))   
 \end{pmatrix} \]
where $u_{\psim^\tau}= \tau(u_{\psim})$, $ u_{\bar\gp,\psim^\tau} =\tau (u_{\gp,\psim})$ and 
$ u_{\bar\gp,\psim}  =\tau (u_{\gp,\psim^\tau})$.  By looking at the first column entries, Lemma~\ref{l:slope} implies
that the rows are linearly independent over $\Q$. It remains to show that the columns
are linearly independent over $\Q$ as well, for which we will only look at the first row entries. 
Using the facts that $\log_p(\mathrm{N}_{H/K}(u))=0$ and $\log_p(\mathrm{N}_{H/\Q}(y_0))=0$, a linear relation over $\Q$ between those $3$ entries would give integers $a$, $b$ and $c$, not all zero, such that
\[ \sum_{1\ne g \in C} (\psim(g)-1)\left(a\cdot \log_p(g^{-1}(u))+ b \cdot\log_p(g^{-1}(y_0)\right)+
\sum_{g \in C} (c\cdot\psim(g)+b) \log_p(g^{-1}\tau(y_0))=0. \]
The numbers $\{\tau(y_0)\}\cup \{ g(u), g(y_0),g \tau(y_0) |  1\ne g \in C\}$ are (multiplicatively) linearly independent over $\Q$ and are all sent by $\iota_p$ to $\cO_{H,v_0}^\times\simeq\Z_p^\times$, hence their $p$-adic logarithms are linearly independent over $\Q$. This contradicts Baker--Brumer's Theorem~\ref{thm:BB}.  

The last claim follows by applying the  Four Exponentials Conjecture~\ref{conj:4exp} to the minors obtained by removing the second (or the third) column of the above matrix.\end{proof}

\section{Galois deformations}\label{deformation}

Recall from \S\ref{def-eta} the unique cohomology class $[\eta]\in \rH^1(K, \psim)$ whose restriction to $\I_{\gp}$ is given by the $p$-adic logarithm. As $\psim-\mathbf{1}$ is a basis of the coboundaries, 
fixing, once and for all, an element $\gamma_0\in \Gamma_K$  such that $\psim(\gamma_0) \ne 1$, we let $\eta: \Gamma_K \to \overline{\Q}_p$ be the  unique representative of $[\eta]$   such that  $\eta(\gamma_0)=0$. 

We let $\rho=\left( \begin{smallmatrix} \psi& \eta\psi^\tau\\ 0 & \psi^\tau \end{smallmatrix}\right):\Gamma_K \to \GL_2(\overline{\Q}_p)$ in the canonical basis  $(e_1,e_2)$ of  $\overline{\Q}{}_p^2$.  The unique $\Gamma_{K_{\gp}}$-stable  
filtration of  $V_{\rho}=\overline{\Q}_p^2$ with unramified quotient is  given by 
 \begin{align}\label{filtration-rho}
0 \to F_{\rho}=\overline{\Q}{}_p e_1\to V_\rho \to V_\rho/F_{\rho} \to 0.
\end{align}

\subsection{Ordinary deformations of \texorpdfstring{$\rho$}{}}\label{ord-def}

Let $\cC$ be the category of complete Noetherian local $\overline{\Q}_p$-algebras $A$ with 
maximal ideal $\gm_A$ and residue field $\overline{\Q}{}_ p$, where the morphisms are local homomorphisms of 
$\overline{\Q}_p$-algebras.

Consider the functor $\cD^{\univ}_{\rho}$ associating to $A$ in $\cC$ the set of lifts  $\rho_A:\Gamma_K \to \GL_2(A)$ of $\rho$  
modulo strict equivalence and satisfying 
\begin{align}\label{tau-inv}
\tr \rho_A(\tau g \tau)=\tr \rho_A(g) \text{ for any }g \in \Gamma_K.
\end{align}

As $\psim$ and $[\eta]$ are both non-trivial, the centralizer of the image of $\rho$ consists only of scalar matrices, 
hence $\cD^{\univ}_{\rho}$ is representable by a universal deformation ring $\cR^{\univ}_{\rho}$ (see   \cite{Mazur}).

Let $\cD_{\Fil}$ be the functor assigning to $A$ in $\cC$ the set of free  $A$-submodules $F_A\subset A^2$ 
such that $F_A \otimes_{A} \overline{\Q}_p= F_{\rho}$ (such $F_A$ is  necessarily a  direct rank $1$ summand of $A^2$). It is representable by the strict completed local ring $\cR_{\Fil}\simeq \overline{\Q}_p\lsem U\rsem $ of the projective  space $\mathbb{P}^1$ at the point corresponding to the line $F_{\rho} \subset \overline{\Q}{}_p^2$. The universal submodule  $F_{\cR_{\Fil}} \subset \cR_{\Fil}^2$ has basis given by $e_1+U  e_2$.
\begin{defn}\label{cond1}
The functor $\cD^{\ord}_{\rho}$   assigns to $A\in  \cC$ the set of tuples $(\rho_A,F_A)$ such that 
\begin{enumerate}[wide]
\item $\rho_A:\Gamma_K \to \GL_2(A)$ is a continuous representation such that
 $\rho_A \mod \gm_A = \rho$, and  $\tr \rho_A(\tau g \tau)=\tr \rho_A(g)$ for all
 $g \in \Gamma_K$, and 
\item $F_A \subset A^2$ is a free direct factor  over $A$ of rank $1$  which is $\Gamma_{K_{\gp}}$-stable and 
such that $\Gamma_{K_{\gp}}$ acts  on $A^2/ F_A$ by an unramified character, denoted $\chi_A$, 
 \end{enumerate}
 modulo the strict equivalence relation  $[(\rho_A,F_A)]=[(P\rho_AP^{-1},P\cdot F_A)]$ with $P\in 1+M_2(\gm_A)$. 
\end{defn}

One  has $F_A \otimes_{A} \overline{\Q}_p= F_{\rho}$ as the latter is the unique $\Gamma_{K_{\gp}}$-stable line in $\overline{\Q}_p^2$ with  unramified quotient. 
As in \cite[\S1.2]{BDPozzi}, given a representative of the universal deformation of $\rho$ to $\cR^{\univ}_{\rho}$, 
there exists a monic natural transformation $\cD^{\ord}_{\rho} \hookrightarrow \cD^{\univ}_{\rho}\times\cD_{\Fil}$.  Moreover, it has been shown in {\it loc. cit.} that the composition $\cD^{\ord}_{\rho} \hookrightarrow \cD^{\univ}_{\rho}\times\cD_{\Fil} \to \cD^{\univ}_{\rho}$ is again monic  and is independent of the above choice of representative. Therefore  $\cD^{\ord}_{\rho}$ is a sub-functor of $\cD^{\univ}_{\rho}$ defined by a closed condition  guaranteeing its representability. In particular, $(\rho_A,F_A) \in \cD^{\ord}_{\rho} (A)$ is characterised uniquely by $\rho_A$ alone, {\it i.e.} when the ordinary filtration exists, then it is unique. For this reason, and because  $\chi_A$ plays a major role in this paper, we will  denote the  point $(\rho_A,F_A)$ of $\cD^{\ord}_{\rho} (A)$ by $(\rho_A,\chi_A)$.  

\begin{lemma}\label{extdet}
\begin{enumerate}[wide]
 
\item The functor $\cD^{\ord}_{\rho}$ is representable by a quotient $\cR^{\ord}_{\rho}$ of $\cR^{\univ}_{\rho}$ and a universal 
ordinary deformation  $\rho_{\cR}:\Gamma_K \to \GL_2(\cR^{\ord}_\rho)$.

\item The determinant of $\rho_{\cR}$ extends in two different ways to a character   $\Gamma_{\Q} \to (\cR^{\ord}_\rho)^{\times}$.
\end{enumerate}
\end{lemma}

\begin{proof} (i) The argument is exactly the same as in \cite[\S1.2]{BDPozzi}.

(ii) As $2 \det(\rho_{\cR}(g))=\tr(\rho_{\cR}(g))^{2} - \tr(\rho_{\cR}(g^{2}))$, we deduce that 
$\det(\rho_{\cR})$ is a $\tau$-invariant character of $\Gamma_K$,  hence can be extended to a character of 
$\Gamma_{\Q}$ by sending the order $2$ element $\tau$ either to $1$ or to $-1$. 
Alternatively, one can observe that the character $\det(\rho_{\cR})\det(\rho)^{-1}:\Gamma_K\to 1+\gm_{\cR^{\ord}_\rho}$
extends uniquely,  because $\cR^{\ord}_\rho$ is complete and $K$ has a unique $\Z_p$-extension on which $\tau$ acts trivially (this argument is more general as it does not use the existence of a section 
of  $\Gamma_{\Q}\to\Gamma_{K/\Q}$ given by the complex conjugation $\tau$).\end{proof}

As the two extensions in Lemma~\ref{extdet}(ii)  differ by the quadratic character $\epsilon_K$ of $\Gamma_{K/\Q}$, exactly one of them is odd, and we denote it $\det(\rho_{\cR}): \Gamma_{\Q} \to \cR^{\ord}_{\rho}$. It reduces modulo the maximal ideal to 
$\det( \Ind^\Q_K \psi)= (\psi\circ\mathrm{Ver})\cdot \epsilon_K : \Gamma_{\Q} \to \overline{\Q}{}_p^\times$, 
where $\mathrm{Ver}: \Gamma_{\Q}^\mathrm{ab}\to \Gamma_{K}^\mathrm{ab}$ denotes the transfer homomorphism. 
As $\varLambda\in \cC$ is the universal deformation ring of that character  (see \cite[\S6]{bellaiche-dimitrov}) the natural transformation 
$\rho_A\mapsto \det(\rho_A)$, endows $\cR^{\ord}_{\rho}$ with a $\varLambda$-algebra structure.

\subsection{Reducible and CM deformations of \texorpdfstring{$\rho$}{}}\label{red-def}

\begin{defn}\label{defn:red}\

 \begin{enumerate}[wide]

\item Let  $\cD^{\red}_{\rho}$ be the subfunctor of $\cD^{\ord}_{\rho}$ consisting of $\Gamma_K $-reducible deformations. 

\item Let $\cD^K_{\rho}$ be the subfunctor of $\cD^{\ord}_{\rho}$ consisting of deformations of $\rho$ having a free  $\Gamma_K $-quotient of rank $1$  which is unramified at $\gp$. 

\end{enumerate}
\end{defn}

\begin{lemma}\label{lemma-red}\ \begin{enumerate}[wide]

\item The functor $\cD^{\red}_{\rho}$ is representable by a quotient $\cR^{\red}_{\rho}$ of $\cR^{\ord}_{\rho}$. 
\item The functor $\cD^K_{\rho}$ is representable by a quotient $\cR^K_{\rho}=\cR^{\ord}_{\rho}/(c(g); g\in \Gamma_K )$, where $\rho_{\cR}(g)=\left(\begin{smallmatrix} a(g)&b(g)\\ c(g)& d(g)\end{smallmatrix}\right)$ in an ordinary basis. \end{enumerate}
\end{lemma}
\begin{proof} 
(i) As $\rho(\gamma_0)$ is a diagonal matrix with  distinct eigenvalues,  Hensel's lemma implies that 
$\rho_{\cR}(\gamma_0)$ can be strictly conjugated to a diagonal matrix as well. The assertion then follows from exactly the same arguments as those in the proof of   \cite[Lem.~1.3]{BDPozzi}. 

(ii) The statement is clear as the  $\Gamma_K$-filtration necessarily coincides with the 
ordinary filtration. 
\end{proof}

Finally, for $? \in \{\univ, \ord,\red,K\}$ we  let  $\cD^{?}_{\rho,0}$ denote the sub-functor of  $\cD^{?}_{\rho}$ classifying deformations with fixed determinant, represented by $\cR^{?}_{\rho,0}=\cR^{?}_{\rho}/\gm_{\varLambda}\cR^{?}_{\rho}$ (see   \S\ref{ord-def}).

\begin{rem}
While $\rho$ is reducible and ordinary for the same filtration \eqref{filtration-rho}, this is not necessarily  true for
deformations in $\cD^{\red}_{\rho}$. In fact, we show in \S\ref{tangent} that 
\[\dim \cD^K_{\rho}(\overline{\Q}_p[\epsilon])=1 < \dim \cD^{\red}_{\rho}(\overline{\Q}_p[\epsilon])=\dim \cD^{\ord}_{\rho}(\overline{\Q}_p[\epsilon])=2.\]

\end{rem}

\subsection{Non-CM deformations of \texorpdfstring{$\rho$}{}}\label{daggerfunct}

We recall the indecomposable representation
 $\rho= \left(\begin{smallmatrix} \psi & \eta\psi^\tau \\ 0 &\psi^\tau \end{smallmatrix}\right)$  from \S\ref{ord-def}
 and its universal ordinary deformation $\rho_{\cR}$. 
 
 We can do a similar construction with $\psim^\tau$ instead of $\psim$. Namely, we let  
 $\eta':\Gamma_K\to\overline{\Q}_p$ be the unique representative of $[\eta'] \in \rH^1(K,\psim^\tau)$ 
whose restriction to  $\I_{\gp}$ is given by the $p$-adic logarithm and such that $\eta'(\gamma_0)=0$. Then we consider the  indecomposable representation 
\[\rho' =
\left(\begin{smallmatrix} \psi^\tau & \eta'\psi \\ 0 &\psi\end{smallmatrix}\right):\Gamma_K\to \GL_2(\overline{\Q}_p)\]
and denote by  $\cD^{\ord}_{\rho' }$ its ordinary deformation functor, representable by a  universal deformation ring $\cR^{\ord}_{\rho' }$. By definition the universal deformation $\rho'_{\cR}$ has a $\Gal(K/\Q)$-invariant trace  and 
 admits a rank $1$  unramified  $\Gamma_{K_{\gp}}$-quotient, denoted $\chi'_{\cR}$. 

We highlight that as  $[\eta^{\tau}]$ and $[\eta']$ are proportional (in fact $\eta^{\tau}=\mathscr{S}(\psim) \eta'+\eta^{\tau}(\gamma_0) \tfrac{\psim-1}{\psim(\gamma_0)-1}$), it follows that   
$\rho'\simeq \rho^\tau$. However  $\rho'_{\cR}$ need {\it not}  be isomorphic to  $\rho _{\cR}^\tau$
as the former admits a $\Gamma_{K_{\gp}}$-unramified quotient, while the latter 
admits a $\Gamma_{K_{\bar\gp}}$-unramified quotient, and not necessarily {\it vice versa}.

The following definition (similar to \cite[Def.1.6]{BDPozzi})  is an attempt to describe Galois theoretically the local ring at $f$ of the closed analytic subspace $\cE^{\perp}$ of $\cE^{\ord}$  defined by the union of irreducible components without CM by $K$.  As we will later  show in Theorem~\ref{existencenon-cm} one indeed has $f \in \cE^{\perp}$.

\begin{defn}\label{cond2}
Let $\cD^{\perp}$ be the functor assigning to  $A \in \cC$ the set of strict equivalence classes of pairs $((\rho_A, \chi_A),(\rho'_A, \chi'_A))$ in $\cD^{\ord}_{\rho}(A)\times \cD^{\ord}_{\rho'} (A)$ such that $\tr(\rho_A)=\tr(\rho'_A )$ and
  $\chi_A(\Frob_{\gp})= \chi'_A(\Frob_{\gp})$.
\end{defn}
The definition of $\cD^{\perp}$ is meant to  exclude $\varLambda$-adic deformations of $\rho$ having CM by $K$. 
Note that the $\tau$-conjugate of any $\gp$-ordinary deformation of $\rho$ is a $\bar \gp$-ordinary deformation of $\rho'$ which is  
 not necessarily $\gp$-ordinary, so it does not in general define a point of our functor.

The functor $\cD^{\perp}$ is representable by a $\varLambda$-algebra $\cR^{\perp}$, quotient
 of $ \cR^{\ord}_{\rho} \widehat{\otimes} \cR^{\ord}_{\rho'} $ by the ideal 
\[ \left(\tr (\rho_{\cR})(g) \otimes 1 - 1\otimes \tr(\rho'_{\cR})(g),
\chi_{\cR}(\Frob_{\gp})\otimes 1 - 1\otimes \chi'_{\cR}(\Frob_{\gp}); g\in \Gamma_K 
\right). \]

Let $\rho_{\cR^\perp}$ (resp. $\rho'_{\cR^\perp}$) be the $\gp$-ordinary deformation of $\rho$ (resp. $\rho' $) obtained by  functoriality from the natural homomorphism $\cR_{\rho}^{\ord} \rightarrow \cR^{\perp}$ (resp. $\cR_{\rho' }^{\ord} \rightarrow \cR^{\perp}$).
\begin{lemma}\label{lem-surj}
The natural homomorphisms $\cR^{\ord}_{\rho} \to \cR^{\perp}$ and $\cR^{\ord}_{\rho'} \to \cR^{\perp}$ are 
surjective. Moreover  $\rho'_{\cR^\perp}$ and   $\rho _{\cR^\perp}^\tau$ are conjugated by an element of $\GL_2(\cR^\perp)$, in particular
$\rho _{\cR^\perp}^\tau$ is $\gp$-ordinary. 
\end{lemma}

\begin{proof} The first claim follows from similar arguments that have already been used in \cite[Lem.~1.7]{BDPozzi}. 
In fact, as $\dim  \rH^1(K, \psim)=  \dim  \rH^1(K, \psim^\tau)=1$ by Proposition~\ref{restinj}, one can apply \cite[Cor.~1.1.4(ii)]{kisin}
 to prove that 
the algebras $\cR^{\univ}_{\rho}$ and $\cR^{\univ}_{\rho'}$ (and hence any of their quotients) are generated over $\bar\Q_p$ by the traces of their universal deformations. 

As highlighted before, the representations   
$\left(\begin{smallmatrix} \mathscr{S}(\psim) & 
\tfrac{\eta^{\tau}(\gamma_0)}{1-\psim(\gamma_0)} \\ 0 &1 \end{smallmatrix}\right)^{-1}
\rho _{\cR^\perp}^\tau \left(\begin{smallmatrix} \mathscr{S}(\psim) & 
\tfrac{\eta^{\tau}(\gamma_0)}{1-\psim(\gamma_0)} \\ 0 &1 \end{smallmatrix}\right)$ and $\rho'_{\cR^\perp}$
both reduce to $\rho'$ and share the same trace, hence  by {\it loc. cit.} they correspond to  the same point of $\cD^{\univ}_{\rho'}(\cR^\perp)$ and thus they are 
 conjugated by an element of $1+ M_2(\gm_{\cR^\perp})$.\end{proof}

The local ring $\cR^{\perp}_{0}=\cR^{\perp}/\gm_{\varLambda}\cR^{\perp}$ represents the subfunctor $\cD^{\perp}_{0}$ of $\cD^{\perp}$ consisting  of deformations with fixed determinant.

\subsection{Tangent spaces}\label{tangent}
Using  results from Galois Cohomology, Class Field Theory and Transcendence Theory we will now  
describe the  tangent spaces of the functors introduced in the previous subsections and 
compute their dimensions.  
 
 Let $\overline{\Q}_p[\epsilon]=\overline{\Q}_p \lsem X\rsem/(X^2)$ be the $\overline{\Q}_p$-algebra of dual numbers.
There is a  natural injection 
\[ t^{\univ}_{\rho}=\cD^{\univ}_{\rho}(\overline{\Q}_p[\epsilon]) \hookrightarrow \rH^1(K, \ad(\rho)),\]
inducing  an injection $t^{\univ}_{\rho,0}=\cD^{\univ}_{\rho,0}(\overline{\Q}_p[\epsilon]) \hookrightarrow \rH^1(K, \ad^0(\rho))$, where $\ad(\rho)$, resp. $\ad^0(\rho)$,  is the adjoint representation 
on $\End_{\overline{\Q}_p}(V_{\rho})$, resp. on its trace zero elements $\End^0_{\overline{\Q}_p}(V_{\rho})$. We have a natural decomposition of $\Gamma_K$-modules $\ad \rho \simeq \ad \rho^0 \oplus \mathbf{1}$.
 For $? \in\{\ord,\red,K\}$,  the  tangent space 
$t^{?}_{\rho}=\cD^{?}_{\rho}(\overline{\Q}_p[\epsilon])$, resp. $t^{?}_{\rho,0}=\cD^{?}_{\rho,0}(\overline{\Q}_p[\epsilon])$, is  thus naturally a subspace of $\rH^1(K, \ad(\rho))$, resp. $\rH^1(K, \ad^0(\rho))$. Let 
\begin{align}\label{eq:ses-W}
0 \to W_{\rho}\to \End_{\overline{\Q}_p}(V_{\rho}) \to \Hom_{\overline{\Q}_p}(F_{\rho}, V_{\rho}/F_\rho) \simeq \psim^{\tau}\to 0
\end{align}
be the short exact sequence 
of $\overline{\Q}_p[\Gamma_{K}]$-modules arising from  the fact that $\rho$ is reducible and   let  $W_\rho^0=W_\rho\cap \End^0_{\overline{\Q}_p}(V_{\rho})$.  As $\rho=\left( \begin{smallmatrix} \psi& \eta\psi^\tau\\ 0 & \psi^\tau \end{smallmatrix}\right)$, the subspace $W_\rho$ of $\End_{\overline{\Q}_p}(V_{\rho})= \mathrm{M}_2(\overline{\Q}_p)$ consists of upper triangular matrices  $\left( \begin{smallmatrix} a& b\\ 0 & d \end{smallmatrix}\right)$.  Consider  the short exact sequences  of $\overline{\Q}_p[\Gamma_{K}]$-modules
\begin{align}\label{eq:ses-W'}
\begin{split}
0 \to & W'_{\rho}\to W_{\rho} \xrightarrow{d} \overline{\Q}_p \to 0, \text{ and }  \\ 
0 \to &  W'_\rho{}^0\simeq\psim  \to W_\rho  \xrightarrow{(a,d)} \overline{\Q}{}_p^2 \to 0. 
\end{split}
\end{align}

\begin{lemma}\label{a=dtau} Let $a,d \in \Hom(\Gamma_K,\overline{\Q}_p)$ and $s \in \{\pm 1\}$. 
If  $\psi a + s  \psi^\tau d= s  \psi^\tau a^{\tau} + \psi d^\tau$ as functions on $\Gamma_K$, then 
$d=a^\tau$. 
\end{lemma}

\begin{proof}
 The relation can be rewritten as $\psim (a-d^\tau)= s  (a^{\tau}-d)$. 
As $\psim$ is unramified at $p$, it follows that the elements $(a-d^\tau)$ and $s  (a^{\tau}-d)$ of $\Hom(\Gamma_K,\overline{\Q}_p)$
coincide  on both inertia subgroups $\I_{\gp}$ and $\I_{\bar{\gp}}$, hence they are equal. However $\psim\ne \mathbf{1}$,  hence  $a-d^\tau=0$. \end{proof}

\begin{defn}\label{def:tau-inv}
Let $\Hom(\Gamma_K, \overline{\Q}{}_p^2)^\tau$  be the subspace  of  $ \Hom(\Gamma_K, \overline{\Q}{}_p^2)$
consisting of elements $(a,d)$ such that $d=a^{\tau}$, and let 
 $\rH^1(K,W_\rho)^\tau$ be the inverse image of $\Hom(\Gamma_K, \overline{\Q}{}_p^2)^\tau$ under the morphism $  (a_*,d_*): \rH^1(K,W_\rho) \to  \Hom(\Gamma_K, \overline{\Q}{}_p^2)$  coming by functoriality from \eqref{eq:ses-W'}. 
 
 Put  $\rH^1(K,W_\rho^0)^\tau=\rH^1(K,W_\rho)^\tau \cap \rH^1(K,W_\rho^0)$.
\end{defn}

\begin{prop} \label{prop:td} We have $t^{\red}_{\rho}=t^{\ord}_{\rho}\simeq \rH^1(K,{W_\rho})^{\tau}$ and  $t^{\red}_{\rho,0}=t^{\ord}_{\rho,0}\simeq \rH^1(K,{W^0_\rho})^{\tau}$.

\end{prop}

 \begin{proof}  By Definition~\ref{defn:red} we have  $t^{\red}_{\rho} \subset t^{\ord}_{\rho}$ . 
 To prove the opposite inclusion, we represent  an infinitesimal ordinary deformation   $[ \rho_\epsilon] \in t^{\ord}_{\rho}$ by a cocycle 
  \begin{align}
 \rho_\epsilon=\left(1+\epsilon \left(\begin{smallmatrix}a & b \\ c & d \end{smallmatrix}\right) \right)\rho, 
 \text{ where } \left[\left(\begin{smallmatrix}a & b \\ c & d \end{smallmatrix}\right)\right]\in \rH^1(K, \ad(\rho)). 
 \end{align}

As $\rH^0(K,\psim^\tau)=\{0\}$, we obtain from \eqref{eq:ses-W}  the following long exact sequence
\begin{align}\label{exact-adjoint}
0 \to \rH^1(K,W_\rho) \longrightarrow \rH^1(K,\ad \rho) \xrightarrow{c_*} \rH^1(K,\psim^\tau)\to \rH^2(K,W_\rho),
\end{align}
where the  map $\ad \rho\to\psim^\tau$   is given by  $\left[\left(\begin{smallmatrix}a & b \\ c & d \end{smallmatrix}\right) \right]\mapsto[c]$. 
 By ordinarity (see   \cite[Proof of Prop.~2.1]{BDPozzi}), the image of $t^{\ord}_{\rho} \subset \rH^1(K,\ad \rho)$ under the morphism $c_*$ lands in 
\[\ker \left(\rH^1(K, \psim^\tau)\to \rH^1(K_{\gp}, \overline{\Q}_p) \right),\] 
 which is trivial by Proposition~\ref{restinj}. Thus  $t^{\ord}_{\rho} = t^{\red}_{\rho} \subset \rH^1(K,{W_\rho})$. 
 
 By \eqref{exact-adjoint} any element of  $\rH^1(K,{W_\rho})$  
  can be represented by an infinitesimal reducible lift
\[\rho_\epsilon= \left(1+\epsilon  \begin{pmatrix} a & b\\ 0 & d \end{pmatrix} \right )\rho=
\begin{pmatrix}   \psi(1+\epsilon a)  & \psi^\tau(\eta  + \epsilon (a \eta + b ))\\  0 & \psi^\tau(1 +\epsilon d)\end{pmatrix}. \]
If the lift comes from an element of $t^{\ord}_{\rho}$, because of the relation $\tr \rho_\epsilon=\tr\rho^\tau_\epsilon$
(see \eqref{tau-inv}), Lemma~\ref{a=dtau} implies that $a=d^{\tau}$, hence  $t^{\ord}_{\rho} \subset \rH^1(K,{W_\rho})^{\tau}$ (see  Definition~\ref{def:tau-inv}). 
 The same argument also shows that $t^{\red}_{\rho,0}=t^{\ord}_{\rho,0}\subset \rH^1(K,{W^0_\rho})^{\tau}$.

 It remains to show the  inclusion $ \rH^1(K,{W_\rho})^\tau \subset t^{\ord}_{\rho}$. As observed above, any element of 
 $\rH^1(K,{W_\rho})^\tau$   can be represented by an infinitesimal reducible lift 
 \begin{align}\label{rho-adjoint}
\rho_\epsilon=\left(1+\epsilon  \begin{pmatrix} d^\tau & * \\ 0 & d \end{pmatrix} \right )\rho=
\begin{pmatrix}  \psi(1+\epsilon d^\tau)  & \eta \psi^\tau + \epsilon\cdot *\\ 0 & \psi^\tau(1 +\epsilon d)\end{pmatrix},
\end{align}
 where $d \in \rH^1(K, \overline{\Q}_p)$.
 It is sufficient to show that $\rho_\epsilon$ is  $\gp$-ordinary, {\it i.e.}, find $\alpha\in \overline{\Q}_p$ such that the line of $\overline{\Q}_p[\epsilon]^2$ generated by  $e_{1}+ \epsilon \alpha  e_{2}$  is $\Gamma_{K_{\gp}}$-stable  
with  unramified quotient.  A direct computation shows that the restriction of $\rho_\epsilon$ to
$\Gamma_{K_{\gp}}$ in the {\it ordinary} basis $(e_{1}+ \epsilon \alpha  e_2,e_2)$ equals
\[
\left( \begin{matrix} 1 & 0\\ -\epsilon \alpha & 1 \end{matrix}\right)
\left( \begin{matrix} \psi(1+\epsilon d^\tau) & \eta \psi + \epsilon \cdot *\\ 0 & \psi(1 +\epsilon d) \end{matrix}\right)
\left( \begin{matrix} 1 & 0\\ \epsilon \alpha & 1 \end{matrix}\right)=
\left( \begin{matrix}  \psi(1+\epsilon ( d^\tau +\alpha \eta))  &  \eta \psi + \epsilon \cdot * \\ 0 & \psi(1+\epsilon ( d- \alpha \eta)) \end{matrix}\right).
\]
As $\eta_{|\I_{\gp}}=\log_{\gp}$ generates the image of the restriction map $ \rH^1(K, \overline{\Q}_p)\to \rH^1(\I_{\gp}, \overline{\Q}_p)$,
there exists a unique $\alpha \in \overline{\Q}_p$ such that $d_{|\I_{\gp}}= \alpha\eta_{|\I_{\gp}}$, for which the 
character $\chi_\epsilon= \psi(1+\epsilon ( d- \alpha \eta))$ will be unramified at $\gp$. This completes the proof.\end{proof}

 Let  $\rH^1(K,W^0_\rho) \xrightarrow{d_*} \rH^1(K,\overline{\Q}_p)$ be  morphism induced by \eqref{eq:ses-W'} sending  $\left[\left(\begin{smallmatrix}-d & b \\ 0 & d \end{smallmatrix}\right)\right]$  to  $[d]$.

\begin{lemma}\label{H2}\
\begin{enumerate}[wide]
\item  One has $\dim_{\overline{\Q}_p} \rH^2(K, \psim)=0$.
\item  The map $\rH^1(K,W^0_\rho) \xrightarrow{d_*}  \rH^1(K,\overline{\Q}_p)$ is an isomorphism.
\end{enumerate}
\end{lemma}
\begin{proof} 
(i) By the  global Euler characteristic formula the dimension of $\rH^2(K, \psim)$ equals 
\[\dim \rH^1(K,\psim)-\dim \rH^0(K,\psim)+ \dim \rH^0(\C,\overline{\Q}_p)-[K:\Q]\dim(\psim)=1-0+1-2=0.\]

(ii) As $\rH^2(K,\psim)=\{0\}$ and $\dim\rH^1(K, \psim)=1$,  the map $d_*$  in the long exact sequence
\[ 0 \to \rH^0(K,\overline{\Q}_p) \xrightarrow{\sim} \rH^1(K, \psim) \to \rH^1(K,W_\rho^0) \xrightarrow{d_*}  \rH^1(K,\overline{\Q}_p) \to 0\]
coming from  \eqref{eq:ses-W'}  is necessarily an isomorphism.\end{proof}

\begin{prop}\label{dim t_R} One has $\dim  t^{\ord}_{\rho}=2$, $\dim t^{\ord}_{\rho,0}=\dim t^K_{\rho}=1$   and  $t^K_{\rho,0}=\{ 0\}$.
\end{prop}

\begin{proof} Using the direct sum decomposition $W_\rho=W_\rho^0\oplus \overline{\Q}_p$, Lemma~\ref{H2}(ii)
implies that 
\[(a_*-d_*, a_*+d_*): \rH^1(K,W_\rho)\xrightarrow{\sim} \rH^1(K,\overline{\Q}_p)\oplus \rH^1(K,\overline{\Q}_p)\]
is an  isomorphism, from which one deduces the  isomorphism
\begin{align}\label{sole-d}
d_*: t^{\ord}_{\rho}=\rH^1(K,W_\rho)^{\tau}\xrightarrow{\sim}
 \rH^1(K,\overline{\Q}_p)^{\tau=-1}\oplus \rH^1(K,\overline{\Q}_p)^{\tau=1}= \rH^1(K,\overline{\Q}_p).
 \end{align} 
 In view of  Proposition~\ref{prop:td}, we deduce  $\dim t^{\red}_{\rho}=\dim t^{\ord}_{\rho}=\dim \rH^1(K,\overline{\Q}_p)=2$. 
 Moreover, as $d_*$ 
 induces an isomorphism between $\rH^1(K,W_\rho^0)^{\tau}$
and the anti-cyclotomic line  $\rH^1(K,\overline{\Q}_p)^{\tau=-1}$, one deduces that $\dim t^{\red}_{\rho,0}=\dim t^{\ord}_{\rho,0}=1$. 

By definition $t^K_{\rho}= \ker( t^{\red}_{\rho}\simeq  \rH^1(K,\overline{\Q}_p) \xrightarrow{\res_{\gp}} \Hom(\I_{\gp},\overline{\Q}_p))
$, hence $\dim t^K_{\rho}=1$. Finally 
 \begin{align}\label{cmequa}
t^K_{\rho,0}= \ker\left(t^{\red}_{\rho,0}\simeq  \rH^1(K,\overline{\Q}_p)^{\tau=-1}  \xrightarrow{\res_{\gp}} \Hom(\I_{\gp},\overline{\Q}_p)\right)=\{ 0 \}. 
\end{align} \qedhere        \end{proof}

Let $t^{\perp}=\cD^{\perp}(\overline{\Q}_p[\epsilon])$, resp.  $t^{\perp}_0=\cD^{\perp}_{0}(\overline{\Q}_p[\epsilon])$, be the tangent, resp. 
relative  tangent,  space of the functor $\cD^{\perp}$. Recall the $\cL$-invariants introduced in Definition~\ref{d:new-L-inv}
\[\cLm(\psim^\tau)=(\eta-\eta_{\gp}+\eta_{\bar{\gp}})(\Frob_{\gp})  \text{ and } \cLm(\psim)=(\eta'-\eta_{\gp}+\eta_{\bar{\gp}})(\Frob_{\gp}).\]

\begin{thm}\label{cuspidal-tangent} 
\begin{enumerate}[wide]
We have $\dim t^{\perp}=1$. Moreover,  $t^{\perp}_0=\{0\}$ if and only if $\cLm(\psim)+\cLm(\psim^\tau) \ne 0$. In particular,  $t^{\perp}_0=\{0\}$ if  $\psim$ is quadratic.
\end{enumerate}
\end{thm}

\begin{proof}  By definition 
\[t^{\perp}=\cD^{\perp}(\overline{\Q}_p[\epsilon])=\left\{[(\rho_\epsilon,\chi_\epsilon), (\rho'_\epsilon,\chi'_\epsilon )] \in t^{\ord}_{\rho} \times t^{\ord}_{\rho'} \mid \tr(\rho_\epsilon)= \tr(\rho' _\epsilon), \chi_\epsilon(\Frob_{\gp})=\chi'_\epsilon(\Frob_{\gp})\right\}.\]
By \eqref{sole-d}, $d_*$ identifies $t^{\ord}_{\rho}$ with $\rH^1(K,\overline{\Q}_p)$, hence  $[\rho_\epsilon]$ is uniquely determined  by an element  $d\in  \rH^1(K,\overline{\Q}_p)$ as follows (see \eqref{rho-adjoint})
\begin{align}\label{eq:infinitesimal}
 \rho_\epsilon =\begin{pmatrix}
\psi (1+\epsilon d^\tau) &  \psi^\tau \eta + \epsilon \cdot * \\ 0 &\psi^\tau (1+\epsilon d) \end{pmatrix}. 
\end{align}
Write $d=\alpha \eta_{\gp} + \beta \eta_{\bar{\gp}} $ with $\alpha, \beta \in \overline{\Q}_p$ in the basis $\{\eta_{\gp},\eta_{\bar{\gp}}\}$
of $\rH^1(K, \overline{\Q}_p)$ (see \S\ref{def-eta}). Note that
 \begin{align}\label{dag-trace}
  \tr(\rho_\epsilon)= \psi + \psi^\tau+\epsilon(\psi d^{\tau} +  \psi^\tau d)  \text{ and }   \det(\rho_\epsilon)= \psi \psi^\tau (1+\epsilon(d+ d^{\tau})).
\end{align}

Analogously, an element of $t^{\ord}_{\rho'}$ can be represented by an infinitesimal reducible lift 
\[ \rho'_\epsilon =\begin{pmatrix}
\psi^\tau (1+\epsilon d'{}^\tau) &  \psi\eta' + \epsilon \cdot *  \\ 0 &\psi (1+\epsilon d') \end{pmatrix},  \]
 uniquely determined  by  $d'\in \rH^1(K,\overline{\Q}_p)$.  
As $\tr \rho_\epsilon'=\tr \rho_{\epsilon}$, Lemma~\ref{a=dtau} implies that $d'=d^\tau$. 

By the end proof of Proposition~\ref{prop:td}, one has $\chi_\epsilon= \psi(1 + \epsilon (d-\alpha \eta))$, hence 
 \begin{align}\label{U_p equation1}  
 \chi_\epsilon(\Frob_{\gp})=(1 + \epsilon (d-\alpha \eta)(\Frob_{\gp})) \psi(\gp)= (1 + \epsilon [\alpha (\eta_{\gp}-\eta)(\Frob_{\gp}) + \beta \eta_{\bar{\gp}}(\Frob_{\gp})])\psi(\gp). \end{align}
Analogously, from $d'=d^\tau=\beta \eta_{\gp} +  \alpha\eta_{\bar{\gp}} $, one deduces $\chi'_\epsilon= \psi(1 + \epsilon (d^\tau-\beta \eta'))$
hence 
\[ \chi'_\epsilon(\Frob_{\gp})=(1 + \epsilon (d^\tau-\beta \eta')(\Frob_{\gp})) \psi(\gp)= (1 + \epsilon (\beta (\eta_{\gp}-\eta')(\Frob_{\gp}) + \alpha \eta_{\bar{\gp}}(\Frob_{\gp})))\psi(\gp).\]

Comparing with  \eqref{U_p equation1} shows that 
\[ t^{\perp}\simeq\{(\alpha,\beta) \in \overline{\Q}{}_p^2\, | \,\beta (\eta_{\gp}-\eta' )(\Frob_{\gp}) + \alpha \eta_{\bar{\gp}}(\Frob_{\gp}) =
 \alpha (\eta_{\gp}-\eta)(\Frob_{\gp}) + \beta \eta_{\bar{\gp}}(\Frob_{\gp})\}. \]
Using Lemma~\ref{linvariant1} and Proposition~\ref{linvariant2} the tangent space can be  rewritten in terms of $\cL$-invariants: \begin{align}\label{U_pL}
t^{\perp}\simeq\{(\alpha,\beta) \in \overline{\Q}{}_p^2\, | \, \alpha \cLm(\psim^\tau) =\beta \cLm(\psim)\}.
\end{align}
By Proposition~\ref{6expocusp},  $\cLm(\psim^\tau)$ and  $\cLm(\psim)$ cannot be both zero. Hence $\dim t^{\perp}=1$.

(ii)  It follows from  \eqref{dag-trace} that the tangent space  $t^{\perp}_0$ is parametrized by $(\alpha,\beta) \in t^{\perp}$ such that  $\alpha + \beta=0 $. Hence $t^{\perp}_0=\{0\}$ if and only if $\cLm(\psim)+\cLm(\psim^\tau)\ne 0$.\end{proof}

Taking $(\alpha,\beta)$ from the above computations as  coordinates of the plane  $t^{\ord}_{\rho}$, 
the lines $t^{K}_{\rho}$ , $t^{\perp}$ and $t^{\ord}_{\rho,0}$ can be  drawn as  follows
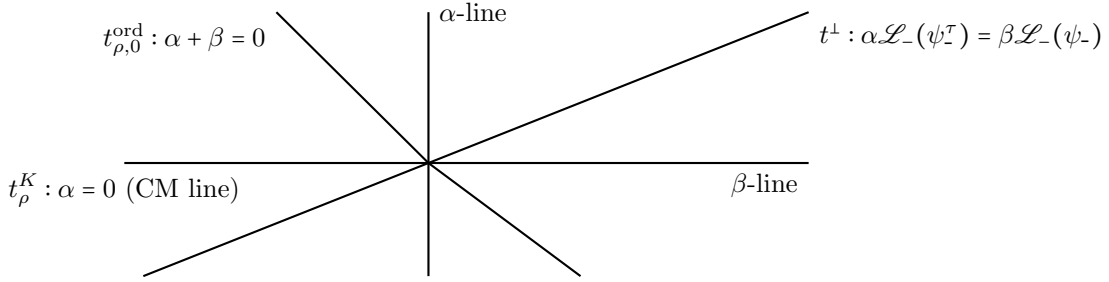
\begin{center} 
\begin{tikzpicture}
\draw[thick,] (0,0) -- (5,0) node[anchor=north east]  {{\small $\beta$-line}};
\draw[thick,] (0,0) -- (0,2) node[anchor=west ] {{\small $\alpha$-line}};
\draw[thick,] (0,0) -- (5,2) node[anchor=north west] {{\small $t^{\perp}:\alpha \cL_{-}(\psim^{\tau})=\beta \cL_{-}(\psim)$}};
\draw[thick,] (0,0) -- (2,-1.5) node[anchor=south east] {};
\draw[thick,] (0,0) -- (0,-1.5) node[anchor=south east] {};
\draw[thick,] (0,0) -- (-4,0) node[anchor=north ] {{\small $t_{\rho}^{K}:\alpha=0$  (CM line)}};
\draw[thick,] (0,0) -- (-3.75,-1.5) node[anchor=north east] {};
\draw[thick,] (0,0) -- (-2,2) node[anchor=north east] {{\small $t^{\ord}_{\rho,0}:\alpha + \beta=0$}};
\end{tikzpicture}
 
 \end{center}
 
\begin{cor}\label{system2}  We have $t^{\ord}_{\rho}=t^{\perp} \oplus t^K_{\rho}$ if and only if $\cLm(\psim) \ne 0$.
\end{cor}
\begin{proof}  As in the proof of Theorem~\ref{cuspidal-tangent} one can use 
$(\alpha,\beta) \in \overline{\Q}{}_p^2$ as coordinates on $t^{\ord}_{\rho}$. By \eqref{U_pL}  the equation defining $t^{\perp}$ is 
$\alpha \cLm(\psim^\tau) =\beta \cLm(\psim)$, while by \eqref{cmequa} the equation defining $t^K_{\rho}$ is  $\alpha=0$.\end{proof}

\section{Components of the  eigencurve containing \texorpdfstring{$f$}{f}} \label{section3}

The weight  $1$ theta series $\theta_\psi$ has level  $N=D\cdot \mathrm{N_{K/\Q}}(\gc_\psi) $, 
where $-D$ is the fundamental discriminant of $K$ and $\gc_\psi$ is the conductor of $\psi$. 

The $p$-adic cuspidal  eigencurve $\cE$  of tame level $N$ is endowed with a flat and locally finite morphism
 $\kappa: \cE \rightarrow \cW$ of reduced rigid analytic spaces over $\Q_{p}$, where the weight space $\cW$ represents continuous homomorphisms from  $\Z_p^{\times}\times (\Z /N \Z)^{\times}$ to $\mathbb{G}_m$. 
Locally  $\cO_\cE$ is generated as $\cO_\cW$-algebra by the  Hecke operators $U_p$ and $T_\ell$ for $\ell \nmid Np$.  
There exists a universal $2$-dimensional pseudo-character $\tau_{\cE}:\Gamma_{\Q} \rightarrow  \cO(\cE)$
unramified at  $\ell \nmid Np$ and sending $\Frob_\ell$ on $T_\ell$ (see \cite[(15)]{bellaiche-dimitrov}).  It is defined by interpolation, so that
its specialization at any classical point of $\cE(\overline{\Q}_p)$ equals the trace of the semi-simple $p$-adic Galois representation $\Gamma_{\Q} \to \GL_2(\overline{\Q}_p)$ attached to the corresponding eigenform.

If $\psim$ is trivial then the corresponding theta series $\theta_\psi$  coincides with the Eisenstein series attached to  $\psi$ and $\psi\epsilon_K $, and the local structure of the eigencurve at such points has been extensively studied in \cite{BDPozzi}. 
 If $\psim(\gp)\neq1$ then  $\theta_\psi$ is a cuspform whose  $p$-stabilizations
\[\theta_\psi(z) - \psi(\bar{\gp}) \theta_\psi(pz) \text{ and  } \theta_\psi(z) - \psi(\gp) \theta_\psi(pz)\] 
belong each  to a unique Hida family  necessarily having  CM by $K$ (see \cite[Cor.~1.2]{bellaiche-dimitrov}).

We henceforth assume that $\psim$ is non-trivial but $\psim(\gp)=1$,  and we let $f$ denote  the unique $p$-stabilization of the cuspform $\theta_\psi$.

\subsection{Components with CM by K containing $f$}\label{cm-chars}
The reciprocity map  in  Class Field Theory, denoted by  $\mathrm{Art}$,  is normalized so that it sends uniformizers   to 
arithmetic Frobenii.  We let $\varepsilon_{\mathrm{cyc}}:\Gamma_\Q \twoheadrightarrow  \Z_p^{\times}$ be the $p$-adic cyclotomic character, $\omega_p$ be the Teichm\"{u}ller character and
$\varepsilon_p=\varepsilon_{\mathrm{cyc}} \omega_p^{-1} :\Gamma_\Q \twoheadrightarrow 1+p^\nu \Z_p$, where $\nu=2$ if $p=2$, and $\nu=1$ otherwise. We recall the universal cyclotomic character 
$\chi_p:\Gamma_\Q \to \bar\Q_p\lsem 1+p^\nu \Z_p\rsem^{\times} \simeq \varLambda^\times$ as well as the character
$\eta_{\mathbf{1}}\in\rH^1(\Q,\bar\Q_p)$ from  \cite[\S2.4]{BDPozzi}. 

As  $(p)=\gp\bar\gp$ splits in $K$, we have a short exact sequence of  groups 
\[1\to  \left(\cO_{K,\gp}^\times/\cO_{K}^\times\right)^{(p)} \to \cC\ell_K^{(p)}(\gp^{\infty})\to \cC\ell_K^{(p)}\to 1,\]
where the exponent $(\cdot)^{(p)}$ denotes the $p$-primary part of an abelian group. 
Letting  $\W_{\gp}$ denote the maximal torsion-free quotient   of $\cC \ell_K(\gp^{\infty})$, seen as the Galois group of  the unique $\gp$-ramified $\Z_p$-extension of $K$, for any prime to $p$ ideal $\gc$ of $K$, there exists an isomorphism 
\[\mathrm{Art}: \cC\ell_K^{(p)}(\gp^{\infty}\gc)_{/\tor} \xrightarrow{\sim} \W_{\gp}\]
sending $1+p^\nu\Z_p\subset \cO_{K,\gp}^\times = \Z_p^\times$ to the  inertia subgroup $\I_{\gp} \subset \W_{\gp} $ at $\gp$. 
The quotient $\W_{\gp}/\I_{\gp}$ is cyclic of order $p^h$, for some $h\in\Z_{\geqslant 1}$, and we let 
 $w_{\gp} \in \W_{\gp}$ be the  topological generator such that $w_{\gp}^{p^h}=\mathrm{Art}(1+p^\nu)\in\I_{\gp} $. 
 A $p$-adic avatar of a Hecke character of infinity type $(-1,0)$ is given by the  continuous character valued in the ring of integers $\cO$ of a sufficiently large  $p$-adic field
  \begin{align}\label{eq:epsilon} 
\varepsilon_{\gp} : \Gamma_K \twoheadrightarrow \W_{\gp} \to \cO^{\times},
\end{align} 
where the second map sends $w_{\gp}$ to   $z_{\gp}\in 1+\gm_\cO$ such that   
 $z_{\gp}^{p^h}=(1+p^\nu)^{-1}$.
 
The homomorphism $\eta_{\gp}:\Gamma_K \to \overline{\Q}_p$ from \S\ref{def-eta} factors through $\W_{\gp}$ and its restriction to 
$\I_{\gp}$ is given by $\log_p\circ \mathrm{Art}^{-1}$, hence $\eta_{\gp}=-\log_p\circ\varepsilon_{\gp}$. 
One can  check that  $\varepsilon_{\gp}\circ \mathrm{Ver}=\varepsilon_p$ and  
$\eta_{\gp}\circ \mathrm{Ver}= \eta_{\mathbf{1}}$, where $\mathrm{Ver}: \Gamma_{\Q}^\mathrm{ab}\to \Gamma_{K}^\mathrm{ab}$ is the transfer homomorphism.

 The completed strict  local ring $\varLambda_\gp$ of $\cO \lsem \W_{\gp} \rsem$
  at the augmentation ideal 
 $\ker\left(\cO \lsem \W_{\gp} \rsem \to \cO\right)$  is naturally isomorphic to $\varLambda$.
 Indeed, the inclusion   $\cO \lsem \I_{\gp} \rsem\hookrightarrow \cO \lsem \W_{\gp}\rsem$ can be identified  to the generically \'etale map  $\cO \lsem X\rsem\hookrightarrow\cO \lsem  W\rsem$ sending $X$ to $(1+W)^{p^h}-1$.
 The resulting map
\begin{align}\label{eq:Lambda-iso}
\varLambda\xrightarrow{\sim} \varLambda_\gp
\end{align}
is an isomorphism allowing  to consider  the   universal character
  $\chi_{\gp}:\Gamma_K \twoheadrightarrow \W_{\gp} \to \cO \lsem \W_{\gp} \rsem^{\times} \hookrightarrow \varLambda^{\times}_\gp$ as $\varLambda$-valued. It follows that $\chi_{\gp}\circ \mathrm{Ver}=\chi_p$ and moreover (compare with \cite[(25)]{BDPozzi}) one has 
  \begin{align}\label{eq:univ-coates-wiles}
  \chi_{\gp}\equiv 1 - \frac{\eta_{\gp}}{\log_p(1+p^{\nu})}X\pmod{X^2}.  
\end{align}
Similarly, we define the universal character $\chi_{\bar\gp}=\chi_\gp^{\tau}:\Gamma_K \to \varLambda_\gp^\times$ unramified away from $\bar\gp$.

  Recall the CM Hida family $\Theta_{\psi}$ whose 
 specialization  in weight   $k\in \Z_{\geq 1}$ is given by the classical theta series  $\theta_{\psi,k}$
corresponding to the Hecke character $\lambda_{\psi,k} : \cC\ell_K(\gp^{\infty}\gc_\psi)\to  
\cO \lsem \W_{\gp} \rsem^{\times} \to \overline{\Q}{}_p^{\times}$ of infinity type $(1-k,0)$
obtained as the composed map of  the algebra homomorphism  attached to the group homomorphism 
$\varepsilon_{\gp}^{k-1}:\W_{\gp} \rightarrow \overline{\Q}{}_p^{\times}$ and the $\psi$-projection
  $ \cC\ell_K(\gp^{\infty}\gc_\psi) \to  \cO \lsem \W_{\gp} \rsem^\times$ sending 
$(z,w)\in  \cC\ell_K(\gp^{\infty}\gc_\psi)_{\tor}\times \W_{\gp}=\cC\ell_K(\gp^{\infty}\gc_\psi)$ to $\psi(z,w)[w]$.

\begin{defn} \label{d:CM}
The CM by $K$ part $\cE^K$ of $\cE$ is  the closed locus where 
$\tau_{\cE}(g)$ vanishes for all $g\in \Gamma_{\Q}\backslash \Gamma_K$. We will say that an  
irreducible component  $\cZ$ of $\cE$  has CM by $K$, if $\cZ\subset \cE^K$. 
\end{defn}

As $\rho_f=\Ind^\Q_K \psi$ is irreducible, the localization of the pseudo-character $\tau_{\cE}$ at $f$ yields 
\begin{align}\label{eq:rhoT}
\rho_{\cT}:\Gamma_{\Q} \to \GL_2(\cT)
\end{align} 
 deforming $\rho_f$ and such that $ \tr(\rho_{\cT})(\Frob_\ell)=T_\ell$ for all primes  $\ell \nmid Np$ (see \cite{nyssen}).

\begin{prop}\label{Tcmgor} Let $\cT^K$ be the completed local ring of $\cE^K$ at $f$. There is an isomorphism of local $\varLambda$-algebras $\cT^K\xrightarrow{\sim}\varLambda_\gp \times_{\overline{\Q}_p} \varLambda_{\gp}$. Moreover, the   
Galois representations of  the two components $\Theta_{\psi}$ and $\Theta_{\psi^\tau}$ of $\cE^K$ containing $f$ are   $\Ind^\Q_K (\psi\chi_{\gp})$ and $\Ind^\Q_K (\psi^\tau \chi_{\gp})$, respectively.   
\end{prop}

\begin{proof} 
Let $\cZ$  be an irreducible component of $\cE$ having CM by $K$ and containing $f$, and denote by 
 $A=\cO_{\cZ,f}^\wedge$ its  completed local $\varLambda$-algebra. 
As $\rho_f$ is irreducible, the pseudo-character $\tau_A:\Gamma_{\Q}\to A$ comes from a representation 
$\rho_A:\Gamma_{\Q}\to \GL_2(A)$, which in view of  Definition~\ref{d:CM} (and the uniqueness of the lifting) is such that 
$\rho_A\simeq \rho_A\otimes \epsilon_K$.  
By slightly adapting the arguments of \cite[Lem.~3.2]{DIH} to local rings in the category $\cC$ we obtain that
$\rho_A=\Ind^\Q_K \psi_A$, where $\psi_A:\Gamma_K \to A^{\times}$ is a character lifting, say, $\psi$ (otherwise it lifts $\psi^\tau$ and the argument is the same). By $\gp$-ordinarity either the character $\psi_A\psi^{-1}: \Gamma_K\to 1+ \gm_A$ or its $\tau$-conjugate factors through $\W_{\gp}$. Moreover one has $\chi_p=\det(\Ind^\Q_K(\psi_A\psi^{-1}))\cdot \epsilon_K=(\psi_A\psi^{-1})\circ\mathrm{Ver}$, hence 
$\psi_A\psi^{-1}=\chi_{\gp}$ or $\chi_{\bar\gp}$. 
So far we have shown that there is an injection  
\[\cT^K\hookrightarrow \varLambda_{\gp} \times_{\overline{\Q}_p} \varLambda_{\gp}, \,\,
T_\ell\mapsto \left((\chi_{\gp}\psi+\chi_{\bar\gp}\psi^\tau)(\lambda),(\chi_{\gp}\psi^\tau+\chi_{\bar\gp}\psi)(\lambda)\right), \]
where $\ell\nmid Np$ is a prime splitting in $K$ as  $(\ell)=\lambda\bar\lambda$. 
To show that it is an isomorphism, it suffices to see that $\chi_{\gp}\psi+\chi_{\bar\gp}\psi^\tau$ and 
$\chi_{\gp}\psi^\tau+\chi_{\bar\gp}\psi$ are not congruent mod $X^2$, which in view of  \eqref{eq:univ-coates-wiles}
amounts to show that $\eta_{\gp}\psi+\eta_{\bar\gp}\psi^\tau\ne \eta_{\gp}\psi^\tau+\eta_{\bar\gp}\psi$. This is clear as $\psim \ne \mathbf{1}$.\end{proof}

 \subsection{Existence of a component without CM by K containing  \texorpdfstring{$f$}{f}}\label{nonCMsection}
The full eigencurve $\cE_{\full}$ is locally generated over $\cE$ by the Hecke operators $U_q$ for  $q \mid N$. The resulting 
morphism $\cE_{\full} \to \cE$  is locally finite, surjective, compatible with all the structures, and not an  isomorphism when $N>1$. 
Hida's $p$-ordinary Hecke algebra of tame level $N$ is a canonical integral model for the ordinary locus of $\cE_{\full}$ which is 
the admissible open of $\cE_{\full}$ defined by  $|U_p|_p=1$, and whose irreducible components correspond to Hida families of tame level $N$. 
In this language Proposition~\ref{Tcmgor} asserts that $\Theta_{\psi}$ and $\Theta_{\psi^{\tau}}$ are the only 
Hida families containing $f$ and having CM by $K$.  

Let $\cU$ be an affinoid subdomain of $\cW$ containing $\kappa(f)$ and small enough to ensure the existence of the finite type, projective
$\cO(\cU)$-module $S^{\dagger,0}_\cU(N)$ of $p$-ordinary families of cuspidal overconvergent modular forms of tame level $N$.    
By construction, there exists an affinoid neighbourhood  $\cV$ of $f$ in $\cE_{\full}$ such that $\cU=\kappa(\cV)$, and $\cO(\cV)$ is the finite $\cO(\cU)$-subalgebra of $\mathrm{End}_{\cO(\cU)}(S^{\dagger,0}_\cU(N))$ generated by  $T_\ell$ for $\ell \nmid Np$ and $U_q$ for $q \mid Np$ (see \cite[\S7.1]{coleman-mazur}). Taking the first  coefficient in $q$-expansions at the cusp $\infty$
yields a natural isomorphism  of  $\cO(\cU)$-modules 
(see Hida  \cite[\S2]{hida2} and Coleman \cite[Prop.~B.5.6]{coleman}):\begin{align} \label{duality}
S^{\dagger,0}_\cU(N)\xrightarrow{\sim} \Hom_{\cO(\cU)-\mathrm{mod}}(\cO(\cV), \cO(\cU)), \,\,\, 
\mathcal{G}\mapsto (T\mapsto a_1(T\cdot \mathcal{G})). 
\end{align}

By   \cite[Prop.~7.1]{bellaiche-dimitrov}, $\cE$ and $\cE_{\full}$ are locally isomorphic at $f$. 
In particular, $U_q\in \cT $ for all  $q \mid N$. 
It follows then from \eqref{duality}  (see also \cite[Proof of Prop.~1.1]{DLR1}) that \begin{align}\label{hidadualitywt1}
\Hom_{\overline{\Q}_p-\mathrm{mod}}(\cT/\gm_{\varLambda}\cdot\cT,\overline{\Q}_p) \simeq S_{1}^{\dagger}(N)\lsem f\rsem. \end{align} 
where $\cT/\gm_{\varLambda}\cdot\cT$ is the local ring of the fiber $\kappa^{-1}(\kappa(f))$ at $f$,  and $S_{1}^{\dagger}(N)\lsem f\rsem$ is the generalized eigenspace attached to $f$ inside the space of weight $1$, level $N$, ordinary  $p$-adic modular forms.

Let  $\cT_{\spl}$  be the sub-algebra  of $\cT$ generated over $\varLambda$ by  $T_\ell$ for
 $\ell\nmid Np$  split in $K$.

\begin{thm}\label{existencenon-cm}
\begin{enumerate}[wide]
\item The operator $U_p$ does not belong to  $\cT_{\spl}$.
\item There exists an irreducible component $\cZ$ of $\cE^{\ord}$  containing $f$ and without CM by $K$.
\end{enumerate}
\end{thm}

\begin{proof} (i) A direct computation shows that $(U_p-\psi(\gp))(\theta_\psi)=\psi(\gp) f$. Moreover, $\theta_\psi$ and $f$ sharing
the same Hecke eigenvalues away from $U_p$, the  space spanned by those two forms is contained in  $S_{1}^{\dagger}(N)\lsem f\rsem$. The above computation also shows that  the action of $U_p$ on that space  is {\it not} semi-simple, while   $\cT_{\spl}$ acts on it semi-simply (diagonally).

(ii) Assume that any component of $\cE$ containing $f$ has CM by $K$, {\it i.e.}, $ \cT= \cT^K$. 
Then, by Proposition~\ref{Tcmgor},  $\cT=\varLambda_\gp \times_{\overline{\Q}_p} \varLambda_{\gp}$ and $\cT/\gm_{\varLambda}\cdot\cT \simeq \overline{\Q}_p[X]/(X^2)$,
 hence by \eqref{hidadualitywt1} we have 
\[\dim_{\overline{\Q}_p} S_{1}^{\dagger}(N)\lsem f\rsem= \dim_{\overline{\Q}_p} (\cT/\gm_{\varLambda}\cdot\cT)=2,\]
in particular $\{ f, \theta_\psi \}$ is a basis of $S_{1}^{\dagger}(N)\lsem f\rsem$. 
On the other hand, Proposition~\ref{Tcmgor} also implies (under the identification $\varLambda_\gp=\varLambda$) that the image of $U_p$ in $\cT^K$ equals $((\psi \chi_{\bar\gp})(\gp),(\psi^{\tau} \chi_{\bar\gp})(\gp))$ which belongs to  $\varLambda^\times$ embedded diagonally in  $\cT^K$ as $\psim(\gp)=1$. Hence $(U_p-\psi(\gp))\in \gm_{\varLambda}$ and therefore $U_p$ acts as $\psi(\gp)$  on 
$S_{1}^{\dagger}(N)\lsem f\rsem$,  yielding a contradiction with the fact, established in (i), that its action on that space is not semi-simple.\end{proof}

\begin{rem} Theorem~\ref{existencenon-cm}(ii) not only answers a question raised in \cite[\S7.4(1)]{dim-ghate} about the existence of a Hida family without CM by $K$ containing a weight $1$ form having CM by $K$, but it further asserts that this is systematically the case:  any CM weight $1$ cuspform irregular at $p$ belongs to such a family. 
\end{rem} 

 \begin{defn}
Let  $\cE^{\perp}$ be the closed  analytic subspace of $\cE$, union of irreducible components having no  CM by $K$.  
\end{defn}

Theorem~\ref{existencenon-cm} implies that $f\in \cE^{\perp}$.  As well known  from Hida theory, $\cT$ is equidimensional of dimension $1$, 
hence so is  the completed local ring 
$\cT^{\perp}$  of $\cE^{\perp}$ at $f$, which can also  be seen as the largest non-CM by $K$ quotient of $\cT$. 
We will use this modular input to prove that the Krull dimension of $\cR^\perp$  is at least $1$ 
(recall that by Theorem \ref{cuspidal-tangent} its dimension is at most $1$). 

The specialization of  \eqref{eq:rhoT} by the surjective homomorphism $\pi^\perp:\cT\to \cT^\perp$,   yields a deformation 
$\rho_{\cT}^{\perp}: \Gamma_{\Q} \to \GL_2(\cT^{\perp})$ of $\rho_f$. Next, we show that after conjugating
$\rho_{\cT^{\perp}|\Gamma_K}$ by a matrix in the total field of fractions $Q(\cT^{\perp})$, we  obtain $\cT^\perp$-valued deformations of 
both $\rho$ and  $\rho'$.

\begin{prop} \label{rhott} Denote $\chi_\cT^{\perp}$  the unramified character of $\Gamma_{K_{\gp}} \subset \Gamma_K$  sending  $\Frob_{\gp}$ to $U_p$. There exists a deformation $\rho_{\cT^\perp}: \Gamma_K  \to \GL_2(\cT^{\perp})$ of $\rho$ such that  $\rho_{\cT^{\perp} \mid \Gamma_{K_{\gp}}}=\left(\begin{smallmatrix} * & * \\ 0 & \chi_\cT^{\perp} \end{smallmatrix}\right)$, and $\tr \rho_{\cT^\perp}=\tr(\rho^{\perp}_{\cT \mid \Gamma_K})$. Similarly, there exists a deformation $\rho'_{\cT^\perp}:\Gamma_K  \to \GL_2(\cT^{\perp})$ of ${\rho'}  $ such that 
$\tr(\rho'_{\cT^\perp})= \tr( \rho_{\cT^\perp})$ and $\rho'{}_{\!\!\cT^{\perp} \mid \Gamma_{K_{\gp}}}=\left(\begin{smallmatrix} * & * \\ 0 &  \chi_\cT^{\perp} \end{smallmatrix}\right)$, hence there is a $\varLambda$-algebra homomorphism  $\varphi^{\perp}:\cR^{\perp} \to \cT^{\perp}$.
\end{prop}

\begin{proof} We first observe that for any localization $L$ of $\cT^\perp$ at a minimal prime, 
the resulting representation $\rho_{L}: \Gamma_{\Q} \to \GL_2(L)$ remains irreducible when restricted to $\Gamma_K$. 
In fact, if  $\Hom_{\Gamma_K}(\rho_{L},  \psi_L)\ne \{0\}$ for some character $\psi_L: \Gamma_{K} \to L^\times$, the Frobenius Reciprocity Theorem 
would imply that  $\Hom_{\Gamma_{\Q}}(\rho_{L}, \Ind^\Q_K \psi_L)\ne \{0 \}$, thus contradicting the assumption that the above minimal prime does not correspond to a component having  CM by $K$. 
 Proposition~\ref{restinj} then allows us to  apply \cite[Cor.~2]{bellaiche-chenevier} to find an adapted basis $(e_1^{\perp},e_2^{\perp})$ (resp. $(e_1^{\perp'},e_2^{\perp'})$) of $Q(\cT^{\perp})^2$ such that the representation $\rho_{Q(\cT^{\perp})}: \Gamma_K  \to \GL_2(Q(\cT^{\perp}))$ takes values in $\GL_2(\cT^{\perp})$ and reduces to $\rho$ (resp. $\rho' $) modulo the maximal ideal of $\cT^{\perp}$.  The resulting representations, denoted  $\rho_{\cT^\perp}$ and $\rho'_{\cT^\perp}$, share the same trace and determinant.

For the ordinariness statement at $\gp$, we write an exact sequence of $\cT^{\perp}[\Gamma_{K_{\gp}}]$-modules as in  \cite[(16)]{bellaiche-dimitrov} and adapt the argument as follows: the fact that    $\rho_{\mid \I_{\gp}}$ (resp. $\rho'_{\mid \I_{\gp}}$) has an infinite image and  admits a unique $\I_{\gp}$-stable line,   shows that the last term of that exact sequence is a monogenic $\cT^{\perp}$-module, hence it is free, as it is generically free of rank $1$ and $\cT^{\perp}$ is reduced.\end{proof}

As $\cR^{\perp}$ is  generated over $\varLambda$ by the trace of its universal deformation (see 
Lemma~\ref{lem-surj}), 
the Chebotarev Density Theorem and Proposition~\ref{rhott} imply that the image  of the $\varLambda$-algebra homomorphism 
$\varphi^{\perp}:\cR^{\perp} \to \cT^{\perp}$ equals   the image  $\cT^\perp_{\spl}$ of $\cT_{\spl}$ in $\cT^\perp$. 
 
\begin{thm}\label{isom-cusp}  The $\varLambda$-algebra $\cR^{\perp}$ is a discrete valuation ring isomorphic to $\cT^{\perp}_{\spl}$  and  $\rho_{\cR^\perp}$ is irreducible.  The structural 
homomorphism $\varLambda\to \cR^{\perp}$ is \'etale if and only if  $\cLm(\psim)+\cLm(\psim^\tau)\ne 0$.
\end{thm}

\begin{proof}
By Theorem~\ref{cuspidal-tangent} the Zariski tangent space $t^{\perp}$ of $\cR^{\perp}$ has dimension $1$. 
As $ \cT^{\perp}_{\spl}$ is finite and flat over  $\varLambda$, it follows that its Krull dimension is $1$ and that the surjective 
$\varLambda$-algebra homomorphism  $\cR^{\perp} \twoheadrightarrow\cT^{\perp}_{\spl}$ is an isomorphism of 
regular local rings of dimension $1$, {\it i.e.} discrete valuation rings.  The irreducibility of $\rho_{\cR^\perp}$ then follows from the proof of  Proposition~\ref{rhott}.
 The last claim follows directly from 
Theorem~\ref{cuspidal-tangent} where one gives a necessary and sufficient condition for  the vanishing of the relative tangent space $t^{\perp}_0$ of $\cR^{\perp}$ over $\varLambda$.\end{proof}

\subsection{The reducibility ideal of \texorpdfstring{$\cR^\perp$}{}} \label{reducibility-ideal}

We recall that $\rho_{\cR^\perp}$ (resp. $\rho'_{\cR^\perp}$) is the $\gp$-ordinary deformation of $\rho$ (resp. $\rho' $) obtained by  functoriality from the natural surjection $\cR_{\rho}^{\ord} \twoheadrightarrow \cR^{\perp}$ (resp. $\cR_{\rho' }^{\ord} \twoheadrightarrow \cR^{\perp}$) from Lemma~\ref{lem-surj}. In order to determine $\cT^\perp$ we need to study the possible extensions of 
$\tr(\rho_{\cR^\perp})=\tr(\rho'_{\cR^\perp})$ to $\Gamma_{\Q}$. For this we first observe that the reducibility ideal (see \cite[\S1.5]{bellaiche-chenevier-book}) of the pseudo-character $\tr(\rho_{\cR^\perp})=\tr(\rho'_{\cR^\perp})$ is non-zero, as   $\rho_{\cR^\perp}$ is  irreducible by Theorem~\ref{isom-cusp}. We recall that $\cR^{\perp}$ is a discrete valuation ring and is flat over $\varLambda$. 

\begin{defn}\label{d:ideal-red}
\begin{enumerate}[wide]
\item Let $e\geqslant 1$ be  the ramification index of $\cR^{\perp}$ over $\varLambda$, and let $Y$ denote a uniformizer of $\cR^{\perp}$, so that  $\cR^{\perp}= \overline{\Q}_p \lsem Y\rsem \simeq \overline{\Q}_p \lsem X^{1/e}\rsem$.  
\item Let $r\geqslant 1$ be the valuation of the reducibility ideal  of $\tr(\rho_{\cR^\perp})$. 
\end{enumerate}
\end{defn}

We state for later use a lemma whose elementary proof uses the fact that
 $\rho$ has non-isomorphic Jordan-H\"older factors, and that $\cR^{\perp}/(Y^{r})$ is a complete local ring, hence Henselian. 
\begin{lemma}\label{unicity-decomp}
For  $m\leqslant r$, $\tr(\rho_{\cR^\perp})\mod{Y^m}$ can be uniquely written as sum of two characters. 
\end{lemma}

We recall from \S\ref{deformation} the element $\gamma_0\in \Gamma_K$   such that $\psim(\gamma_0) \ne 1$ and $\eta(\gamma_0)=0$ and we write  $\rho_{\cR^\perp}=\left(\begin{smallmatrix} a & b \\ c & d \end{smallmatrix}\right)$ in a basis $(e_1,e_2)$ where $\rho_{\cR^\perp}(\gamma_0)=\left(\begin{smallmatrix} * & 0 \\ 0 & * \end{smallmatrix}\right)$. 

It follows from Proposition~\ref{prop:td} that $r\geqslant 2$.   
Much of what  follows  will depend on $r$, so we first give a numerical criterion for its  lowest (conjecturally  only possible) value.

\begin{lemma}\label{reducn=2} One has $r = 2$ if and only if  $\cLm(\psim)\cdot\cLm(\psim^\tau) \ne 0$.
\end{lemma}
\begin{proof} By \eqref{eq:infinitesimal} one has 
\begin{align}\label{eq:modY3}
\rho_{\cR^\perp} \equiv \left(\begin{smallmatrix} \psi + \psi (\beta \eta_{\gp} + \alpha \eta_{\bar\gp}) Y+  Y^2\cdot * & \psi^{\tau}\eta +  Y\cdot *  \\ \psi c  Y^2 &  \psi^{\tau} + \psi^{\tau} (\alpha \eta_{\gp} + \beta \eta_{\bar\gp})   Y+Y^2 \cdot * \end{smallmatrix}\right) \pmod{Y^3},
\end{align}
 where  $c\in \mathrm{Z}^1(\Gamma_K, \psim^\tau)$ is such that $c(\gamma_0)=0$. 
 By \eqref{U_pL}, after rescaling the uniformizer $Y$ by an element of $\overline{\Q}{}_p^{\times}$, we may and do assume that $\alpha=\cLm(\psim)$  and  $\beta=\cLm(\psim^\tau)$. 
 By Proposition~\ref{extensionpseudo}
one has $r \geqslant 3$ if and only if $c=0$ which in view of the injectivity of the restriction 
$\rH^1(K, \psim^{\tau}) \rightarrow \rH^1(K_{\gp}, \overline{\Q}_p)$ (see Proposition~\ref{restinj}) is also equivalent to 
$c_{|\Gamma_{K_{\gp}}}=0$. We will now use the $\gp$-ordinarity of  $\rho_{\cR^\perp}$ and show that the vanishing of
$c_{|\Gamma_{K_{\gp}}}$ is equivalent to $\cLm(\psim)\cdot\cLm(\psim^\tau)=0$. 
 
By the proof of Proposition~\ref{prop:td}  the $\gp$-ordinary filtration modulo $(Y^3)$ is generated by a vector of the form 
 $e_1 + (\cLm(\psim) Y  + Y^2\cdot *) e_2 $, and a direct computation using \eqref{eq:modY3} shows that 
\[c(g)= \cLm(\psim)\cdot\left(\cLm(\psim)\cdot (\eta+\eta_{\bar\gp} -\eta_{\gp})(g)-\cLm(\psim^\tau)\cdot (\eta_{\bar\gp}-\eta_{\gp})(g)  \right)= \cLm(\psim)\cdot\cLm(\psim^\tau)\cdot \eta'(g)\]
for all $g \in \Gamma_{K_{\gp}}$, which  completes the proof, as $\eta'_{|\Gamma_{K_{\gp}}}\ne 0$.\end{proof}

Let   $\iota$ be the automorphism sending $Z$ to $-Z$ of the local $\varLambda$-algebra 
\begin{align}\label{eq:R-tilde}
\widetilde{\cR}^{\perp}=\cR^{\perp}[Z]/(Z^2-Y^r)= \begin{cases}
\cR^{\perp} \times_{\overline{\Q}_p} \cR^{\perp}, & r=2 \\
 \varLambda \times_{\varLambda/(X^{\frac{r}{2}})} \varLambda, & r \geq 4 \text{ even,  } \\
  \varLambda[Z]/(Z^2-X^{r}), & r \geq 3 \text{ odd}. \end{cases}
  \end{align}

Note that  the latter ring $\widetilde{\cR}^{\perp}\simeq \overline{\Q}_p\lsem X \rsem [X^{\frac{r}{2}}]$ is not normal.

We write 
 $\rho'_{\cR^\perp}=\left(\begin{smallmatrix} a' & b' \\ c' & d' \end{smallmatrix}\right)$ in a basis  where $\rho'_{\cR^\perp}(\gamma_0)=\left(\begin{smallmatrix} * & 0 \\ 0 & * \end{smallmatrix}\right)$.

\begin{prop}\label{extensionpseudo}  The reducibility ideal equals $(c(g); g \in \Gamma_K)=(c'(g); g \in \Gamma_K)=(Y^r)$. Moreover  
$ \widetilde\rho^\perp_{\cR}= \left(\begin{smallmatrix} Z& 0 \\ 0& 1 \end{smallmatrix}\right)  \rho_{\cR^\perp} 
\left(\begin{smallmatrix} Z^{-1} & 0 \\ 0& 1 \end{smallmatrix}\right)$ 
extends to a representation $\widetilde\rho^\perp_{\cR}:\Gamma_{\Q}\to \GL_2(\widetilde\cR^{\perp})$  
reducing  to $\rho_f$ modulo the maximal ideal of $\widetilde\cR^{\perp}$. Finally 
$\iota(\widetilde\rho^{\perp}_{\cR}) =\left(\begin{smallmatrix} -1 & 0 \\ 0& 1 \end{smallmatrix}\right)
(\widetilde\rho^{\perp}_{\cR} \otimes \epsilon_K)
\left(\begin{smallmatrix} -1 & 0 \\ 0& 1 \end{smallmatrix}\right)$. 
\end{prop}

\begin{proof} Recall that $(Y^r)$ is the smallest ideal modulo which  $\tr(\rho_{\cR^\perp})=\tr(\rho'_{\cR^\perp})$ can be written as a sum of two characters.  A direct computation shows that $(Y^r)=(c(g)b(g'); g,g' \in \Gamma_K)=(c'(g)b'(g'); g,g' \in \Gamma_K)$ implying the first claim, as   $(b(g'); g' \in \Gamma_K)=(b'(g');g' \in \Gamma_K)=\cR^{\perp}$. 
 
 As  $\dim \rH^1(K,\psim)=\dim \rH^1(K,\psim^\tau)=1$, there exist upper triangular matrices $P_0$ and $P_1$ in $\GL_2(\overline{\Q}_p)$
 such that the  pair of $\cR^\perp$-valued $\Gamma_K$-representations 
\[\left(P_0
\left(\begin{smallmatrix} 0 & 1 \\ Y^{r}& 0 \end{smallmatrix}\right)  \rho^\tau_{\cR^\perp} \left(\begin{smallmatrix} 0 & Y^{-r} \\ 1 & 0 \end{smallmatrix}\right)
P_0^{-1}, P_1
\left(\begin{smallmatrix} 0 & 1 \\ Y^{r}& 0 \end{smallmatrix}\right)  \rho'{}^\tau_{\!\!\cR^\perp} \left(\begin{smallmatrix} 0 & Y^{-r} \\ 1 & 0 \end{smallmatrix}\right)  
P_1^{-1}\right)\]
 reduces to $(\rho,\rho')$ modulo $\gm_{\cR^{\perp}}$. We will now check that it  defines the same  $\cR^\perp$-point of $\cD^{\perp}$ as  $(\rho_{\cR^\perp}, \rho'_{\cR^\perp})$.  In fact, all representations are $\cR^\perp$-valued and share the same $\tau$-invariant trace and determinant, and 
  by  Lemma~\ref{lem-surj} they are all $\gp$-ordinary (with same unramified character). Hence the two pairs are strictly conjugated, in particular  there 
  exists $P\in \GL_2(\cR^\perp)$  such that for all $g\in \Gamma_K$ we have
  \begin{align}\label{eq:tau-conj}
 \left(\begin{smallmatrix} d^\tau(g) & Y^{-r}c^\tau(g) \\ Y^{r} b^\tau(g) & a^\tau(g) \end{smallmatrix}\right) =
\left(\begin{smallmatrix} 0 & 1 \\ Y^{r}& 0 \end{smallmatrix}\right)
 \rho_{\cR^\perp}(\tau g\tau) \left(\begin{smallmatrix} 0 & Y^{-r} \\ 1 & 0 \end{smallmatrix}\right)=
P \rho_{\cR^\perp}( g)P^{-1} =P \left(\begin{smallmatrix} a(g) & b(g) \\ c(g) & d(g) \end{smallmatrix}\right) P^{-1}.
\end{align}
As $d^\tau\not\equiv d \pmod{Y}$, reducing the above equation modulo $Y^r$ yields $P\in 
\left(\begin{smallmatrix} * & * \\ 0& * \end{smallmatrix}\right) \pmod{Y^r}$. 
 Letting    $M=\left(\begin{smallmatrix} 0 & 1 \\ Z& 0 \end{smallmatrix}\right) P \left(\begin{smallmatrix}Z^{-1} & 0 \\ 0& 1 \end{smallmatrix}\right)\in \GL_2(\widetilde\cR^{\perp})$ the equation   \eqref{eq:tau-conj} can be rewritten as 
 \begin{align}\label{eq:tau-conj-bis}
\widetilde\rho^\perp_{\cR}(\tau g\tau)= M \widetilde\rho^\perp_{\cR}(g) M^{-1},
\end{align}
where the representation $ \widetilde\rho^\perp_{\cR}= \left(\begin{smallmatrix} Z & 0 \\ 0& 1 \end{smallmatrix}\right)  \rho_{\cR^\perp} 
\left(\begin{smallmatrix} Z^{-1} & 0 \\ 0& 1 \end{smallmatrix}\right):\Gamma_K\to \GL_2(\widetilde\cR^{\perp})$  is 
 generically absolutely  irreducible  (as $\rho_{\cR^\perp}$ is), while  reduces  to $\psi\oplus\psi^\tau$ modulo the maximal  ideal. 
As   $\tau^2=1$, Schur's lemma applied to  a generic point of $\widetilde\cR^{\perp}$ implies that 
 $M^2$ is scalar. Reducing \eqref{eq:tau-conj-bis} modulo the maximal ideal of $\widetilde\cR^{\perp}$ shows that $M$ is not scalar, hence necessarily $\tr(M)=0$. After rescaling $M$ one can assume that $\det(M)=-1$. Then  $M^2=1$ and letting $\widetilde\rho^\perp_{\cR}(\tau)=M$ (or $-M$) yields the desired extension.  Note  that $ \rho_f$ is the unique extension of 
$\psi\oplus\psi^\tau$ to $\Gamma_{\Q}$. 

The last point is a direct computation.\end{proof}

Denote by  $\rho^{\perp}_{\cR}$ the push-forward of $\widetilde\rho^\perp_{\cR}$ along the projection 
$\widetilde\cR^{\perp} \twoheadrightarrow \cR^{\perp}$ sending  $Z$ to  $Y^{\frac{r}{2}}$, when $r$ is even, and let
$\rho^{\perp}_{\cR} =\widetilde\rho^{\perp}_{\cR}$, when $r$ is odd. 
By the Frobenius Reciprocity Theorem (applied over a finite extension of the fraction field of the
integral domain $\cR^{\perp}$), $\rho^{\perp}_{\cR}$  and $\rho^{\perp}_{\cR} \otimes \epsilon_K$ are the only possible 
extensions of $ \rho_{\cR^\perp}$ to $\Gamma_{\Q}$.

\begin{cor}\label{coefficientsring}
The   pseudo-characters $\tr(\rho^{\perp}_{\cR})$  and $\tr(\rho^{\perp}_{\cR} \cdot \epsilon_K)$ are the only possible 
extensions of $\tr(\rho_{\cR^\perp})$ to $\Gamma_{\Q}$. 
Moreover  $\tr(\widetilde\rho^{\perp}_{\cR}) \cdot \epsilon_K=\iota\circ \tr(\widetilde\rho^{\perp}_{\cR})$.
\end{cor}

 \subsection{Components without CM by K containing $f$}\label{wihtoutCMgenericpoints}

Theorem~\ref{existencenon-cm} gives a  Hida family $\cF$ containing $f$ and having no CM by $K$ . 
As $f\otimes \epsilon_K=f$, the  twist $\cF \otimes \epsilon_K$ of $\cF$ is another such family. 

\begin{prop}\label{tildeRtrace}
The local ring $\widetilde\cR^{\perp}$ is topologically generated over $\varLambda$ by $\widetilde\tau_{\cR}^{\perp}=\tr (\widetilde\rho^\perp_{\cR}(\Gamma_{\Q}))$.
\end{prop}
\begin{proof} Let $A$ be the subring of $\widetilde\cR^{\perp}$ topologically generated  over $\varLambda$ by  $\widetilde\tau_{\cR}^{\perp}(\Gamma_{\Q})$. As $A\supset \cR^{\perp}$ (see Theorem~\ref{isom-cusp}), it suffices to show that $A$ contains the local parameter $Z$. By  \cite[Thm.~1]{nyssen}  there exists a deformation $\rho_A=\left(\begin{smallmatrix} a &b\\ c&d \end{smallmatrix}\right):\Gamma_{\Q} \to \GL_2(A)$ of $\rho_f$ such that $\tr(\rho_A)=\widetilde\tau_{\cR}^{\perp}$ and $\rho_A(\gamma_0)$ is diagonal. 
 Then the reducibility ideal of $\widetilde\tau^{\perp}_{\cR\mid \Gamma_K}=\tr \rho_{\cR^\perp}$ equals  $(Y^r)=(b(g)c(g'); g,g'\in \Gamma_K)$. We first prove that the ideals $(b(g); g\in \Gamma_K)$ and $(c(g); g\in \Gamma_K)$ of $A$ are equal by showing the double inclusion. To show $(b(g); g\in \Gamma_K)\subset (c(g); g\in \Gamma_K)$ one reduces 
the equality  $\rho_A(\tau g \tau)= \rho_A(\tau)\rho_A(g)\rho_A(\tau)$ for a given $g\in  \Gamma_K$  modulo $(c(g); g\in \Gamma_K)$ 
and uses the fact that $\rho_A(g)$ and $\rho_A(\tau g \tau)$ are upper triangular and that  $\rho_A(\gamma_0)$ is diagonal.

Assume first that  $r$ is odd. In that case $\widetilde\cR^{\perp}$ is an integral domain, which is  finite, flat and ramified over $\cR^{\perp}$.
It suffices then to show that $A$ contains an element of $Y^{\frac{r}{2}}\cdot\overline{\Q}_p \lsem Y^{\frac{1}{2}}\rsem^\times$, and 
any  $b(g_0)\in A$ generating $(b(g); g\in \Gamma_K)=(Y^{\frac{r}{2}})$ would do.

Assume next that $r$ is even. In that case $\widetilde\cR^{\perp} \simeq \cR^{\perp} \times_{\overline{\Q}_p[Y]/(Y^{\frac{r}{2}})} \cR^{\perp}$, and $A \subset \widetilde\cR^{\perp} $ surjects on both  components $\cR^{\perp}$. Hence $A=\cR^{\perp} \times_{\overline{\Q}_p[Y]/(Y^{s})} \cR^{\perp}$ with $s \geq \frac{r}{2}$, and by Proposition~\ref{extensionpseudo} one has
\[\left(\begin{smallmatrix} -1 & 0 \\ 0& 1 \end{smallmatrix}\right)(\widetilde\rho^{\perp}_{\cR} \otimes \epsilon_K)\left(\begin{smallmatrix} -1 & 0 \\ 0& 1 \end{smallmatrix}\right) = \iota(\widetilde\rho^{\perp}_{\cR}) \equiv \widetilde\rho^{\perp}_{\cR}    \pmod{Y^s},\]
where $\iota$ is the involution of $\widetilde\cR^{\perp}$ exchanging its two components.   Writing $b(g)=(b_1(g),b_2(g))\in A$,  there exists $g_0\in \Gamma_K$ such that $b_1(g_0)=-b_2(g_0)\in Y^{\frac{r}{2}}\cdot\overline{\Q}_p \lsem Y\rsem^\times$. The above congruence implies that $-b_1(g_0)\equiv b_1(g_0)\pmod{Y^s}$, hence $s=\frac{r}{2}$.\end{proof}

As $\cT^{\perp}$ carries a pseudo-character of $\Gamma_{\Q}$, extending  the $\cR^\perp\simeq\cT^{\perp}_{\spl}$-valued
pseudo-character of $\Gamma_K$, Corollary~\ref{coefficientsring} implies that given any  minimal prime ideal $\gp$ of $\cT^\perp$ the resulting pseudo-character  $\Gamma_{\Q}\to \cT^\perp/\gp$ factors through $\tr(\widetilde\rho^{\perp}_{\cR})$. 
Together with  Proposition~\ref{tildeRtrace} this implies the existence of a  $\varLambda$-algebra homomorphism  $\widetilde\varphi^{\perp}:  \widetilde\cR^{\perp}\to \cT^{\perp}$ fitting in the commutative diagram
\[
\xymatrix{
\Gamma_K  \ar@{^{(}->}[d] \ar^{\tau_{\cR}^{\perp}}[r]  &  \cR^\perp   \ar@{^{(}->}[d] \ar^{\sim}[r]  &   \cT^{\perp}_{\spl} \ar@{^{(}->}[d]   \\
\Gamma_{\Q}  \ar^{\widetilde\tau_{\cR}^{\perp}}[r]  &   \widetilde\cR^{\perp} \ar^{\widetilde\varphi^{\perp}}[r]  &   \cT^{\perp}.
}\]
As by definition $\widetilde\varphi^{\perp}\circ \widetilde\tau_{\cR}^{\perp}$ is surjective, so is $\widetilde\varphi^{\perp}$. 
Both algebras being reduced, it suffices to check injectivity on generic points, which holds as 
$\widetilde\cR^{\perp}$ is a free $\cR^\perp$-module of rank $2$, while $\cT^{\perp}$ has rank $\geqslant 2$ over $\cT^{\perp}_{\spl}$.  As a consequence we  determine the $\varLambda$-algebras structure of $\cT^{\perp}$.

\begin{thm}\label{geomCdag}
There is an isomorphism of  $\varLambda$-algebras $\widetilde\varphi^{\perp}: \widetilde\cR^{\perp}\xrightarrow{\sim}  \cT^{\perp}$
extending $\varphi^{\perp}: \cR^{\perp}\xrightarrow{\sim}  \cT_{\spl}^{\perp}$ and 
under which  $\iota$ is sent to the involution of $\cT^{\perp}$ mapping $T_\ell$ to $\epsilon_K(\ell) T_\ell$, for $\ell\nmid Np$. 
\end{thm}

In summary, we have shown that $f$ belongs to exactly $3$ or $4$  irreducible
components of $\cE$ depending on the parity of $r$, corresponding to the Hida families $\Theta_{\psi}$, $\Theta_{\psi^{\tau}}$, $\cF$ and  $\cF \otimes \epsilon_K $ containing $f$ (the last two being Galois conjugates if $r$ is odd).

\section{Geometry of the  eigencurve and congruence ideals} \label{section4}

\subsection{Iwasawa cohomology and the CM ideal}\label{CM-Ideal}

Let $K_{\infty}^-$ be the anti-cyclotomic $\Z_p$-extension of $K$ and
\[\chim=\frac{\chi_{\gp}}{\chi_{\gp}^\tau} :\Gamma_K   \twoheadrightarrow \Gal(K_{\infty}^-/K) \to  \varLambda^{\times}_\gp\]
be the universal $p$-adic anti-cyclotomic character. By \eqref{eq:univ-coates-wiles}, it satisfies the congruence 
 \begin{align}\label{eq:univ-coates-wiles2}
\chim \equiv 1 + \frac{\eta_{\bar{\gp}}-\eta_{\gp}}{\log_p(1+p^{\nu})} X\pmod{X^2}. 
\end{align}

 Put $\Psi=\chim \psim$ and for $n\in \Z_{\geqslant 1}$ let $\Psi_n=\Psi \mod{X^n}$.  
Finally, let  $\Gamma_K^{Np}$ be the Galois group of the maximal extension of $K$ unramified outside  $Np$.  

\begin{prop} \label{def-red-rho}
\begin{enumerate}[wide]
\item 
The $\varLambda_\gp$-module $\rH^1(\Gamma_K^{Np}, \Psi^{\pm 1})$ is free of rank $1$ and  
for all $n\in \Z_{\geqslant 1}$   the following natural map is an isomorphism 
\[ \rH^1(\Gamma_K^{Np}, \Psi^{\pm 1})\otimes_{\varLambda_\gp} \varLambda_{\gp}/(X^n) \xrightarrow{\sim} \rH^1(\Gamma_K^{Np}, \Psi_n^{\pm 1}).\]
\item  Assume that $\cLm(\psim) \ne 0$. Then for every $n\geqslant 1$,  the natural restriction map is injective
\[ \rH^1(\Gamma_K^{Np}, \Psi_{n}^{-1} ) \to \rH^1(\Gamma_{K_{\gp}}, \Psi_{n}^{-1}).\]
\end{enumerate}
\end{prop}
\begin{proof}
(i) The proof is similar to \cite[Prop.~2.8]{BDPozzi}, as  $\dim \rH^1(K,\psim)=\dim \rH^1(K,\psim^{\tau})=1$ (by Proposition~\ref{restinj}) and $\dim \rH^2(K,\psim)=\dim \rH^2(K,\psim^{\tau})=0$ (by Lemma~\ref{H2}).

(ii)  As in {\it loc. cit.} we have to check that $\eta'$ and 
$\frac{\eta_{\gp}-\eta_{\bar\gp}}{\log_p(1+p^{\nu})}\equiv\frac{\chim^\tau-1}{X} \mod X$ (see \eqref{eq:univ-coates-wiles2})
generate distinct lines in $ \rH^1(K_{\gp}, \overline{\Q}_p)$. The claim follows  from the fact that  $(\eta'-\eta_{\gp}+\eta_{\bar{\gp}})_{\mid \I_{\gp}}=0$ while  $\cLm(\psim)=(\eta'-\eta_{\gp}+\eta_{\bar{\gp}})(\Frob_{\gp}) \ne 0$.\end{proof}

We recall that the completed local  ring at $f$ of the irreducible component of $\cE$ corresponding to $\Theta_{\psi}$ is naturally isomorphic to $\varLambda_{\gp}$. 

\begin{cor} \label{cor:def-red} There exists  a natural isomorphism  of 
   complete Noetherian local $\varLambda$-algebras   $\cR^K_\rho\xrightarrow{\sim} \varLambda_\gp$. 
\end{cor}
\begin{proof}  Taking a  generator of the free rank $1$ $\varLambda_\gp$-module $\rH^1(\Gamma_K^{Np}, \Psi)$ (see   Proposition~\ref{def-red-rho}) yields a  point in $\cD^K_{\rho}(\varLambda_\gp)$,  hence a $\varLambda$-algebra homomorphism
 $\cR^K_\rho\to \varLambda_\gp$.  By Proposition~\ref{dim t_R} the relative tangent space $t^K_{\rho,0}$ of the 
 structural homomorphism $\varLambda \rightarrow \cR^K_\rho$ vanishes, and  hence $\cR^K_\rho/X\cR^K_\rho=\overline\Q_p$ ({\it i.e.} $\varLambda \rightarrow \cR^K_\rho$ is unramified). As    $\cR^K_\rho$ is separated for the $X$-adic topology, Nakayama's lemma implies that $\varLambda\to \cR^K_\rho$ is  surjective. The injectivity follows from the fact that   the composed map  
$\varLambda\twoheadrightarrow \cR^K_\rho \to \varLambda_\gp$ is given by the natural isomorphism  \eqref{eq:Lambda-iso}. 
\end{proof}

Let   $J$ be the kernel of the $\varLambda$-algebra homomorphism $\cR^{\ord}_{\rho} \twoheadrightarrow \cR_{\rho}^{K}\simeq
\varLambda_{\gp}$. As  $\cR^{\ord}_\rho$ is Noetherian, we know that $J/J^2$ is a module of finite type over $\cR^{\ord}_\rho/J\simeq\varLambda_\gp$. 

  \begin{prop} \label{J}  If $\cLm(\psim)  \ne 0$, then the $\varLambda_\gp$-module $J/J^2$ is torsion of length $\leqslant 1$.
\end{prop}

\begin{proof} 
It is enough to show that for  $n\geqslant 1$ the  length of the $\varLambda_\gp$-module
 $\Hom_{\varLambda_\gp}(J/J^2, \varLambda_\gp/(X^n))$  is at most $1$. 
We  use a homological machinery as in \cite[Prop.~4.10]{BDPozzi}.  Let us 
   fix a realization  $\rho_{\cR}(g)= \left( \begin{smallmatrix} 
a(g)& b(g)\\  c(g) & d(g) \end{smallmatrix} \right)$
of the universal representation  in an ordinary basis. 
By Lemma~\ref{lemma-red} the  ideal $J$ of $\cR^{\ord}_{\rho}$ is generated by the set of $c(g)$ for $g \in \Gamma_K$. As
\[ \rho_{\cR} \otimes (\cR^{\ord}_{\rho}/J) = \left( \begin{smallmatrix} 
\psi \chi_{\gp} & \ast\\ 
0 & \psi^\tau \chi_{\gp}^{\tau} \end{smallmatrix} \right)
\]
a direct computation using that the basis is $\Gamma_{K_{\gp}}$-ordinary shows that  the function \begin{align}
\bar{c}: \Gamma_K \to J/J^2 ,\,\, g \mapsto (\psi\chi_{\gp})^{-1}(g) \cdot c(g)\mod{J^2}
\end{align}
belongs to $\ker\left(\rZ^1(\Gamma_K^{Np}, (\chim^\tau\psim^\tau)\otimes_{\varLambda_\gp} (J/J^2))\to
\rZ^1(K_{\gp}, (\chim^\tau\psim^\tau)\otimes_{\varLambda_\gp} (J/J^2))\right)$. As the image of $\bar{c}$ contains 
a set of generators of $J/J^2$ as $\cR^{\ord}_{\rho}/J \simeq \varLambda_\gp$-module, 
the natural map 
\begin{align*}
\Hom_{\varLambda_\gp}(J/J^2, \varLambda_\gp/(X^n))\to  
\ker \left (\rZ^1(\Gamma_K^{Np},\Psi_{n}^{-1})  \to \rZ^1(K_{\gp},\Psi_{n}^{-1})\right ), \,\, h \mapsto h\circ \bar{c} 
\end{align*}
is injective for all $n\geqslant 1$. Moreover by Proposition~\ref{def-red-rho}(ii)  one has 
\begin{align}\label{eq:coboundary}
\ker \left (\rZ^1(\Gamma_K^{Np},\Psi_{n}^{-1})  \to \rZ^1(K_{\gp},\Psi_{n}^{-1})\right )
=\ker \left (\mathrm{B}^1(\Gamma_K^{Np},\Psi_{n}^{-1} ) \to \mathrm{B}^1(K_{\gp},\Psi_{n}^{-1})\right ).
\end{align}
As $\Psi_1|_{\Gamma_{K_{\gp}}}=\psim|_{\Gamma_{K_{\gp}}}=\mathbf{1}$, while 
$\Psi_2|_{\Gamma_{K_{\gp}}}\ne\mathbf{1}$, the coboundary in \eqref{eq:coboundary} is given by the length $1$
$\varLambda_\gp$-module  $(X^{n-1})/(X^n)\simeq \overline{\Q}_p$, proving  the desired assertion.\end{proof}

We begin the study of the congruences between $\Theta_{\psi}$ and a family $\cF$  without CM by $K$,
by  determining the exact congruence ideal at split primes (note that the eigenvalues  of $\cF$ and $\cF\otimes \epsilon_K$ coincide at such primes). 
This  amounts to the study of the image $\cT_\rho^{\spl}$ of  $\cT_{\spl}$ in $\cT_{\spl}^\perp\times\varLambda_\gp$, 
which we will now   relate to the universal deformation ring $\cR_\rho^{\ord}$  from \S\ref{deformation}.

Recall the functorial homomorphism   $\cR^{\ord}_{\rho}  \to \cR^{\perp} \times_{\overline{\Q}_p} \varLambda_\gp$ 
which surjects on both components. As observed in the proof of Lemma~\ref{lem-surj}, the $\varLambda$-algebra  $\cR_\rho^{\ord}$ is generated by the trace of the universal ordinary deformation of $\rho$, hence there exists an integer  $m \geqslant 1$ and a 
commutative diagram of surjective homomorphism of $\varLambda$-algebras:\begin{align}\label{RarrowT} 
\xymatrix{
 \cR^{\ord}_{\rho} \ar@{->>}[d]\ar@{->>}[r] & \cR^\perp \times_{\overline{\Q}_p[X]/(X^m)} \cR^K_\rho \ar@{=}[d] \\
\cT^{\spl}_\rho  \ar@{=}[r] &  \cT_{\spl}^\perp \times_{\overline{\Q}_p[X]/(X^m)} \varLambda_\gp,}
\end{align}
where the   natural surjection $\cR^{\ord}_\rho \twoheadrightarrow \cT^{\mathrm{split}}_\rho$ sends  $\chi_\cR(\Frob_{\gp})$ to $U_p=(U_p(\cF),U_p(\Theta_{\psi}))$.

 \begin{prop}\label{R^ord2comp}
  If $\cLm(\psim) \ne 0$,  there is an isomorphism of  local complete  intersection rings
\[ \cR^{\ord}_{\rho}  \xrightarrow{\sim} \cT^{\spl}_\rho=\cR^\perp \times_{\overline{\Q}_p} \varLambda_\gp.\]
  \end{prop}

\begin{proof} Corollary~\ref{system2} implies that $m=1$. The isomorphism 
 is then a consequence of a variant of Wiles' numerical criterion due to Lenstra \cite{lenstra} in view of Proposition~\ref{J}.\end{proof}
 
As  $\cR_\rho^K$ is a quotient of $\cR_\rho^{\red}$, the CM ideal $(X^m)$ contains the reducibility ideal $(X^r)$ and plays a role analogous to that of the Eisenstein ideal. 
  
  \begin{prop}\label{m=r-1}
Assume that $\cLm(\psim) =0$. Then $m=r-1$, and $\cT^{\spl}_\rho = \varLambda \times_{\overline{\Q}_p[X]/(X^{r-1})} \varLambda_\gp$. 
\end{prop}

\begin{proof}  
  As $\cLm(\psim)=0$, the structural homomorphism   $\varLambda\xrightarrow{\sim} \cR^{\perp}$ is an isomorphism by 
  Theorem~\ref{isom-cusp}.   We recall that $r \geq 3$ by Proposition~\ref{reducn=2},  and that 
\[\rho_{\cR^{\perp}}(g) \equiv \left(\begin{matrix} A(g) &  B(g)\\ \psi(g)\eta'(g) \cdot X^r & D(g)\end{matrix}\right) \pmod{X^{r+1}},\] 
 where $A,B,D$ are polynomials in $X$ of degree $\leqslant r$. As  $\cR^{\perp} = \overline{\Q}_p\lsem X \rsem$ is a discrete valuation ring, the ordinary filtration of $\rho_{\cR^{\perp}}$ is generated by $e_1 + X^s  e_2$ for some $s \geqslant 1$ and, using the corresponding $\gp$-ordinary basis $(e_1 +X^s  e_2, e_2)$, one has
\[\rho_{\cR^{\perp}}(g) \equiv \left(\begin{matrix} A(g)- B(g) \cdot X^s &  B(g) \\ 
   \psi(g)c(g) \cdot X^r+ (A(g)-(D(g)+B(g)\cdot X^s))\cdot X^s & D(g)+ B(g)\cdot X^s\end{matrix}\right) \mod X^{r+1}.\]  

It is crucial to first observe that  $s \leq r-1$. Let us proceed by absurd, assuming that $s \geq r$. The lower left entry of $\rho_{\cR^{\perp}}(g)\mod{X^{r+1}}$  in the ordinary basis $(e_1 +X^s  e_2, e_2)$ equals 
\[\psi(g) \left(\eta'(g) - (\psim^\tau(g)-1)\cdot X^{s-r}\right)\cdot X^r \] 
and it has to be trivial on $\I_{\gp}$, contradicting Proposition~\ref{restinj}, as
$0\ne [\eta']\in \rH^1(K, \psim^\tau)$.  

Therefore $s < r$, hence  $D \mod X^{r}$ is $\gp$-ramified. The $\gp$-ordinarity of $\rho_{\cR^{\perp}}\mod X^{r}$ yields  \begin{align}\label{ordinarinesscond} A(g)\equiv 1 \pmod{X^{r-s}} \text{ and } D(g)\equiv 1 \pmod{X^{s}}, \text{ for all } g\in  \I_{\gp}.
\end{align}

As $m \geqslant 2$ (see  Cor.~\ref{system2}) one has $\tr(\rho_{\cR^{\perp}}) \equiv  \psi\chi_{\gp}+ \psi^{\tau}\chi_{\gp}^{\tau} \mod X^2$.  
 Lemma~\ref{unicity-decomp}  implies 
\[A\equiv \psi\chi_{\gp} \equiv \psi (1+ \eta_{\gp} \cdot X)  \pmod{X^2}, \text{ and } 
D\equiv \psi^{\tau}\chi_{\gp}^{\tau} \equiv \psi^{\tau} (1+ \eta_{\gp}^{\tau} \cdot X)  \pmod{X^2} \]
which together with \eqref{ordinarinesscond} implies that $s=r-1$, as $\eta_{\gp|\I_{\gp}}\neq 0$. 
It then follows from the second relation of \eqref{ordinarinesscond} that $D \mod X^{r-1}$ is unramified at $\gp$,  hence  there exists a morphism $\varLambda_\gp \rightarrow \cR^{\perp}/(X^{r-1})$ sending $\psi^{\tau}\chi_{\gp}^{\tau}$ on $D \mod X^{r-1}$. Using the $\tau$-invariance of $\tr(\rho_{\cR^{\perp}})$, a second application of Lemma~\ref{unicity-decomp} yields that  $A \equiv D^\tau \pmod{X^r}$. Hence $\cR^{\perp}/(X^{r-1})$ is topologically generated by $\{ D(g), g \in \Gamma_K \}$ and therefore the morphism $\varLambda_\gp \rightarrow \cR^{\perp}/(X^{r-1})$ defined above is in fact a surjection, and it yields that $m\geqslant r-1$. 

It remains to show that $m<r$, {\it i.e.} $\tr(\rho_{\cR^{\perp}}) \nequiv  \psi\chi_{\gp}+ \psi^{\tau}\chi_{\gp}^{\tau} \pmod{X^r}$. 
As $A \equiv D^\tau \pmod{X^r}$, it suffices (again by Lemma~\ref{unicity-decomp}) to show that $D \nequiv \psi^{\tau}\chi_{\gp}^{\tau} \pmod{X^r}$ which follows from the fact that $D \mod X^{r}$ is ramified at $\gp$ as  observed earlier.
\end{proof}

Proposition~\ref{m=r-1} implies in particular that if $\cLm(\psim) = 0$ then $U_p(\cF)\equiv U_p(\Theta_{\psi})  \mod X^{r-1}$ 
and we will  now see that this congruence is optimal. 

\begin{prop}\label{U_pcongFtheta} Assume that $\cLm(\psim) = 0$, then $U_p(\cF) \not\equiv U_p(\Theta_{\psi}) \mod X^{r}$.
\end{prop}

\begin{proof}  We proceed by absurd assuming  that $U_p(\cF)\equiv U_p(\Theta_{\psi})  \mod X^{r}$. 
As $\cT^{\mathrm{split}}_{\rho}$ contains $(X^{r-1},0)$ and  $(0,X^{r-1})$, we deduce that $U_p- \psi(\gp)\in X \cdot \cT^{\mathrm{split}}_{\rho}$ hence it is mapped to $0$ by 
\begin{align}\label{R0toT0}\cR^{\ord}_{\rho,0} \twoheadrightarrow \cT^{\mathrm{split}}_{\rho,0}= \cT^{\mathrm{split}}_{\rho}/X \cdot \cT^{\mathrm{split}}_{\rho}. \end{align} 
By Propositions~\ref{dim t_R} and \ref{m=r-1}, \eqref{R0toT0} gives rise to an isomorphism of tangent spaces
\[ \Hom_{\mathrm{alg}}(\cT^{\mathrm{split}}_{\rho,0}, \overline{\Q}_p[\epsilon]) \xrightarrow{\sim}\Hom_{\mathrm{alg}}(\cR^{\ord}_{\rho,0}, \overline{\Q}_p[\epsilon])=t^{\ord}_{\rho,0} .\]
To obtain a contradiction it then suffices to show that the image of $\chi_\cR(\Frob_{\gp})- \psi(\gp)$ under any non-zero element of the line 
$t^{\ord}_{\rho,0}$ does not vanish. Recall from \S\ref{tangent} that $t^{\ord}_{\rho,0}$ is the subspace of 
 $t^{\ord}_{\rho} \simeq \Hom(\Gamma_K,\overline{\Q}_p) \simeq \{ \alpha \cdot \eta_{\gp} +\beta \cdot \eta_{\bar{\gp}}, (\alpha,\beta) \in \overline{\Q}{}_p^2\}$  given by $\alpha=-\beta$. 
  Hence, by  \eqref{U_p equation1}   
\[\chi_\epsilon(\Frob_{\gp}) -\psi(\gp) = \epsilon \alpha [(\eta_{\gp}-\eta)(\Frob_{\gp}) -  \eta_{\bar{\gp}}(\Frob_{\gp})]\psi(\gp)=  -\epsilon \alpha\cLm(\psim^\tau) \psi(\gp).\] 
As $\cLm(\psim) = 0$, Proposition~\ref{6expocusp} implies that  $\cLm(\psim^\tau) \ne 0$ proving the claim. 
 \end{proof}

\subsection{The  congruence ideal}\label{s:cong-ideal}

We recall that    
$C^0_\psi=\pi_\psi (\mathrm{Ann}_{\cT}(\ker(\pi_\psi)))$ is  the congruence ideal attached to the $\Theta_{\psi}$-projection $\pi_\psi:\cT \twoheadrightarrow \varLambda_\gp$. 
By Proposition~\ref{def-red-rho} and Corollary~\ref{cor:def-red}, one may attach to $\Theta_{\psi}$ (resp. $\Theta_{\psi^{\tau}}$) 
a $\varLambda_{\gp}$-valued $\Gamma_K$-reducible deformation of $\rho$ (resp.  $\rho'$) whose semi-simplification 
is given by   $\psi\chi_{\gp}\oplus\psi^\tau \chi_{\gp}^{\tau}$ (resp.  $\psi\chi_{\gp}^{\tau}\oplus\psi^\tau \chi_{\gp}$). 
In order to compute $C^0_\psi$ we will first determine $\cT_{\spl}$, then $\cT$, using the  $\varLambda$-algebra homomorphism 
\begin{align}
(\pi_\psi, \pi^\perp, \pi_{\psi^\tau}): \cT\hookrightarrow \varLambda_{\gp} \times_{\overline{\Q}_p}  \cT^\perp\times_{\overline{\Q}_p}  \varLambda_{\gp}. 
\end{align}    

\begin{thm}\label{T-split} 
Assume that $\cLm(\psim) \ne 0$. Then $\cT_{\spl}[U_p]= \varLambda_\gp \times_{\overline{\Q}_p} \left( \cR^\perp \times_{\overline{\Q}_p[X]/(X^{r-1})} \varLambda_{\gp}\right)$. 

\begin{enumerate}[wide]
\item 
If moreover $ 0\ne \cLm(\psim^\tau)  \ne -\cLm(\psim)$,  then 
\[\cT_{\spl}= \left\{(a,b,c)\in \varLambda_\gp \times_{\overline{\Q}_p} \varLambda\times_{\overline{\Q}_p} \varLambda_{\gp}\, \Big{|} \,
\left(\cLm(\psim)+\cLm(\psim^\tau) \right)b'(0)=\cLm(\psim^\tau) a'(0)+ \cLm(\psim) c'(0)  \right\}. \]
\item 
If  moreover $\cLm(\psim^\tau) =0$, then $e=1$, $r\geqslant 3$ and there exists $\xi\in \overline{\Q}{}_p^\times$ such that 
\[\cT_{\spl}= \left\{(a,b,c)\in  \varLambda_\gp \times_{\overline{\Q}_p} \left( \varLambda \times_{\overline{\Q}_p[X]/(X^r)} \varLambda_{\gp}\right) \,\Big{|}  \,
(b-c)^{(r-1)}(0)= \xi\cdot  (a-b)'(0)  \right\}. \]
\item If  moreover  $\cLm(\psim^\tau)=-\cLm(\psim)$, then $e\geqslant 2=r$ and there exists $\xi\in \overline{\Q}{}_p^\times$ such that
\[\cT_{\spl}= \left\{(a,b,c)\in \varLambda_\gp \times_{\overline{\Q}_p} \cR^\perp  \times_{\overline{\Q}_p} \varLambda_{\gp}  \,\Big{|}  \,
 b'(0)= \xi\cdot  (a-c)'(0)  \right\} . \]
\end{enumerate}
\end{thm}

\begin{proof}
According to Propositions~\ref{Tcmgor}, \ref{R^ord2comp} and \ref{m=r-1} one knows that 
$\cT_{\spl}$ is a $\varLambda$-sub-algebra of the amalgamated product 
$A=\varLambda_\gp \times_{\overline{\Q}_p} \left( \cR^\perp \times_{\overline{\Q}_p[X]/(X^{r-1})} \varLambda_{\gp}\right)$ surjecting to the product of each two amongst the three factors. In particular $\cT_{\spl}$ contains $(X,b,0)$ and $(X,0,c)$, for certain  $b\in \cR^\perp$ and $c\in \varLambda_{\gp}$, hence  their product $(X^2,0,0)$, from which one  deduces that
\[\varLambda_\gp \times_{\overline{\Q}_p[\epsilon]} \left( \cR^\perp \times_{\overline{\Q}_p[X]/(X^{r-1})} \varLambda_{\gp}\right) \subset \cT_{\spl}.\] 
As $\cT_{\spl} \subset \cT_{\spl}[U_p] \subset A$, one deduces that
$\gm_{\cT_{\spl}}$ is given by the kernel of a  $\overline{\Q}_p$-valued linear form on the maximal ideal $\gm_A$, which we will determine as precisely as possible in each case. 
As $U_p\notin \cT_{\spl}$ by Theorem~\ref{existencenon-cm}(i), we deduce that 
the above linear form is non-zero and that 
 $\cT_{\spl}[U_p]=A$, proving the first part of the theorem. Let us now write an 
 equation for $\gm_{\cT_{\spl}}$  in each case. 
 
(i) One sees exactly as in  \cite{BDPozzi} that $\gm_{\cT_{\spl}}\supset (\gm_\varLambda^2)^3$.  
Moreover the image of $\gm_{\cT_{\spl}}$ in $(\gm_\varLambda/\gm_\varLambda^2)^3$ 
is necessarily a plane (as it surjects onto any two amongst the three factors). The precise 
equation follows from \eqref{dag-trace} and Proposition~\ref{Tcmgor}, in view of  \eqref{U_pL}. 

(ii) In this case $\cLm(\psim^\tau)=0\ne \cLm(\psim)$, hence  $e=1$ and $\gm_{\cT_{\spl}}$ contains 
the elements $(X,X,X)$, $(X^2,0,0)$ and $(0,X^r,0)$, generating an ideal which has co-dimension $2$ 
in the  $\overline{\Q}_p$-vector space  $\gm_A$.  Again the surjectivity onto each  two amongst the three factors shows that $\gm_{\cT_{\spl}}$ has co-dimension $1$ in $\gm_A$, hence satisfies 
a linear equation of the desired form. 

(iii) We leave its proof, which is very similar to the above case, to the interested reader. 
\end{proof}

\begin{thm}\label{C0-thm} One has $C^0_\psi=(X)$ if and only if $\cLm(\psim) \ne 0$. 
More precisely
 \begin{enumerate}[wide]
\item If $\cLm(\psim)\cdot \cLm(\psim^\tau) \ne 0$, then $\cT = \varLambda_{\gp} \times_{\overline{\Q}_p} \cR^{\perp} \times_{\overline{\Q}_p} \cR^{\perp} \times_{\overline{\Q}_p} \varLambda_{\gp}.$

\item If  $\cLm(\psim^\tau) =0$, then  $\cT = \varLambda_\gp \times_{\overline{\Q}_p}
\left( \varLambda[Z]/(Z^2-X^r)\times_{\overline{\Q}_p[X]/(X^{r-1})} \varLambda_{\gp}\right)$ 
where the projection to the first  component is given by $\pi_\psi:\cT \twoheadrightarrow \varLambda_\gp$, $r\geqslant 3$ and the  map $\varLambda[Z]/(Z^2-X^r) \twoheadrightarrow \overline{\Q}_p[X]/(X^{r-1})$  is modulo the non-principal ideal $(Z,X^{r-1})$.  In particular $C^0_{\psi^{\tau}}=(X^{r-1})$.
\end{enumerate}
\end{thm}

\begin{proof} The claim about the congruence ideal  would follow from (i) and (ii). 

(i) follows directly from Theorems~\ref{T-split}(i)(iii) and \ref{geomCdag}. 

(ii) By Theorem~\ref{T-split}(ii) it suffices to show that $(0,Z,0)\in \cT$. As 
$\cLm(\psim)+\cLm(\psim^\tau) \ne 0$  we have $\cR^{\perp}=\varLambda$ in this case. 
Recall the involution $\iota$ of $\widetilde{\cR}^{\perp}$ fixing $\cR^{\perp}=\varLambda$ and sending $Z$ to $-Z$. 

For all $\gamma \in \Gamma_K$  one has  \begin{align}\label{traceiotageneZ} \iota \circ \tr (\tilde{\rho}^{\perp}_{\cR})(\gamma\tau)  = \tr (\tilde{\rho}^{\perp}_{\cR}\otimes \epsilon_K)(\gamma\tau)=-\tr (\tilde{\rho}^{\perp}_{\cR})(\gamma\tau),\end{align}
hence  $ \tr (\rho^{\perp}_{\cR})(\gamma\tau) \in Z\cdot \varLambda$. Finally, Proposition~\ref{tildeRtrace} 
implies that the $\varLambda$-sub-module of $\widetilde{\cR}^{\perp}$ generated by 
$\tr (\tilde{\rho}^{\perp}_{\cR})(\tau \Gamma_K)$ contains $Z$.
\end{proof}

We conclude this subsection by determining the  congruences ideal between $\cF$ and $\Theta_{\psi}$, as well as the congruence ideal $C^0_\cF$ attached to the projection  of $\cT$ on its $\cF$-component. 
\begin{cor} 
\begin{enumerate}[wide]
\item Assume that $r \geqslant 3$ is  odd. 

Then $C^0_\cF=(X^{r-1},Z)\subset \widetilde{\cR}^{\perp}=\varLambda[Z]/(Z^2-X^r)$. 
The  congruence ideal between $\cF$ and $\Theta_{\psi}$ is given  by
$(X^{r-1},Z)$, if $\cLm(\psim)=0$, and by $(X,Z)$, otherwise. 

\item Assume that $r \geqslant 2$ is  even. Then $C^0_\cF=(Y^{r-1})\subset \cR^{\perp}=\overline{\Q}_p\lsem Y\rsem$. 
The  congruence ideal between $\cF$ and $\Theta_{\psi}$ is given  by
$(Y^{\frac{r}{2}})$, if $\cLm(\psim) =0$, and by $(Y)$, otherwise. 
\end{enumerate}
\end{cor}

\subsection{The Katz $p$-adic $L$-function at $s=0$}\label{Katz-padic}

Given an ideal $\gc$ of $K$ relatively prime to $p$, Katz  constructed a measure $\mu_{\gc}$ on the 
 ray class group $\cC\ell_K(p^{\infty}\gc)$ whose value on (the $p$-adic avatar of) a Hecke character $\varphi$ 
 of conductor dividing $\gc$ and   infinity type $(k_1,k_2)\in\Z_{\leqslant -1}\times \Z_{\geqslant 0}$ is given by 
(see \cite[Thm.~3.1]{BDP} based on \cite{Katz} and \cite{deShalit}) 
\[(-k_1-1)!  \left(\frac{\Omega_p}{\Omega_\infty}\right)^{k_2-k_1} \left(\frac{2\pi}{\sqrt{D}}\right)^{k_2} (1-\varphi^{-1}(\bar\gp))(1-\varphi(\gp)p^{-1})L_{\gc}(\varphi^{-1},0), \text{ where} \]
 $\Omega_\infty$ (resp. $\Omega_p$) is a complex (resp. $p$-adic) period of the CM elliptic curve attached to $\bar{\gc}^{-1}$.
The above interpolation formula uniquely characterizes $\mu_{\gc}$ as these characters are Zariski dense in $\Spec \cO \lsem \cC\ell_K(p^{\infty}\gc) \rsem$. It does not  apply however to  finite order characters since their infinity type  is $(0,0)$.

The Katz $p$-adic $L$-function of a  finite  order {\it non-trivial anti-cyclotomic }  character $\varphi$ of conductor dividing $\gc$, is 
defined as the following continuous function on $\Z_p^2$
\[L_p(\varphi, s_{\gp}, s_{\bar\gp})= \mu_{\gc}(\varphi \varepsilon_{\gp}^{s_{\gp}}\varepsilon_{\bar\gp}^{s_{\bar\gp}}),\]
where $\varepsilon_{\gp}$ is the character defined in 
\eqref{eq:epsilon} and  $\varepsilon_{\bar\gp}= \varepsilon_{\gp}^\tau$.
In other terms, $L_p(\varphi,s_{\gp},s_{\bar\gp})$ is the $p$-adic Mellin transform of the  push-forward of $\mu_\gc$ by
$\cC\ell_K^{(p)}(p^{\infty}\gc)_{/\tor} \hookrightarrow  \W_{\gp} \times \W_{\bar{\gp}}$.
 As the function $ s\mapsto z_{\gp}^s$  equals the  analytic function $\exp_p(s\cdot p^{-h}\log_p(1+p^\nu))$ on  $p^h  \Z_p$,  and as 
$\cC\ell_K^{(p)}(p^{\infty})_{/\tor} \hookrightarrow  \W_{\gp} \times \W_{\bar{\gp}}\xrightarrow{\sim}
p^{-h}\Z_p\times p^{-h}\Z_p$ consists of pairs with difference  in $\Z_p$, it follows that $L_p(\varphi, s_{\gp}, s_{\bar\gp})$
is analytic on the set of $(s_{\gp}, s_{\bar\gp})\in\Z_p^2$  such that $s_{\gp} + s_{\bar\gp}\in p^h\Z_p$, in particular it is 
locally analytic at $(0,0)$.

Furthermore Katz constructed a $1$-variable improved $p$-adic $L$-function $L_p^*(\varphi, s)$ such that
\begin{align}\label{improved}
L_p(\varphi, s, 0)= (1-(\varphi^{-1} \varepsilon_{\gp}^{-s})(\bar\gp) )\cdot  L_p^*(\varphi, s), 
\end{align}
 (see \cite[\S7.2]{Katz}) and proved a $p$-adic analogue of the Kronecker's Second Limit Formula 
  \begin{align}\label{kronecker}
 L_p^*(\varphi, 0)= -(1-\varphi(\gp)p^{-1}) \cdot  \log_p(\mathfrak{u}_{\varphi}).  
\end{align}  
(see \cite[Cor.~10.2.9]{Katz}, \cite[Thm.~1.5.1]{HT2}), where 
$\mathfrak{u}_{\varphi}\in (\cO_H[\gc^{-1}]^\times \otimes \overline \Q)[\varphi]=(\cO_H^\times \otimes \overline \Q)[\varphi]$ 
is a specific unit, defined using Robert units (see \cite[(1.5.5)]{HT2}).

\begin{rem}
Assume the Heegner hypothesis for the ideal $\gc$, {\it i.e},  $\cO_K/\gc\cO_K \simeq  \Z/C\Z$ and that $D\neq 3,4$. 
Let $Y_1(C)/\Q$ be the open modular curve of level $C$,  $\zeta$ be a $C$-th root of the unity and $g_\zeta(q)=q^{\frac{1}{12}}(1-\zeta) \prod_{n >0}    (1-q^n \zeta)(1-q^n\zeta^{-1}) \in \cO(Y_1(C))^\times$ be the Siegel unit associated to $\zeta$.  The theory of Complex Multiplication  implies that the elliptic curve $(\C/\bar{\gc}^{-1},C^{-1})$ defines a point of $Y_1(C)$ over the ray class field $K_{\gc}$ of conductor $\gc$ of $K$.  The   evaluation of the Siegel units $g_\zeta(q)$ at this   point is an  elliptic unit $u_{\zeta,\gc} \in \cO_{K_{\gc}}^{\times}$,  which can be used to define  $\mathfrak{u}_{\varphi}$. 
\end{rem}

Suppose that  $\varphi(\gp)=1$. The  Euler factor $(1-\varphi(\gp))$ in \eqref{improved}  vanishes yielding a trivial zero 
 \begin{align}\label{trivial-zero}
 L_p(\varphi, 0, 0)=0.
 \end{align}   
As  $\log_p(\mathfrak{u}_{\varphi})\neq 0$ by the Baker--Brumer Theorem, it follows that this is the only case where a trivial zeros of the
Katz $p$-adic $L$-function at $(0,0)$ can occur.

 \begin{rem}
When $\varphi$ is trivial, a case considered in \cite{BDPozzi} which we have excluded from the present paper, then  one has $L_p(\mathbf{1}, 0, 0)= \frac{-1}{2}(1-p^{-1}) \log_p(u_{\bar\gp})\neq 0$, where $u_{\bar\gp}$ is a $\bar\gp$-unit of $K$ of minimal  valuation.   Note that  there is no improved $p$-adic $L$-function is that case. 
\end{rem}

A first insight into what should be the leading term of $L_p(\varphi, s_{\gp}, s_{\bar\gp})$ at the trivial zero point $L_p(\varphi, 0, 0)=0$, is obtained by  differentiating \eqref{improved} \begin{align}\label{derivativegpram}
\frac{\partial L_p(\varphi)}{\partial s_{\gp}}(0,0)=-\varphi(\bar{\gp}) \eta_{\gp}(\Frob_{\bar\gp})\cdot L_p^*(\varphi,0)
=\cL_{\gp} \cdot  L_p^*(\varphi,0)
\end{align}

To gain a further intuition about  the linear term, we consider the Katz  anti-cyclotomic $p$-adic $L$-function $L_p^-(\varphi, s)= L_p(\varphi, s, -s)$.

As $\varphi$ is  anti-cyclotomic of finite order, there exists a finite order Hecke character $\psi$ which can be chosen 
to have  conductor $\gc_{\psi}$ relatively prime to $p$,  such that  $\varphi=\psim=\psi/\psi^\tau$. 

Let $\zeta_{\psim}^{-} \in \Lambda_{\cO}=\cO \lsem \W_{\gp}\rsem$ be the  power series corresponding to  
the push-forward of $\mu_{\gc_{\psim}}$ by 
\[\cC\ell_K(p^{\infty}\gc_{\psim}) \xrightarrow{z\to z/z^{\tau}}\cC\ell_K(p^{\infty}\gc_\psi) \twoheadrightarrow 
\cC\ell_K(\gp^{\infty}\gc_\psi) \to \Lambda_{\cO}^\times,\]
where the last map is the $\psi$-projection (see  \S\ref{cm-chars}).
By definition  $(\varepsilon_{\gp}^s)(\zeta_{\psim}^{-}) =\mu_{\gc_{\psim}}(\psim  \varepsilon_{\gp}^s \varepsilon_{\bar\gp}^{-s})=
L_p^-(\psim, s)$
is analytic in $s\in \Z_p$. 
The vanishing \eqref{trivial-zero} of $L_p^-(\psim, 0)$ implies  $\zeta_{\psim}^-\in \ker\left(\cO \lsem \W_{\gp} \rsem \to \cO\right)$, hence the image of $\zeta_{\psim}^{-}$ under the localization at  $(X)$ morphism $\cO \lsem \W_{\gp} \rsem \hookrightarrow \varLambda_\gp$ lands in $(X)$.

\begin{thm}[Hida--Tilouine] \label{HT-divisibility} 
The anti-cyclotomic $p$-adic $L$-function $\zeta^{-}_{\psim}$ divides the generator of the congruence ideal 
$C^0_\psi$ in $\varLambda_\gp$. 
\end{thm}
\begin{proof}   
Recall from \S\ref{cm-chars} the one variable Iwasawa algebra $\Lambda_{\cO}=\cO \lsem \W_{\gp}\rsem $ and 
the Hecke character $\lambda_{\psi,k}$ corresponding to the  classical theta series  $\theta_{\psi,k}$ which is the  
specialization in weight   $k \in \Z_{\geq 1}$ of the CM Hida family $\Theta_{\psi}$. 

 Let $H_{\psi} \in \Lambda_{\cO}$ be the characteristic power series of the congruence module attached to $\Theta_{\psi}$
  introduced by   Hida and Tilouine  in \cite[(6.9)]{HT}. Given a Hecke finite order character $\phi$ of $K$ having a 
    prime to $p$ conductor $\gc_\phi$ and taking values in  $\cO$,  Hida constructed in \cite[Thm.~I]{hida-AIF88} an element   $\mathscr{D}$ of the fraction field of $\Lambda_{\cO}\widehat{\otimes} \Lambda_{\cO}$ such that $H_{\psi}\cdot  \mathscr{D} \in \Lambda_{\cO}\widehat{\otimes} \Lambda_{\cO}$, 
satisfying for all $k > l \geqslant 2 $ the  interpolation 
   \begin{align}\label{p-adic-rankin} 
(\varepsilon_{\gp}^{k-1},\varepsilon_{\gp}^{l-1})(\mathscr{D})\approx \frac{D_p(1+\frac{l-k}{2}, \theta_{\psi,k},\theta^\tau_{\phi,l})}{(\theta_{\psi,k},\theta_{\psi,k})}\approx \frac{L(0,\lambda_{\phi,l}\lambda_{\psi,k}^{-1}) 
L(0,\lambda_{\phi,l}^{\tau}\lambda_{\psi,k}^{-1})}{L(0,\lambda_{\psi,k}^{\tau}\lambda_{\psi,k}^{-1})}, 
\end{align}
up to a factor made explicit in  \cite[Thm.~I]{hida-AIF88}, 
where $D_p$ is the Rankin-Selberg convolution without  Euler factors at $p$, and $(\theta_{\psi,k},\theta_{\psi,k})$ is the Petersson inner product.

Let $L_1$ (resp. $L_2$,  $L^{-}_{\psi}$)  be the unique element of $\Lambda_{\cO}\widehat{\otimes} \Lambda_{\cO}$ whose specialization at $(\varepsilon_{\gp}^{k-1},\varepsilon_{\gp}^{l-1})$ is given by the value $\mu_{\gc_{\psi/\phi}}( \lambda_{\psi,k}\cdot (\lambda_{\phi,l})^{-1})$ (resp. 
$\mu_{\gc_{\psi/\phi^{\tau}}}( \lambda_{\psi,k} \cdot (\lambda_{\phi,l}^{\tau})^{-1})$, $\mu_{\gc_{\mbox{-}}}(\lambda_{\psi,k}\cdot (\lambda_{\psi,k}^{\tau})^{-1})$) of the Katz $p$-adic measure.  Note that the measure used in \cite{HT} differs from $\mu_\gc$ by the involution $g\mapsto g^{-1}$ of $\cC\ell_K(p^{\infty}\gc)$. Note also that by definition $L^{-}_{\psi}$ and $H_{\psi}$ are elements of the embedding of
$\Lambda_{\cO}$ in $ \Lambda_{\cO}\widehat{\otimes} \Lambda_{\cO}$ via the first coordinate.

Let  $\Psi \in \Lambda_{\cO}\widehat{\otimes} \Lambda_{\cO}$ be the Euler product  defined in \cite[p.249 bottom]{HT}. 
By comparing their respective  specializations at each  $k > l \geqslant 2 $, Hida and Tilouine  showed in \cite[Thm.~8.1]{HT}  that 
the two elements $H_{\psi} \cdot \mathscr{D}$ and $\frac{\Psi L_1 L_2 H_{\psi}}{L^{-}_{\psi}}$ differ by a unit, {\it i.e.} their quotient belongs to 
$(\Lambda_{\cO}\widehat{\otimes} \Lambda_{\cO})[1/p]^{\times}$. The 
divisibility $L^{-}_{\psi} \mid H_{\psi}$ in $(\Lambda_{\cO}\widehat{\otimes} \Lambda_{\cO})[1/p]$ (hence in $\Lambda_{\cO}[1/p]$)
follows from \cite[Thm.~8.2]{HT} where it is shown that  $L_1,L_2$ and $\Psi$ are all relatively prime to $L^{-}_{\psi}$. 

The ideal $C^0_\psi$, resp. $(\zeta^{-}_{\psim})$, of $\varLambda_\gp$ is the localization of $(H_{\psi})$, resp. $L^{-}_{\psi}$ at the height one prime ideal $(X)$ of $\Lambda_{\cO}$ corresponding to  $f$, hence the theorem.  

Alternatively, as we are only interested at establishing this divisibility in $\varLambda_\gp$ ({\it i.e.} after localizing at $(X)$), we can instead argue 
as follows. Choosing $\phi$ such that  $\phi(\bar\gp) \ne 1 \ne \phi(\gp)$, the formulas \eqref{improved} and \eqref{kronecker}
imply the non-vanishing of $L_1$ and of $L_2$ when evaluated at $k=l=1$, and the same is true for $\Psi$ by the explicit formula defining it. 
\end{proof}

The description of the congruence ideal in Theorem~\ref{C0-thm} has the following consequence. 

\begin{cor}  \label{simple-zero}  Assume that  $\cLm(\varphi) \ne 0$. Then $C^0_\psi=(\zeta^{-}_{\psim})=(X)$, 
in particular, the order of vanishing of the anti-cyclotomic Katz $p$-adic $L$-functions $L_p^{-}(\varphi, s)$ at $s=0$ is $1$.
\end{cor}

This relates a conjecture of Hida on the adjoint $p$-adic  $L$-function of the CM  family $\Theta_{\psi}$ (see \cite[p.192]{HT}) to the well-established Four Exponentials Conjecture in Transcendence Theory, and shows that it holds for more than half of the characters, {\it i.e.,} for at least one character  in  each couple  $\{\psim,\psim^\tau\}$.

Combining Corollary~\ref{simple-zero} with   \eqref{derivativegpram} leads us to propose the following formula
\begin{align}\label{linear-term}
L_p(\varphi, s_{\gp}, s_{\bar\gp})\overset{?}{=} \big(\cL_{\gp}\cdot s_{\gp}+(\cL(\varphi)-\cL_{\gp}) \cdot s_{\bar\gp} \big)
 \cdot  L_p^*(\varphi,0)+\text {higher order terms}. 
\end{align}

or equivalently, using Katz  cyclotomic $p$-adic $L$-function $L_p(\varphi, s)= L_p(\varphi, s, s)$
\begin{align}\label{cyc-term}
L_p(\varphi, s)\overset{?}{=} \frac{(1-p^{-1})}{\ord_{v_0}(\tau(u_{\bar\gp, \varphi}))}\left|
\begin{matrix} \log_p(u_{\bar\gp, \varphi}) & \log_p(\mathfrak{u}_{\varphi})\\
 \log_p(\tau(u_{\bar\gp, \varphi})) & \log_p(\tau(\mathfrak{u}_{\varphi}))\end{matrix} \right|\cdot s  + O(s^2). 
 \end{align}
 
\begin{rem}
In \cite{benois}  Benois has formulated a  precise conjecture on trivial zeros in a generality that  
covers the case of rank $2$ Artin motives which are critical  in the sense of Deligne. 
His definition of a cyclotomic $\cL$-invariant depends on a choice of a regular sub-module which, in the case of a 
Galois representation which is locally scalar at $p$, can be arbitrarily chosen. We expect however that there should exist a choice for which 
 \eqref{cyc-term}  concords with his conjecture. 
  We should also mention a recent work \cite{kazim-sakamoto} of B\"uy\"ukboduk and  Sakamoto on the leading term, 
  under the additional hypotheses that $p$ is relatively prime  both to the class number of $K$ and to the order of $\psim$, 
  although their  methods are totally  different from ours. 
\end{rem}

For the rest of this section we assume that $\varphi=\psim$ is quadratic and we will check the compatibility between  \eqref{cyc-term} 
and  the Greenberg--Ferrero formula \cite{ferrero-greenberg}. Denote by $F=\Q(\sqrt{d})$ the real quadratic subfield of $H$ and by $K'\neq K$ its other imaginary  quadratic subfield. Recall Gross's factorization of the  Katz cyclotomic $p$-adic $L$-function as product of two 
 Kubota-Leopoldt $p$-adic $L$-functions (see \cite{gross-factor}, and also \cite{greenberg} for an arbitrary conductor)
  \begin{align}\label{Grossfactorization} 
 L_p(\varphi,s)=L_p( \epsilon_{K'}\cdot \omega_p  ,s) L_p(\epsilon_F,1-s), \end{align}
where $\omega_p$ is the Teichm\"{u}ller character. Then $L_p( \epsilon_{K'}\cdot \omega_p  ,s) $ has a trivial zero at $s=0$ 
and 
\[L'_p( \epsilon_{K'}\cdot \omega_p  ,0) = -\cL(\epsilon_{K'})L(\epsilon_{K'},0),\]
where $\cL(\epsilon_{K'})$ is the cyclotomic $\cL$-invariant of the odd Dirichlet character $\epsilon_{K'}$ (see \cite[(15)]{BDPozzi}), which turns out to be equal to $\cL(\varphi)$. By the Class Number Formula $L(\epsilon_{K'},0)=\tfrac{h_{K'}}{|\cO_{K'}^{\times}{}_{/\{\pm 1\}}|}$.

On the other hand, Leopoldt's formula yields that 
\[L_p(\epsilon_F,1)=-(1-p^{-1}) \sum\limits_{a=1}^d  \epsilon_F(a) \log_p(1-\zeta_d^a).\]
 As  $\sum_{a=1}^d  \epsilon_F(a) \log_p(1-\zeta_d^a)\in \cO_F^\times[\epsilon_F]\subset \cO_H^\times[\varphi]$,
 it follows that  \eqref{cyc-term} holds in this case, up to a multiplication by an element of $\Q^\times$.

\section{An $R=T$ modularity lifting theorem for non-Gorenstein local rings}\label{trianglesec}

The goal of this section is to introduce  a universal ring $\cR^{\triangle}$ representing  deformations of $\rho_f$ which are {\it generically ordinary}, and   show that its nilreduction $\cR^{\triangle}_{\red}$ is isomorphic to $\cT$.

\subsection{Ordinary framed deformations of a $2$-dimensional trivial representation}\label{defntriangle}
The deformation functor  $\cD^{\Box}_{\loc}$ associating to $A$ in $\cC$  the set of framed deformations $\rho_A:\Gamma_{\Q_p} \to \GL_2(A)$ of $\mathbf{1}_2$ is representable by  $\cR_{\loc}^{\Box}$ together with  $\rho^{\Box}:\Gamma_{\Q_p} \to \GL_2(\cR_{\loc}^{\Box})$. 

Recall that $\rho_A\in \cD^{\Box}_{\loc}(A)$ is ordinary if and only if there exists a $\Gamma_{\Q_p}$-stable  direct factor $\Fil_A  \subset A^2$ of rank $1$ over $A$ and
such that $\Gamma_{\Q_p}$ acts  on $A^2/\Fil_A$ by an unramified character $\chi_A$.
Being ordinary is not a deformation condition on $\cD^{\Box}_{\loc}$ as it does not satisfy  Schlessinger's Criterion (see for example \cite[\S3.1]{geraghty}).  

Recall that  for any $\overline{\Q}_p$-scheme $S$, an  $S$-point of  $\mathbb{P}^{1}$ is an invertible $\cO_{S}$-sub-module
  $ \Fil_S \subset \cO_{S}^{2}$ which is  locally a direct summand. We introduce a variant, when the residual characteristic is $0$, 
   of a functor defined in \cite[\S3.1]{geraghty}.

\begin{prop} \label{G-prop}
Let $\cD^{\triangle}_{\mathrm{loc}}$ be the subfunctor of $\mathbb{P}^{1}\times_{\overline{\Q}_p} \cR^{\Box}_{\loc}\lsem U\rsem$ whose $A$-points
correspond to a filtration $\Fil_A \in \mathbb{P}^{1}(A)$  and a morphism $\pi_A: \cR^{\Box}_{\loc}\lsem U\rsem \to  A$,  such that 
  $\Fil_A $ is  $\Gamma_{\Q_p}$-stable for the action on $A^2$ defined via  $\cR^{\Box}_{\loc} \to \cR^{\Box}_{\loc}\lsem U\rsem\overset{\pi_A}{\longrightarrow} A$ and moreover $A^2/ \Fil_A$ is unramified with $\Frob_p$ acting by $\phi_A=\pi_A(1+U)\in A^\times$. 

The functor $\cD^{\triangle}_{\mathrm{loc}}$ is representable by a closed subscheme $\cX \subset \mathbb{P}^{1} \times_{\overline{\Q}_p} \cR^{\Box}_{\loc}\lsem U\rsem$. \end{prop}
\begin{proof}\
See \cite[Lem.~3.1.1]{geraghty}.\end{proof}

Let $\Spec \cR^{\triangle}_{\loc} \subset \Spec \cR^{\Box}_{\loc}\lsem U\rsem$ be the scheme theoretical image of $\cX$ under the second projection $\pr_2:\mathbb{P}^{1} \times_{\overline{\Q}_p} \cR^{\Box}_{\loc}\lsem U\rsem \to \Spec \cR^{\Box}_{\loc}\lsem U\rsem$, that is to say $\Spec \cR^{\triangle}_{\loc}$ is the Zariski closure of $\pr_2(\cX)$. As $\mathbb{P}^1$ is proper (hence universally closed) over $\overline{\Q}_p$ and as $\cX$ is a closed subscheme of $\mathbb{P}^{1} \times_{\overline{\Q}_p} \cR^{\Box}_{\loc}\lsem U\rsem$,  the natural morphism $\pr_2:\cX \to \Spec \cR^{\triangle}_{\loc}$ is surjective.
\begin{align}\label{carre-X} 
\xymatrix{
 \cX \ar@{->>}^{\pr_2}[d]\ar@{^{(}->}[r] &\mathbb{P}^{1} \times_{\overline{\Q}_p}  \cR^{\Box}_{\loc}\lsem U\rsem \ar@{->>}^{\pr_2}[d] \\
\Spec(\cR^{\triangle}_{\loc}) \ar@{^{(}->}[r] & \Spec(\cR^{\Box}_{\loc}\lsem U\rsem).}
\end{align}
We denote by  $\phi\in (\cR^{\triangle}_{\loc})^\times$ the image of $(1+U)$ under the canonical surjection 
$\cR^{\Box}_{\loc}\lsem U\rsem\twoheadrightarrow \cR^{\triangle}_{\loc}$.

\subsection{Generically ordinary deformations of \texorpdfstring{$\rho_f$}{}}\label{Rtriangleconstru}
Let $(e_1,e_2)$ be a basis of $\overline{\Q}{}_p^2$ in which $\rho_{f\mid \Gamma_K}=\left(
\begin{smallmatrix} \psi & 0 \\ 0 & \psi^\tau \end{smallmatrix}\right)$.    
Let $\rho^{\univ} : \Gamma_{\Q} \to \GL_2(\cR_{\rho_f}^{\univ}) $  be the universal deformation  of $\rho_f$.

Let $\cR^{\triangle}=\cR^{\triangle}_{\loc} \widehat{\otimes}_{\cR^{\Box}_{\loc}} \cR_{\rho_f}^{\univ}$ and $\rho^{\triangle}_\cR:\Gamma_{\Q} \to \GL_2(\cR^{\triangle})$ be the corresponding  universal deformation, where the natural morphism $\cR^{\Box}_{\loc}\to \cR_{\rho_f}^{\univ}$ comes from $ \psi^{-1} \otimes\rho^{\univ}_{|\Gamma_{\Q_p}}$.
We next show  that the schematic points of $\Spec \cR^{\triangle}$ correspond to ordinary representations.

\begin{lemma}\label{genericordinary} Let  $x:\cR^{\triangle} \to L$ be a point over a field $L$. Then the corresponding  representation $\rho_x:\Gamma_{\Q} \to \GL_2(L)$ has an unramified rank $1$ $\Gamma_{\Q_p}$-quotient on which $\Frob_p$ acts by $x(\phi\otimes 1)\in L^\times$.
\end{lemma}
\begin{proof} By definition of  $\cR^{\triangle}$, $x$ yields a point $y:\cR^{\triangle}_{\loc}\to L$ such that 
$y(\phi)=x(\phi\otimes 1)\in L^\times$.
As $\pr_2:\cX \to \Spec \cR^{\triangle}_{\loc}$ is surjective,  we can lift $y$ to  $y'\in \cX(L')$, with $L'$  a field containing $L$. By Proposition~\ref{G-prop}, $(\rho_x \otimes_L L')_{\mid \Gamma_{\Q_p}}$ stabilizes a line $\Fil_{L'}\subset (L')^2$ and acts on the quotient $(L')^2/\Fil_{L'}$ by an unramified character sending  $\Frob_p$ to  $\pr_2(y')(\phi)\in (L')^\times$. 

The commutativity of the diagram \eqref{carre-X} implies that    $\pr_2(y')(\phi)= y(\phi)\in L$, and therefore 
$\Fil_{L'}$ has an  $L$-rational basis, and $\Frob_p$ acts on the quotient by  $x(\phi\otimes 1)$ as requested. 
 \end{proof}

\subsection{Modularity in the non-Gorenstein case}
Recall that $\cT$ denotes the completed local $\varLambda$-algebra of $\cE$ at $f$. 
The restriction of the deformation $\rho_{\cT}:\Gamma_{\Q} \to \GL_2(\cT)$ to $\Gamma_{\Q_p}$ yields a  morphism $\cR^{\Box}_{\loc}\lsem U\rsem \to \cT$ sending $\psi(\gp)(1+U)$ on $U_p$. Moreover, $\rho_{\cT} \otimes Q(\cT)$ is  ordinary at $p$, yielding
a $Q(\cT)$-point of $\cX$, hence,  in view of  \eqref{carre-X},  we have a following commutative diagram 
\[\xymatrix{ \cR^{\Box}_{\loc}\lsem U\rsem  \ar[d]\ar@{->>}[r] & \cR^{\triangle}_{\loc}  \ar[d] \ar@{-->}[ld]\\
\cT  \ar@{^{(}->}[r] &Q(\cT)  } \]
Since $\cT$ is reduced,  the above diagram can be commutatively completed with a 
morphism $\cR^{\triangle}_{\loc} \to \cT$. By~\eqref{eq:rhoT} there exists a homomorphism 
$\cR_{\rho_f}^{\univ}\to \cT$. As $\cT$ is topologically generated over $\varLambda$ by the image of
$\tr(\rho^{\univ})$ and by $U_p$,  we obtain a natural surjective $\varLambda$-algebra homomorphism 
\begin{align}\label{Rtri=T} \cR^{\triangle} \twoheadrightarrow \cT. \end{align}

\begin{thm} 
The morphism  \eqref{Rtri=T} induces an isomorphism $\cR^{\triangle}_{\red} \xrightarrow{\sim} \cT$, where $\cR^{\triangle}_{\red}$ is the nilreduction of $\cR^{\triangle}$. In particular $\cR^{\triangle}$ is equidimensional of dimension $1$. 
\end{thm}

\begin{proof}  
It suffices to show that for every minimal prime ideal $\mathfrak{q}$ of $\cR^{\triangle}$ there exists a surjective homomorphism
$ \cT  \twoheadrightarrow \cR^{\triangle}/\mathfrak{q}$  of $\cR^{\triangle}$-algebras. In fact, it will then follow that the kernel of  $\cR^{\triangle} \twoheadrightarrow \cT$ is contained in the nilradical of $\cR^{\triangle}$ (which is given by the intersection of its primes). As $\cT$ is reduced, the desired isomorphism $\cR^{\triangle}_{\red} \xrightarrow{\sim} \cT$ would follow.

To prove the above claim, we let 
$\cA=\cR^{\triangle}/\mathfrak{q}\neq \overline{\Q}_p$ and consider the push-forward 
 $\rho_{\cA}:\Gamma_{\Q} \to \GL_2(\cA)$ of $\rho_\cR^{\triangle}$ along $\cR^{\triangle} \twoheadrightarrow \cA$. As $Q(\cA)$ is a field,  Lemma~\ref{genericordinary} implies that  $\rho_{\cA} \otimes Q(\cA)$ has an unramified free rank $1$ $\Gamma_{\Q_p}$-quotient on which $\Frob_p$ acts by the image of $\phi\otimes 1$ in $\cA^\times$. Note that $\rho_{\cA}$ need not be 
 $p$-ordinary {\it a priori} ({\it a posteriori} this will be the case when   $\cA$ is a discrete valuation ring). 
 One  distinguishes two cases 
\begin{itemize}[wide]
\item  $\rho_{\cA\mid \Gamma_K}$ is reducible. Exactly as in the proof of   Proposition~\ref{Tcmgor}
one shows that  $\rho_{\cA} \simeq \Ind^\Q_K \psi_{\cA}$, where $\psi_{\cA}:\Gamma_K \to \cA^{\times}$ is a character lifting $\psi$. The  $\gp$-ordinariness of $\rho_{\cA} \otimes Q(\cA)$   implies that  $\psi^{-1}\psi_{\cA}$ is equal either to  $\chi_{\gp}$ or to $ \chi_{\gp}^{\tau}$.
In both cases $\cT \twoheadrightarrow  \cA=\varLambda$. 

\item $\rho_{\cA\mid \Gamma_K}$ is irreducible.  As in Proposition~\ref{rhott}, one may construct a point of $\cD^{\perp}(\cA)$, {\it i.e.} a homomorphism of integral domains $\cR^{\perp}\to\cA$, which is necessarily injective as $\cR^{\perp}$ is a discrete valuation ring. 
It follows from Theorem~\ref{geomCdag} that either $\cA\simeq \cR^{\perp}$ when $r$ is even, 
or  $\cA\simeq\widetilde{\cR}^{\perp} $ when $r$ is odd, hence in all cases $\cT \twoheadrightarrow  \cT^{\perp} \twoheadrightarrow   \cA$. \qedhere
\end{itemize}
\end{proof}

The following lemma in commutative algebra justifies the title of the section. 

\begin{prop}\label{commalg} 
The dimension of the $\overline{\Q}_p$-vector space  $\cT/(\gm_\varLambda,\gm_{\cT}^2)$ is  $4$. \\
The eigencurve $\cE$ is not Gorenstein at $f$.
\end{prop}

\begin{proof} The $1$-dimensional reduced local ring $\cT$ is Gorenstein 
if and only if  its quotient $\cT/a \cT$ by a regular element $a \in \cT$ is Gorenstein. 
As $\cT/a \cT$ is Artinian, the claim  is equivalent to showing that its socle  $\cM=\Hom_{\cT}(\overline{\Q}_p,\cT/a \cT)$ is $1$-dimensional. 

Assume first that $\cLm(\psim)\cdot\cLm(\psim^{\tau}) \ne 0$ and  let $a=(X,Y,Y,X)$, where $Y$ denotes a uniformizer of  $\cR^{\perp}$. 
By Theorem~\ref{C0-thm} one has  $\cT = \varLambda \times_{\overline{\Q}_p} \cR^{\perp} \times_{\overline{\Q}_p} \cR^{\perp} \times_{\overline{\Q}_p} \varLambda $ hence $\gm_{\cT}^2=a \cdot \gm_{\cT}$  and $\cM=\gm_{\cT}/a\cT$. 
It follows that $\dim(\cM)=\dim(\gm_{\cT}/(a,\gm_{\cT}^2))=\dim(\gm_{\cT}/\gm_{\cT}^2)-1=
\dim(\gm_{\cT}/(\gm_\varLambda,\gm_{\cT}^2))=3>1$, hence $\cT$ is not Gorenstein.

Assume next  that $\cLm(\psim)\cdot\cLm(\psim^{\tau}) = 0$ and we let $a=(X,X,X)$. Then $ \cT \simeq   \varLambda  \times_{\overline{\Q}_p} \cA$, with
\[\cA=  \varLambda[Z]/(Z^2-X^r)\times_{\overline{\Q}_p[X]/(X^{r-1})} \varLambda.\]
As the  $\overline{\Q}_p$-vector space $\gm_{\cA}/\gm_{\cA}^2$  has basis $\{(X,X),(Z,0),(0,X^{r-1})\}$, one computes  that 
$\gm_{\cA}^2=\gm_\varLambda\gm_{\cA}$ and  $\gm_{\cT}^2=\gm_\varLambda \gm_{\cT}$. 
Again $\cM=\gm_{\cT}/\gm_\varLambda\cT$ 
has dimension $3$, hence $\cT$ is not Gorenstein. 

In both cases $\gm_{\cT}/(\gm_\varLambda,\gm_{\cT}^2)$ is $3$-dimensional, hence 
$\cT/(\gm_\varLambda,\gm_{\cT}^2)$ is  $4$-dimensional. 
\end{proof}

We close this paper by establishing Conjecture 4.1 from  \cite{DLR4} about  the generalized eigenspace $S_{1}^{\dagger}(N)\lsem f\rsem= S_{1}^{\dagger}(N)\lsem \gm_{\cT}\rsem$ attached to $f$ inside the space of weight $1$, level $N$, ordinary $p$-adic modular forms. To be more precise, Darmon, Lauder and Rotger only consider the subspace $S_{1}^{\dagger}(N)[\gm_{\cT}^2]$.
The subspace of classical forms $S_{1}(Np)[\gm_{\cT}^2]$ is two-dimensional having $\{f,\theta_\psi\}$ as basis, and a supplement  
of  this space   is given by the space $S_{1}^{\dagger}(N)[\gm_{\cT}^2]_0$ of normalized generalized  eigenforms, {\it i.e.} forms
 whose  first and $p$-th Fourier coefficients  both vanish.

\begin{cor} \label{DLRconj} The $\overline{\Q}_p$-vector space $S_{1}^{\dagger}(N)[\gm_{\cT}^2]$ has  dimensions $4$. \\
Moreover Conjecture 4.1 from  \cite{DLR4} holds in the CM case, {\it i.e.}  
the two-dimensional vector space $S_{1}^{\dagger}(N)[\gm_{\cT}^2]_0$ is canonically 
isomorphic to $\rH^1(\Q, \ad^0(\rho_f))$. 
\end{cor}

\begin{proof} The first claim is merely a translation of the first claim of Proposition~\ref{commalg}, in view of the duality 
\eqref{hidadualitywt1}. It follows that 
$\dim S_{1}^{\dagger}(N)[\gm_{\cT}^2]_0=4-2=2$ and the conjecture then follows from 
\cite[Lem.~3.2]{bellaiche-dimitrov} asserting that $\dim\rH^1(\Q, \ad^0(\rho_f))=2$ as well. 
\end{proof}

  { \noindent{\it Acknolwedgements:} { \small
The authors are mostly indebted to H.~Hida and J.~Tilouine for numerous discussions related to this project. 
In addition, the second author would like to thank M.~Waldschmidt for an extremely helpful correspondence in Transcendence Theory  and D.~Benois for explaining his conjectures on trivial zeros. 
The first author acknowledges  support from  the EPSRC (grant EP/R006563/1)  and the START-Prize Y966 of the Austrian Science Fund (FWF). The second author is partially supported by  
the Agence Nationale de la Recherche (grants  ANR-18-CE40-0029 and ANR-16-IDEX-0004). Finally, both authors thank the anonymous referees for their carefully reading of the manuscript. } 

\newpage
\section*{Glossary of notation}

\footnotesize

\begin{multicols}{2}
\noindent 
$\cC\ell_K $ \dotfill   ideal class group of $K$\\
$C^0_\psi$  \dotfill  congruence ideal of  $\Theta_{\psi}$, \S\ref{s:cong-ideal} \\
$C^0_\cF$ \dotfill  congruence ideal of  $\cF$, \S\ref{s:cong-ideal}\\ 
$C$ \dotfill Galois group  of $H/K$, \S\ref{L-invariants} \\
$-D$ \dotfill fundamental discriminant of $K$\\
 $e$ \dotfill ramification index of $\varLambda\to \cR^{\perp}$, Def. \ref{d:ideal-red} \\
 $\cE$ \dotfill  eigencurve of tame level $N$, \S\ref{section3} \\
 $\cE^{\perp}$ \dotfill  union of non-CM (by K) irred. components, \S\ref{nonCMsection}\\
$f$ \dotfill the unique $p$-stabilisation of $\theta_\psi$\\
$\Frob_{\gp}$ \dotfill  arithmetic Frobenius in $\Gamma_{K_\gp}$\\
$\cF$ \dotfill eigenfamily without CM by $K$ containing $f$, \S\ref{wihtoutCMgenericpoints}\\
$G$ \dotfill Galois group  of $H/\Q$, \S\ref{L-invariants} \\
$H$ \dotfill splitting field of $\psim$, \S\ref{subs-slope}\\
$\I_{\gp}$ \dotfill inertia subgroup of $\Gamma_{K_\gp}$\\
$K$ \dotfill imaginary quadratic field \\
$K_{\infty}^-$ \dotfill  anti-cyclotomic $\Z_p$-extension of $K$, \S\ref{CM-Ideal} \\
$L_p^{-}$ \dotfill  anti-cyclotomic Katz $p$-adic $L$-function, \S\ref{Katz-padic}\\
$\cL_{\gp}$ \dotfill $\cL$-invariant  \eqref{defn-L1}\\
$\cL(\psim)$ \dotfill  cyclotomic $\cL$-invariant \eqref{defn-L-inv}\\
$\cLm(\psim)$ \dotfill   anti-cyclotomic $\cL$-invariant, Def.~\ref{d:new-L-inv}\\
 $\gp$ \dotfill  place of $K$ above $p$ induced by $\iota_p$\\
 $r$ \dotfill valuation of the reducibility ideal of $\cR^\perp$ \\ 
 $\cR_{\loc}^{\Box}$ \dotfill framed deformation ring, \S\ref{defntriangle}\\
 $\cR^{\univ}_{\rho}$ \dotfill   deformation ring with  $\tau$-invariant trace, \S\ref{ord-def}\\
 $\cR^{\ord}_{\rho}$ \dotfill   $\gp$-ordinary deformation ring  of $\rho$, Def.~\ref{cond1}\\
 $\cR^{\ord}_{\rho'}$ \dotfill   $\gp$-ordinary deformation ring of $\rho'$, \S\ref{daggerfunct}\\
 $\cR^{\red}_{\rho}$ \dotfill   reducible $\gp$-ordinary deformation ring, Def.~\ref{defn:red}\\
 $\cR^{K}_{\rho}$ \dotfill    CM $\gp$-ordinary deformation ring, Def.~\ref{defn:red}\\
 $\cR^{\perp}$ \dotfill   non-CM $\gp$-ordinary deformation ring, Def.~\ref{cond2}\\
 $\widetilde{\cR}^{\perp}$ \dotfill   $\cR^{\perp}[Z]/(Z^2-Y^r)$, \S\ref{reducibility-ideal}\\
 $\cR^{\triangle}$ \dotfill  generically ordinary deformation ring  of $\rho_f$, \S\ref{defntriangle}\\
 $\mathscr{S}(\psim)$ \dotfill slope of $\psim$,  \eqref{eq:slope-def}\\
 $S_{1}^{\dagger}(N)\lsem f\rsem$ \dotfill generalised eigenspace of $f$, \S \ref{nonCMsection} \\
 $\cT$  \dotfill  completed local ring    of $\cE$  at $f$, \S\ref{nonCMsection}  \\
 $\cT^{\perp}$ \dotfill  completed local ring    of $\cE^{\perp}$  at $f$, \S\ref{nonCMsection} \\
 $\cT_{\spl}\subset \cT$ \dotfill  generated  by  $T_\ell$ for
 $\ell\nmid Np$  split in $K$, \S\ref{nonCMsection}\\
 $\cT^\perp_{\spl}$ \dotfill  image of the $\varLambda$-algebra $\cT_{\spl}$ in $\cT^\perp$, \S\ref{nonCMsection}\\
 $u_{\psim}$ \dotfill $\psim$--isotypic unit in $\cO_H^\times \otimes \bar\Q$, \S\ref{subs-slope}\\
$u_{\gp,\psim}$ \dotfill $\psim$-isotypic  $\gp$-unit in $\cO_H[1/\gp]^\times \otimes \bar\Q$ \S\ref{def-eta}\\
$u_{\gp}$  \dotfill $\gp$-unit of $K$\\
 $v_0$ \dotfill  place of $H$ above $\gp$ induced by $\iota_p$ \\
  $\cW$ \dotfill the weight space, \S\ref{section3}\\
  $\mathbb{W}_\gp$ \dotfill Galois group of $\gp$-ramified $\Z_p$-extension, \S\ref{cm-chars}\\
 $X$ \dotfill uniformizer of $\varLambda$, $\varLambda_\gp$ \S\ref{cm-chars}\\
 $Y$ \dotfill uniformizer of $\cR^{\perp}$,  \S\ref{section3}\\
 $\gamma_0$ \dotfill fixed element of $\Gamma_K$,  \S\ref{deformation}\\
$\Gamma_L$ \dotfill  absolute Galois group of a perfect field $L$\\ 
$\Gamma_{K/\Q}$ \dotfill Galois group  of $K/\Q$ \\
$\Gamma_{K_\gp}$ \dotfill decomposition subgroup of $\Gamma_{K}$\\
$\zeta^{-}_{\psim}$ \dotfill $p$-adic inverse Mellin transform of $L_p^{-}(\psim,\cdot)$ \S\ref{Katz-padic}\\
$\eta_{\gp}, \eta_{\bar\gp}$ \dotfill basis of $\rH^1(K, \overline{\Q}_p)$   \S\ref{def-eta}\\
$[\eta], \eta$ \dotfill canonical basis  of $\rH^1(K, \psim)$, Prop.~\ref{restinj}, \S\ref{deformation}\\
$[\eta'], \eta'$ \dotfill canonical basis  of $\rH^1(K, \psim^\tau)$, \S\ref{daggerfunct}\\
$\theta_\psi$ \dotfill weight one theta-series attached to $\psi$\\
$\Theta_\psi,\Theta_{\psi^\tau}$ \dotfill  theta families specializing to $f$, \S\ref{cm-chars}\\
$\iota_p$ \dotfill  embedding of $\bar\Q$ in  $\bar\Q_p$\\
$\kappa$ \dotfill  weight map, \S\ref{section3}\\ 
  $\varLambda$ \dotfill  completed local ring of $\cW$ at $\kappa(f)$, \S\ref{section3}\\    
  $\varLambda_\gp$ \dotfill  completed local ring of  $\Theta_\psi$-component  at $f$, \S\ref{section3}\\
  $\pi_\psi$ \dotfill $\Theta_\psi$-projection, \S\ref{s:cong-ideal}\\
  $\rho_f=\Ind^\Q_K \psi$ \dotfill Artin representation associated to $f$ \\
 $\rho_{\cR}$ \dotfill  ordinary deformation of $\rho$ to $\cR^{\ord}_\rho$, \S\ref{ord-def}\\
 $\rho_{\cR^\perp}$  \dotfill   $\gp$-ordinary deformations of $\rho$ to $\cR^{\perp}$, \S\ref{daggerfunct}\\
 $\rho'_{\cR^\perp}$  \dotfill   $\gp$-ordinary deformations of $\rho'$ to $\cR^{\perp}$, \S\ref{daggerfunct}\\
 $ \widetilde\rho^\perp_{\cR}$ \dotfill  extension of  $\rho_{\cR^\perp}$ to $\Gamma_{\Q}$, Prop.~\ref{extensionpseudo}\\ 
 $\rho_{\cT^\perp}$ \dotfill  deformation of $\rho$ to $\cT^{\perp}$, \S\ref{nonCMsection}\\
 $\rho'_{\cT^\perp}$ \dotfill  deformation of $\rho'$ to $\cT^{\perp}$, \S\ref{nonCMsection}\\
  $\rho_{\cT}$  \dotfill   deformation of $\rho_f$ to $\cT$, \S\ref{nonCMsection}\\
 $\rho_{\cT}^{\perp}$ \dotfill  deformation of $\rho_f$ to $\cT^{\perp}$, \S\ref{nonCMsection}\\
 $\tau$ \dotfill complex conjugation, \S\ref{L-invariants}\\
 $\epsilon_K$ \dotfill Dirichlet character of $K/\Q$, \S\ref{L-invariants} \\
 $\epsilon_\gp$ \dotfill $\gp$-ramified character of type $(-1,0)$, \S\ref{cm-chars} \\
 $\chi_{\cR}$ \dotfill  ordinary character, \S\ref{daggerfunct}\\ 
 $\chi_p$ \dotfill universal cyclotomic character, \S\ref{cm-chars}  \\
 $\chi_\gp$ \dotfill universal  $\gp$-ramified character, \S\ref{cm-chars}\\
 $\chim$ \dotfill universal anti-cyclotomic character, \S\ref{CM-Ideal}\\
  $\psi, \psi^\tau, \psim$ \dotfill finite order Hecke characters of $K$, \S\ref{L-invariants}\\
$\varphi$ \dotfill anti-cyclotomic Hecke character of $K$, \S\ref{Katz-padic} 
 \end{multicols}

\normalsize
\newpage

\bibliographystyle{siam}

\end{document}